\documentclass[11pt]{amsart}
\usepackage[letterpaper,hmargin=1in,vmargin=1in]{geometry}

\usepackage{wrapfig}
\usepackage{comment}
\usepackage{xcolor}
\usepackage{mathbbol}
\usepackage{thmtools}
\usepackage{thm-restate}
\usepackage{mathtools}
\usepackage{amsmath}
\usepackage{amsthm}
\usepackage{amssymb}
\usepackage{tikz}
\usepackage{tikz-cd}
\usetikzlibrary{matrix,arrows,cd,calc}
\usepackage[hidelinks]{hyperref}
\usepackage{xcolor}
\definecolor{darkgreen}{rgb}{0,0.45,0}
\hypersetup{colorlinks,urlcolor=blue,citecolor=darkgreen,linkcolor=darkgreen,linktocpage}
\usepackage[capitalize]{cleveref}
\crefname{algocf}{Algorithm}{Algorithms}
\usepackage{xspace}
\usepackage[margin=0.5cm]{caption}
\usepackage{subcaption}
\usepackage{enumerate}
\usepackage{centernot}
\usepackage{textcomp}
\usepackage[noend,ruled,linesnumbered]{algorithm2e}
\usepackage{algpseudocode}
\usepackage[shortlabels]{enumitem}
\setlist{leftmargin=7mm}
\usepackage{stmaryrd}
\usepackage{xfakebold}
\usepackage{float}

\newtheorem{theorem}{Theorem}[section]
\newtheorem{lemma}[theorem]{Lemma}

\newtheorem{proposition}[theorem]{Proposition}
\newtheorem{corollary}[theorem]{Corollary}

\crefname{theoremx}{Theorem}{Theorems}

\crefname{corollaryx}{Corollary}{Corollaries}
\newtheorem{questionx}{Question}

\theoremstyle{definition}
\newtheorem{definitionx}{Definition}
\crefname{definitionx}{Definition}{Definitions}

\newtheorem{definition}[theorem]{Definition}
\newtheorem*{definition*}{Definition}

\newtheorem{construction}[theorem]{Construction}
\crefname{construction}{Construction}{Constructions}
\newtheorem{remark}[theorem]{Remark}
\newtheorem*{remark*}{Remark}
\newtheorem{notation}[theorem]{Notation}
\newtheorem*{notation*}{Notation}

\newcommand{\luis}[1]{\noindent{\color{blue}[Luis: #1]}}

\newcommand{\todo}[1]{\noindent{\color{red}[todo: #1]}}

\renewcommand{\to}{\xrightarrow{\;\;\;}}

\renewcommand{\epsilon}{\varepsilon}
\renewcommand{\phi}{\varphi}

\let\ker\undefined
\DeclareMathOperator{\ker}{ker}
\DeclareMathOperator{\coker}{coker}
\DeclareMathOperator{\Hom}{Hom}

\DeclareMathOperator{\End}{End}
\DeclareMathOperator{\add}{add}

\DeclareMathOperator{\proj}{proj}
\newcommand{\indsimpl}{\mathsf{ind.simpl}}
\newcommand{\indproj}{\mathsf{ind.proj}}

\newcommand{\dimhom}{\mathrm{dhom}}
\renewcommand{\lim}{\mathrm{lim}}
\newcommand{\colim}{\mathrm{colim}}

\newcommand{\im}{\mathrm{im}}
\newcommand{\id}{\mathrm{Id}}
\newcommand{\idx}{\mathrm{gr}}
\newcommand{\rk}{\mathsf{rk}}
\newcommand{\rank}{\mathrm{rank}}

\DeclareMathOperator{\dec}{dec}

\newcommand{\births}{{\mathrm{Bth}}}
\newcommand{\deaths}{{\mathrm{Dth}}}

\newcommand{\kerelem}[2]{{#1}^{#2}}
\newcommand{\cokerelem}[2]{{#1}_{#2}}
\newcommand{\conv}{\mathrm{convex}}
\newcommand{\conc}{\mathrm{concave}}
\newcommand{\inconv}{\conv^\vee}
\newcommand{\inconc}{\conc^\vee}
\let\dim\relax
\DeclareMathOperator{\dim}{dim_\kbb}

\newcommand{\uprightcorner}{%
{\begin{tikzpicture}[scale=0.3]
   \draw[-] (1,1) -- (0,1);
   \draw[-] (1,0) -- (1,1);
\end{tikzpicture}}}
\newcommand{\downleftcorner}{%
{\begin{tikzpicture}[scale=0.3]
   \draw[-] (0,0) -- (1,0);
   \draw[-] (0,1) -- (0,0);
\end{tikzpicture}}}
\newcommand{\uprightcornerc}{%
{\begin{tikzpicture}[scale=0.3]
   \draw[-] (1,1) -- (0,1);
   \draw[-] (1.15,1.15) -- (0,1.15);
   \draw[-] (1,0) -- (1,1);
   \draw[-] (1.15,0) -- (1.15,1.15);
\end{tikzpicture}}}

\newcommand{\kerxyc}[1]{\,{#1}\text{\raisebox{1px}{\hspace*{-7px}{$\uprightcornerc$}}}\,}
\newcommand{\kerxc}[1]{\,{{#1}\|}\,}
\newcommand{\keryc}[1]{\,\overline{\overline{#1}}\,}
\newcommand{\kerxy}[1]{\,{#1}\text{\raisebox{1px}{\hspace*{-7px}{$\uprightcorner$}}}\,}
\newcommand{\cokerxy}[1]{\,\text{\raisebox{-2.5px}{$\downleftcorner$}}\hspace*{-7px}{#1}\,}
\newcommand{\topleft}[1]{\,\text{\raisebox{6px}{$\centerdot$}}{#1}\,}
\newcommand{\botright}[1]{\,{#1}\text{\raisebox{-1.5px}{$\centerdot$}}\,}
\newcommand{\kerx}[1]{\,{{#1}|}\,}
\newcommand{\cokerx}[1]{\,{|{#1}}\,}
\newcommand{\kery}[1]{\,\overline{#1}\,}
\newcommand{\cokery}[1]{\,\underline{#1}\,}

\newcommand{\kerxyspread}{\kerxy}
\newcommand{\cokerxyspread}{\cokerxy}

\newcommand{\Top}{\mathrm{Top}}
\newcommand{\vect}{\mathrm{vec}}
\newcommand{\Vect}{\mathrm{Vec}}

\renewcommand{\mod}{\mathrm{mod}}
\newcommand{\Mod}{\mathrm{Mod}}
\newcommand{\grmod}[1]{#1\text{-}\mathrm{gr.Mod}}

\newcommand{\rep}{\mathrm{rep}}
\newcommand{\Rep}{\mathrm{Rep}}
\newcommand{\repf}{\mathrm{rep}_0}

\newcommand{\fintwoparam}{\rep(\Gcal^2)}

\newcommand{\Rk}{\mathrm{Rk}}
\newcommand{\spreads}{\mathsf{spreads}}
\newcommand{\hooks}{\mathsf{hooks}}
\newcommand{\rects}{\mathsf{segments}}
\newcommand{\spreadcurves}{\mathsf{spread.curves}}
\newcommand{\spreadcurvesf}{\spreadcurves_0}
\newcommand{\chigpd}{\mathcal{N}^{\mathrm{gen.pers.diag.}}}

\newcommand{\chirkdec}{\mathcal{N}^{\mathrm{sign.barc.}}}
\newcommand{\chirkexact}{\mathcal{N}^{\mathrm{hook.dec.}}}
\newcommand{\chispreadexact}{\mathcal{N}^{\mathrm{int.Eul.char.}}}

\newcommand{\Ncaltwo}{\Ncal^2}
\newcommand{\Ncaldec}{\Ncal^\mathrm{dec}}
\newcommand{\Ninclexcl}{\Ncal^{\mathrm{incl.excl.}}}
\newcommand{\Ncalone}{\Ncal^1}
\newcommand{\sq}{\mathsf{sq}}

\newcommand{\boundaries}{\mathsf{Bnds}}

\newcommand{\op}{\mathsf{op}}

\DeclareMathAlphabet{\mathpzc}{OT1}{pzc}{m}{it}
\usepackage[mathscr]{euscript}
\makeatletter
\newcommand\DEFINEALPHABETLOOP[3]{%
  \ifx\relax#3\expandafter\@gobble\else\expandafter\@firstofone\fi
  {\expandafter\newcommand\expandafter*\csname#3#1\endcsname{#2{#3}}%
   \DEFINEALPHABETLOOP{#1}{#2}}%
}%
\newcommand\Definealphabet[2]{%
  \DEFINEALPHABETLOOP{#1}{#2}abcdefghijklmnopqrstuvwxyzABCDEFGHIJKLMNOPQRSTUVWXYZ\relax
}%
\makeatother
\Definealphabet{bb}{\mathbb}
\Definealphabet{cal}{\mathcal}
\Definealphabet{bf}{\mathbf}
\Definealphabet{sf}{\mathsf}
\Definealphabet{rm}{\mathrm}
\Definealphabet{frak}{\mathfrak}
\Definealphabet{scr}{\mathscr}
\author{Thomas Brüstle}
\address{Bishop's University; Université de Sherbrooke; Québec, Canada}
\author{Steve Oudot}
\address{Inria Saclay and \'Ecole polytechnique; Palaiseau, France}
\author{Luis Scoccola}
\address{Centre de Recherches Mathématiques et Institut des sciences mathématiques;
Laboratoire de combinatoire et d'informatique mathématique de l'Université du Québec à Montréal;
Université de Sherbrooke; Québec, Canada}
\author{Hugh Thomas}
\address{D\'epartement de Math\'ematiques, LACIM, Universit\'e du Qu\'ebec \`a Montr\'eal}

\title{Counts and end-curves in two-parameter persistence}

\begin{document}

%\tableofcontents
%
%\luis{test}
%\steve{test}
%\thomas{test}
%\hugh{test}
%\todo{test}
%
%{
%    \footnotesize
%To-do:
%\begin{itemize}
%    \item Improve notation for $\kerxy{M}$ and other boundaries.
%\end{itemize}
%}
%
%
%\newpage

\begin{abstract}
    Given a finite dimensional, bigraded module over the polynomial ring in two variables, we define its two-parameter count, a natural number, and its end-curves, a set of plane curves.
    These are two-dimensional analogues of the notions of bar-count and endpoints of singly-graded modules over the polynomial ring in one variable, from persistence theory.
    We show that our count is the unique one satisfying certain natural conditions; as a consequence, several inclusion-exclusion formulas in two-parameter persistence yield the same positive number, which equals our count, and which in turn equals the number of end-curves, giving geometric meaning to this count.
    We show that the end-curves determine the classical Betti tables by showing that they interpolate between generators, relations, and syzygies.
    Using the band representations of a certain string algebra, we show that the set of end-curves admits a canonical partition, where each part forms a closed curve on the plane; we call this the boundary of the module.
    As an invariant, the boundary is neither weaker nor stronger than the rank invariant, but, in contrast to the rank invariant, it is a complete invariant on the set of spread-decomposable representations.
    Our results connect several lines of work in multiparameter persistence, and their extension to modules over the real-exponent polynomial ring in two variables relates to two-dimensional Morse theory.
    % and the singularity theory of maps into the plane.
\end{abstract}

\maketitle

\section{Introduction}

\subsection{Context}

Any poset $\Pscr$ gives rise to a category, usually also denoted by $\Pscr$, with objects the elements of $\Pscr$ and with exactly one morphism $x \to y$ whenever $x \leq y \in \Pscr$.
A (linear) \emph{representation} of $\Pscr$ is a functor $\Pscr \to \Vect_\kbb$, where $\Vect_\kbb$ is the category of vector spaces over a field~$\kbb$, fixed throughout this paper.
Poset representations arise in several areas of pure and applied mathematics; notably in the representation theory of finite dimensional algebras \cite{simson}, and, most relevant to this paper, in \emph{persistence theory} \cite{oudot,botnan-lesnick}, an area grown out of applied topology and Morse theory.

In this introduction, we motivate the study of representations of products of linear orders from the points of view of applied topology, Morse theory, symplectic topology, and other areas in geometry and analysis.
Briefly, there exist several geometric constructions whose output is a \emph{one parameter representation} (a representation of a linear poset) or, in higher dimensional cases, a \emph{multiparameter representation} (a representation of a product of linear orders).
The case of one-parameter representations is well understood
algebraically~\cite{chazal-et-al},
and also in connection to Morse theory~\cite{bauer-medina-schmahl,buhovsky-et-al}
and for statistical purposes~\cite{fasy-et-al,skraba}.
The two-parameter case is significantly richer, and much less well understood, already at the algebraic level.
Many algebraic invariants of two-parameter representations have been proposed with the goal of interpreting and comparing these representations.
But our understanding of these invariants remains limited, with many of these invariants being incomparable between them in terms of strength,
having large output complexity, and lacking, so far, a clear geometric interpretation.

\subsubsection*{Poset representations in applied topology}
Geometric data can be studied using tools from algebraic topology \cite{frosini-landi,frosini-mulazzani,edelsbrunner-letsher-zomorodian,fasy-et-al} such as homology.
One way in which this is done is by using the input data to construct a nested family of simplicial complexes $K_1 \subseteq \cdots \subseteq K_n$, called a \emph{filtration}, and then taking homology with field coefficients, yielding a representation of the linear order $\{1 < \cdots < n\}$, with the goal of encoding the multiscale topological structure of the data algebraically.
Any finite linear order is of finite representation type, since its indecomposable representations correspond to the intervals, as is familiar from the representation theory of quivers of type~$A$.
Thus, the homology of a one-parameter filtration can be represented as a set of intervals, known as a \emph{barcode}, interpreted as multiscale topological features of the input data \cite{oudot}.
\emph{One-parameter persistence} refers to the study of filtrations and linear representations indexed by linear orders.

Applications to time-dependent, noisy, and otherwise parametrized data, motivate the development of \emph{multiparameter persistence}, a persistence theory for filtrations indexed by non-linear orders~\cite{botnan-lesnick}.
The fundamental challenge is that non-linear orders are typically of wild representation type \cite{nazarova} (i.e., their indecomposable representations cannot be effectively classified), which implies that many techniques and results from one-parameter persistence do not readily generalize~\cite{carlsson-zomorodian,bauer-scoccola}.
Because of this, many works have proposed incomplete invariants of poset representations%
~\cite{cerri-et-al,
biasotti-et-al,
carlsson-zomorodian,
lesnick-wright,
kim-memoli,
mccleary-patel,
scolamiero-et-al,
asashiba-escolar-nakashima-yoshiwaki,
asashiba-escolar-nakashima-yoshiwaki-2,
blanchette-brustle-hanson,
botnan-oppermann-oudot,
miller2,
bjerkevik},
as well as restricted families of representations whose indecomposables can be effectively classified
\cite{
botnan-lesnick-2,
bjerkevik-2,
botnan-lebovici-oudot,
bindua-brustle-scoccola}.

Persistence theory has been successfully applied to a wide range of scientific problems, including in neuroscience \cite{schneider-et-al,gardner-et-al}, biology \cite{rizvi-et-al,benjamin-et-al}, and physics \cite{sale-et-al,gilpin}.

%\luis{observe that, in one-parameter persistence, $\Ncaldec(M) = \dim M - \dim (M/\xbf M)$}
%\luis{motivate multiparameter persistence}

\subsubsection*{Poset representations in Morse theory}
The sublevel sets of a function from a topological space to the real numbers form a nested family of topological spaces indexed by the poset of real numbers $(\Rbb, \leq)$,
and by taking homology with field coefficients, one gets a representation of $\Rbb$, known as the \emph{sublevel set persistence} of the function.
This poset representation encodes the changes in topology as one moves across sublevel sets, and it is at the heart of Morse theory.
Although the interpretation of Morse theory from the point of view of poset representations has only recently been made explicit, it has been implicit since its inception:
critical values of generic Morse functions correspond exactly to barcode endpoints,
several invariants already considered by Morse~\cite{morse} are invariants of the barcode (e.g., the length of intervals),
and the fact that the sublevel set persistence of a Morse function admits a barcode is known since at least the 90's~\cite{barannikov};
see \cite{bauer-medina-schmahl} for a modern account.

Similarly, the sublevel set persistence of a function into the plane is a representation of the product poset $\Rbb^2 = (\Rbb, \leq) \times (\Rbb, \leq)$, connecting two-parameter persistence with two-dimensional Morse theory \cite{whitney,smale,wan,gay-kirby}, and indeed there has already been work exploring this connection
\cite{cerri-ethier-frosini,budney-kaczynski,assif-baryshnikov}
and the connection to Cerf theory \cite{bubenik-catanzaro}.
%This makes two-parameter persistence is intimately related to two-parameter persistence, and this connection has been explored in
A fundamental result in two-dimensional Morse theory
\cite{cerri-ethier-frosini,budney-kaczynski,assif-baryshnikov},
%,bubenik-catanzaro
implicit in earlier work~\cite{wan},
states that the homologically critical values of the sublevel set persistence of a two-dimensional Morse function (i.e., the points where the sublevel set homology changes) consist of a finite set of plane curves, known as the Pareto grid.
The construction of these curves is analytical, and uses the singular points of the function.
%Fundamental relations between a two-dimensional Morse function and its sublevel set persistence are studied in 
%But the at the algebraic level algebraic underpinnings of this connection are not, thus far, not well understood:
%for example, it is known that the homologically critical values of the sublevel set persistence of a two-dimensional Morse function (that is, the points where the sublevel set homology changes) consist of a set of plane curves
%(known as the Pareto grid) \cite{cerri-ethier-frosini,budney-kaczynski,assif-baryshnikov},
%but the definition of these curves relies on analytic properties of the Morse function, and
A definition of a suitable Pareto grid for general two-parameter representations, not necessarily coming from Morse functions, is not currently available, which stands in contrast to the one-parameter case, where homologically critical values can be recognized both analytically, as singular values, and purely algebraically, as endpoints of bars.
%\luis{This last thing needs to be reworded, since one could interpret Erza's work as doing exactly that}.

%This stands in contrast to one-dimensional Morse theory, where the homologically critical values are the same as analytical cr of the sublevel set persistence of a Morse function can be recognized both analytically, as singular values, and purely algebraically, as the endpoints of bars.

\subsubsection*{Poset representations in geometry and analysis}
Invariants derived from barcodes also show up in the context of filtered Floer homology, whose main construction, the Floer complex, gives rise to representations of $\Rbb$.
Classical invariants in symplectic topology such as
the spectral invariant~\cite{viterbo,schwarz,oh}, the boundary depth~\cite{usher,usher-2},
and the torsion exponents~\cite{fukaya-1,fukaya-2}
correspond to invariants of the barcode, namely
the endpoints of infinite intervals and the length of the longest finite interval \cite{polterovich-shelukhin}, and the length of the $k$th longest finite interval \cite{usher-zhang}, respectively.
In this context, persistence not only serves to reinterpret known invariants, but also to define new ones and prove new results; see, e.g., \cite{biran-cornea-zhang,shelukhin}.

Beside the endpoints and lengths of intervals, another invariant of the barcode is the number of intervals.
This number can be used to, for example, define a coarse version of the number of zeros of a function, which behaves well with repect to perturbations of the function, enabling coarse generalizations of classical results such as Courant's theorem \cite{buhovsky-et-al} and B\'ezout's theorem \cite{bujovsky-2}.
In the latter, a two-parameter representation encoding the coarse zeros of any map between normed spaces is described \cite[Remark~6.3]{bujovsky-2}; the analysis of this type of representation is left to future work, in part due to the lack of available invariants which would enable such study.

There exist several other applications of barcodes in geometry and analysis
such as in fractal geometry \cite{schweinhart-2}, metric geometry \cite{lim-memoli-okutan},
and quantitative homotopy theory \cite{block-manin-weinberger};
for more examples of these types of interactions, see \cite{polterovich-rosen-samvelyan-zhang}.

\subsubsection*{Signed invariants of poset representations}
With the above motivations, there is now a considerable literature focused on the design and study of invariants of representations of posets of wild representation type using
tools from graded commutative algebra~\cite{carlsson-zomorodian,lesnick-wright},
order theory~\cite{kim-memoli,asashiba-escolar-nakashima-yoshiwaki,mccleary-patel},
sheaf theory~\cite{curry,kashiwara-schapira-2,berkouk-petit},
and the representation theory of finite dimensional algebras~\cite{blanchette-brustle-hanson,asashiba-escolar-nakashima-yoshiwaki-2,
    botnan-oppermann-oudot,
    botnan-oppermann-oudot-scoccola}.
%In full generality, an \emph{invariant} on a category $\Ccal$ is an equivalence relation on the set of isomorphism classes of objects of $\Ccal$.
%For example, the \emph{complete invariant} is simply the finest equivalence relation 
%An \emph{invariant} is usually taken to be a function $\alpha$ mapping each representation $M$ of $\Pscr$ to an element of some prescribed set, with the property that $\alpha(M) = \alpha(N)$ whenever $M \cong N$.
%Such an invariant $\alpha$ is \emph{aditive} if $M \cong M'$ and $N \cong N'$ implies $M \oplus N \cong M' \oplus N'$.
These invariants usually take the form of an isomorphism-invariant function $\alpha$ from the category of representations of the poset in question to some set.
%a set $G$, and it is standard for these invariants to be \emph{additive}, meaning that the codomain $G$ is endowed with an Abelian group structure, and that $\alpha(M \oplus N) = \alpha(M) + \alpha(N)$ for any pair of representations $M$ and $N$.
%meaning that
%$\alpha(M) = \alpha(M')$ and $\alpha(N) = \alpha(N')$ implies that $\alpha(M \oplus N) = \alpha(M' \oplus N')$.
%Up to an is
%Invariants tend to be \emph{additive}, in that $\alpha(M) 

The simplest invariant is arguably the pointwise dimension of the representation.
Two strictly stronger invariants, already discussed in \cite{carlsson-zomorodian}, are the classical Betti tables, which record the grades of generators in a minimal resolution, and the \emph{rank invariant}, which records the rank of all the structure morphisms of the representation.
The rank invariant is particularly interesting since, in the one-parameter case, it is a complete invariant, meaning that it determines the isomorphism type of the representation \cite[Theorem~12]{carlsson-zomorodian}.
The main drawback of the rank invariant is that it does not have, a priori, a clear geometric interpretation like the barcode, and, outside of the one-parameter case, it is hard to derive interpretable invariants from it.

Because of this, the goal of more recent contributions is to mimic the barcode more closely, and to assign, to each representation $M$ of a fixed poset, a multiset~$\alpha(M)$ of spreads%
\footnote{A \emph{spread} of a poset is a subset that is poset-connected and poset-convex (see \cref{section:spreads}).
The spreads of a totally ordered sets are the intervals;
see \cref{figure:boundary-of-spread}(\emph{left}) for a two-dimensional example.
Spreads are also referred to as ``intervals'', but we prefer ``spread'' from \cite{blanchette-brustle-hanson} to avoid confusion with the standard notion of interval of a poset.
}
of the poset.
The main examples of this type of construction include the
\emph{generalized persistence diagram}~\cite{kim-memoli},
the \emph{signed barcode}~\cite{botnan-oppermann-oudot},
the \emph{minimal rank decomposition by hooks}~\cite{botnan-oppermann-oudot-scoccola},
the \emph{interval Euler characteristic}~\cite{escolar-kim},
the \emph{interval-decomposable replacements}~\cite{asashiba-escolar-nakashima-yoshiwaki},
and the \emph{minimal Hilbert decomposition}~\cite{oudot-scoccola};
see \cref{section:counts-from-known-invariants} for a description of these invariants using the language of this paper.
Although these constructions succeed in encoding part of the structure of a representation geometrically using spreads, they are fundamentally different from the one-parameter barcode in one key aspect: the multiset~$\alpha(M)$ is actually a signed multiset, that is, it is a formal linear combination $\alpha(M) = \sum_{j} c_j \cdot \Ical_j$ of spreads~$\Ical_j$, where the constants $c_j \in \Zbb$ can be (and often are) negative.
%This is an issue for deriving interpretable invariants from these constructions, as we explain below.

%\luis{``rank-type invariants'', ``numerical invariants'', ``dim-hom invariants''}

%Since the classification of indecomposable representations of a wild poset $\Pscr$
%An \emph{invariant} for an additive category $\alpha : \Acal \to G$.

\subsection{Problem statement and short summary of contributions}
As outlined in this introduction, there are three main invariants of interest that can be derived from the one-parameter barcode:
the number of intervals (known as the \emph{bar-count}),
the endpoints of the intervals (consisting of the \emph{birth-points} and the \emph{death-points}),
and the set of lengths of the intervals (known as the \emph{persistence} of the intervals).
We propose the goal of finding multiparameter analogues of these three invariants, even in the absence of a satisfactory analogue of the barcode.

%In this paper, we focus on finding notions of bar-count and of endpoints for 
%finite dimensional representations of the poset~$\Zbb^2$
%(or equivalently, for finite dimensional, bigraded $\kbb[\xbf,\ybf]$-modules).
%Finding a well-behaved notion of persistence, as well extensions of these notions to the continuous setting of $\Rbb^2$ and to the higher dimensional setting of $\Zbb^n$ and $\Rbb^n$ ($n \geq 3$), is left to future work.

We propose to associate a count $\Ncal^\alpha$ to each signed invariant $\alpha$ by defining $\Ncal^\alpha(M) = \sum_{i} c_i$, whenever $\alpha(M) = \sum_{i} c_i \cdot \Ical_i$, and ask the following questions:
%Since the current candidates for a multiparameter persistence barcode contain elements with negative multiplicities, it is not obvious how the multiparameter analogue of the bar-count should be defined.
%all multiplicities of $\alpha(M)$.

\smallskip
\emph{Do the counts associated with the different signed invariants in the literature coincide?
Are these counts positive?}
\smallskip

%, but, in this case, is this count positive? Moreover, is the count the same if we use the generalized persistence diagram, the signed barcode, or the interval approximation?
\noindent We give affirmative 
% (consequences of our two first main results \cref{theorem:universal-property-count,theorem:slice-monotonicity}) give affirmative 
answers to both of these questions 
in the (discrete) two-parameter case and
for the invariants given by the
generalized persistence diagram,
the signed barcode,
the minimal rank decomposition by hooks,
the interval Euler characteristic,
and the interval-decomposable replacements.
% (\cref{corollary:our-count-and-other-counts,corollary:main-properties-count}).
%\luis{We call this count the ``two-parameter count''}
We call the equivalent counts induced by these invariants the \emph{two-parameter count}, and give a simple inclusion-exclusion type formula for it in \cref{definition:two-parameter-count}.
The next natural question is:

\smallskip
\emph{
Is the two-parameter count counting a suitable notion of ``end-curves'' for two-parameter representations?}
\smallskip

\noindent We observe that a standard construction which, in the one-parameter case, outputs the sets of birth- and death-points of a representation, in the two-parameter case outputs two sets of plane curves, which we call the birth-curves and the death-curves of the representation (\cref{definition:endcurves}).
We give a positive answer to the second question by showing, also in \cref{corollary:our-count-and-other-counts}, that the two-parameter count equals both the number of birth-curves and the number of death-curves.
%Once we have a good candidate for the bar-count, we can ask about what the birth- and death-points should be.
%Since, in one parameter persistence, counting intervals is the same thing as counting birth-points, which is the same thing as counting end-points, we
%Note that whatever these are, it should be the case that the bar-count is equal to the number of birth-curves and to the number of death-curves.
%We follow the intution is that, in two-parameter persistence, the one-parameter endpoints, being $0$-dimensional, should be replaced with ``end-curves'', which ought to be $1$-dimensional.

Our other main results concern
the monotonicity of the two-parameter count with respect to one-parameter slices (\cref{theorem:slice-monotonicity}),
%(\cref{theorem:curves-betti-tables,theorem:boundary-tame})
and the discriminating power and planar geometry of the birth- and death-curves, particularly in connection with classical Betti tables (\cref{theorem:curves-betti-tables}), the rank invariant (\cref{theorem:boundary-tame,corollary:boundary-complete-spread-dec}), and spread-decomposable representations.
We also prove a computational result (\cref{proposition:main-computation-result}), asserting that projective presentations of end-curves can be computed in linear time.

%The posets that are most common in persistence theory can be finite, e.g., $\{1 < \dots < n\}$ or $\{1 < \dots < n\}^2$, infinite and discrete, e.g., $\Zbb$ or $\Zbb^2$, or infinite and continuous, e.g., $\Rbb$ or $\Rbb^2$.
%The main algebraic difficulty already shows up when moving from 
%$\{1 < \dots < n\}$ to $\{1 < \dots < n\}^2$, since the first one has finite representation type while the second one is wild (as long as $n \geq 4$).
%For this reason, when introducing a new invariant, it is common to start with finite or infinite discrete posets, in order to avoid the extra technicalities that arise from considering continuous posets.

\subsection{Main results}
%For details about standard notions, see background \cref{section:background}.
Any finite dimensional, graded $\kbb[\xbf]$-module $A$ decomposes in an essentially unique way as $A \cong \bigoplus_{i = 1}^k \kbb_{[b_i,d_i)}$, where $\{[ b_i,d_i ) \subseteq \Zbb\}_{i = 1}^k$ is a finite multiset of non-empty intervals of~$\Zbb$, and $\kbb_{[ b,d ) }$ is the graded $\kbb[\xbf]$-module $\left(\kbb[\xbf] / (\xbf^{d-b})\right)[-b]$.
The \emph{bar-count} of $A$ is $\Ncalone(A) = k$,
%, that is the count of indecomposable summands,
and the endpoints of $A$ consist of the \emph{birth-points} $\{b_i \in \Zbb\}_{i = 1}^k$ and the \emph{death-points} $\{d_i \in \Zbb\}_{i = 1}^k$.
The bar-count satisfies $\Ncalone(A) = \dim A - \dim \xbf A$, which motivates the following:

\begin{definitionx}
    \label{definition:two-parameter-count}
    The \emph{two-parameter count} of a finite dimensional, bigraded $\kbb[\xbf,\ybf]$-module $M$ is
    \[
        %\Ncaltwo(M) \coloneqq \dim(M/\xbf M) + \dim(M/\ybf M) - \dim(M/\xbf \ybf M) \;\; \in \;\; \Zbb.
        \Ncaltwo(M) \coloneqq \dim M - \dim \xbf M  - \dim \ybf M + \dim \xbf \ybf M \;\; \in \;\; \Zbb.
    \]
\end{definitionx}

The two-parameter count is always positive.
This follows from \cref{corollary:our-count-and-other-counts}, below, but it can also be proven directly using Sylvester rank inequality; see \cref{proposition:positivity}. 

If $A$ is the $\kbb[\xbf]$-module of the first paragraph, then the birth-points (resp.~death-points) of $A$ can be recovered from the indecomposable decomposition of the kernel (resp.~cokernel) of the morphism $A[-1] \to A$ given by multiplication by $\xbf$.
Indeed, there is an exact sequence
\[
    0 \; \to \; \bigoplus_{i = 1}^k\, \kbb_{d_i} \; \to \; A[-1] \; \xrightarrow{\,\, - \,\cdot\, \xbf\,\, } \; A \; \to \; \bigoplus_{i = 1}^k\, \kbb_{b_i} \; \to \; 0\,,
\]
where $\kbb_a = \kbb_{[a,a+1)}$ is the only indecomposable module (up to isomorphism) with support $\{a\}\subseteq \Zbb$.
% is the simple graded $\kbb[\xbf]$-module whose only non-zero homogeneous component is that corresponding to $a \in \Zbb$.
%Thus, if $\supind(B) = \{\sup(B_i) : 1 \leq i \leq k\}$
% can be computed as $\Ncalone(A) = \dim A - \dim \xbf A$ (since this is true for indecomposables), which motivates the following:
With this motivation, we prove
in \cref{corollary:birth-death-curves}
that, given a pointwise finite dimensional, bigraded $\kbb[\xbf,\ybf]$-module $M$, the morphism $M[-1,-1] \to M$ given by multiplication by $\xbf \ybf$ induces an exact sequence
\[
    %0 \to \kerxy{M} \to M[-1,-1] \xrightarrow{\,\, - \,\cdot\, \xbf\ybf\,\, } M \to \cokerxy{M} \to 0
    0 \; \to \; \bigoplus_{D \in \deaths(M)} \kbb_D \; \to \; M[-1,-1] \; \xrightarrow{\,\, - \,\cdot\, \xbf\ybf\,\, } \; M \; \to \; \bigoplus_{B \in \births(M)} \kbb_B \; \to \; 0\, ,
\]
where $\births(M)$ and $\deaths(M)$ are multisets of certain types of subsets of $\Zbb^2$
called spread curves (introduced next)
and $\kbb_I$ for $I$ a spread curve is the only indecomposable (up to isomorphism) with support $I \subseteq \Zbb^2$.
A \emph{spread curve} is a subset $I \subseteq \Zbb^2$ that is connected in the Hasse diagram of $\Zbb^2$, and that is as thin as possible (formally, such that $p \in I$ implies $p+(1,1) \notin I$).
%, and is corresponding \emph{curve module} is a bigraded $\kbb[\xbf,\ybf]$-module \luis{can we avoid giving the definition? also, add picture of a curve, and curve module}
See \cref{figure:main-figure,figure:kxy-gentle} for illustrations of spread curves.

\begin{definitionx}
    \label{definition:endcurves}
    The \emph{birth-curves} of
    a pointwise finite dimensional, bigraded $\kbb[\xbf,\ybf]$-module
    $M$ are the curves in $\births(M)$, and the 
    \emph{death-curves} of $M$ are the curves in $\deaths(M)$.
    %given by
    %    %The \emph{birth-curves} (resp.~\emph{death-curves}) of $M$ are given by
    %    %$\{\supp(M_i) \subseteq \Zbb^2 : 
    %    \begin{align*}
    %        \births(M)\, ,\, \deaths(M) & \;:\; \left\{\text{spread curves of $\Zbb^2$}\right\} \to \Zbb           \\
    %        \births(M)(I)               & \;\coloneqq\; \text{multiplicity of $\kbb_I$ in $\cokerxy{M}$}          \\
    %        \deaths(M)(I)               & \;\coloneqq\; \text{multiplicity of $\kbb_I$ in $\kerxy{M}$}
    %        %\substack{\text{multiplicity of $\kbb_I$ in}\\
    %        %          \text{indecomposable decomposition}\\
    %        %          \text{of $M/\xbf\ybf M$}}\\
    %        %\deaths(M)(I) &\coloneqq \text{multiplicity of $\kbb_I$ in indecomposable decomposition of $\ker(M \xrightarrow{\xbf\ybf} \xbf\ybf M)$}.
    %    \end{align*}
\end{definitionx}

To see that the kernel and cokernel of multiplication by $\xbf\ybf$ decompose as spread curves,
we observe that both these modules are \emph{ephemeral}, that is, they are annihilated by $\xbf\ybf$, and are thus bigraded $\kbb[\xbf,\ybf]/(\xbf\ybf)$-modules.
The bigraded algebra $\kbb[\xbf,\ybf]/(\xbf\ybf)$ is a string algebra, so ephemeral modules admit a simple classification:
\cref{theorem:decomposition-theorem-ephemeral}
(an easy consequence of the representation theory of string algebras \cite{crawley-boevey})
implies that every pointwise finite dimensional, indecomposable, bigraded $\kbb[\xbf,\ybf]/(\xbf\ybf)$-module is isomorphic to $\kbb_I$ for $I \subseteq \Zbb^2$ a spread curve.

\medskip

%\cref{definition:two-parameter-count,definition:endcurves} are clean from the perspective of bigraded modules
To connect these definitions with the literature, which often deals with representations of finite posets, we recall the following:
%but connections with the literature are more direct through the language of the representation theory of finite dimensional algebras and finite posets:
The category of finite dimensional, bigraded $\kbb[\xbf,\ybf]$-modules is equivalent to $\repf(\Zbb^2)$, the category of pointwise finite dimensional representations of $\Zbb^2$ of finite support.
If $\Gcal^2 = [0,a) \times [0,b) \subseteq \Zbb^2$ is a finite grid, then there is a fully faithful embedding $\rep(\Gcal^2) \hookrightarrow \repf(\Zbb^2)$, given by padding by zeros.
This allows us to see any finite dimensional representation of $\Gcal^2$ as a finite dimensional, bigraded $\kbb[\xbf,\ybf]$-module.
%If $\Gcal^2 \coloneqq \{1 \leq 2 \leq \dots \leq a\} \times \{1 \leq 2 \leq \dots \leq b\}$ is a finite, two-dimensional grid poset,
%padding by zeros induces a fully faithful embedding $\rep(\Gcal^2) \hookrightarrow \repf(\Zbb^2)$, which allows us to see any representation of $\Gcal^2$ as a finite dimensional, bigraded $\kbb[\xbf, \ybf]$-module.
%Moreover, up to a translation, any finite dimensional bigraded $\kbb[\xbf, \ybf]$-module is in the image of such an embedding.

%\luis{skip thm A and just state the corollary, say that it follows from a universality result}
%Our first main result is a universal property for the two-parameter count on representations of two-dimensional grids.
%For the notion of spread-generated invariant, see \cref{definition:generated-invariant}.
%\luis{introduce $\preceq$ notation}
%\luis{introduce spreads and additive invariants somewhere}

%Then, there is a fully faithful embedding from $\fintwoparam$ to the category of finite dimensional, bigraded $\kbb[\xbf,\ybf]$-modules, so $\Ncaltwo$ can be seen as an additive invariant on the category $\fintwoparam$ taking values in the Abelian group $\Zbb$.

%Our informal interpretation of \cref{theorem:universal-property-count} is that the count $\Ncaltwo$ is the only one that takes the value~$1$ on spreads, and that is fully determined by spreads.
Our first main result says that $\Ncaltwo$ coincides with several counts derived from signed invariants of two-parameter persistence representations, as well as with the number of birth-curves ($\Ncal^{\births}$) and with the number of death-curves ($\Ncal^\deaths$).
This gives geometric meaning to the count $\Ncaltwo$, and in particular implies positivity.
The signed invariants in the statement are
generalized persistence diagram~\cite{kim-memoli},
the signed barcode~\cite{botnan-oppermann-oudot},
the minimal rank decomposition by hooks~\cite{botnan-oppermann-oudot-scoccola},
the interval Euler characteristic~\cite{escolar-kim},
and the interval-decomposable replacements~\cite{asashiba-escolar-nakashima-yoshiwaki};
see \cref{section:counts-from-known-invariants}.
The count derived from the minimal Hilbert decomposition does not agree with the other counts, and is in fact quite uninteresting (\cref{remark:hilbert-decomposition-count}).
\cref{corollary:our-count-and-other-counts} is a consequence of a universality result for the two-parameter count (\cref{theorem:universal-property-count}),
which relies on the fact that, for finite lattices, the count $\chirkdec$ of a spread-representation equals the topological Euler characteristic of the spread (\cref{theorem:count-is-euler}).

\begin{restatable}{theoremx}{ourcountandothercounts}
    \label{corollary:our-count-and-other-counts}
    As invariants $\fintwoparam \to \Zbb$, we have
    \[
        \Ncaltwo
        = \chigpd
        = \chirkdec
        = \chirkexact
        = \chispreadexact
        = \Ncal^{\mathrm{int.dec.repl.}}
        = \Ncal^{\births} = \Ncal^\deaths.
    \]
\end{restatable}

%To simplify connections with the literature, we stated \cref{corollary:our-count-and-other-counts} at the level of representations of finite, two-dimensional grids, but the same methods imply an analogous result for finite dimensional representations of $\Zbb^2$, and suitable generalizations of the invariants from the literature to this setting.
Based on \cref{corollary:our-count-and-other-counts}, we extend the count $\Ncaltwo$ to the category of finitely generated, bigraded $\kbb[\xbf,\ybf]$-modules by $\Ncaltwo = \Ncal^\births$.

%\begin{restatable}{theoremx}{birthdeathdecomposition}
%    \label{theorem:birth-death-decomposition}
%    Let $M \in \fintwoparam$.
%    The modules $M^D,M^B \in \fintwoparam$ admit a unique decomposition as finite direct sums of spread curve modules.
%\end{restatable}

\medskip

Our second main result states that the two-parameter count is monotonic with respect to one-parameter slices and the one-parameter count.
%relates  with the one-parameter count $\Ncalone$ by stating that the two-parameter count is monotonic with respect to slicing.
%, meaning that the bar-count of a restriction of finite dimensional representation of $\Zbb^2$ to a linearly ordered subset of $\Zbb^2$ is bounded above by the two-parameter count of the original representation.
%If one interprets counts as counting features,
This is a natural property, which is however not satisfied by $\Ncaldec$, the count of indecomposables summands with multiplicity \cite{buchet-escolar}.

\begin{restatable}{theoremx}{slicemonotonicity}
    \label{theorem:slice-monotonicity}
    If $\ell : \Zbb \hookrightarrow \Zbb^2$ is an injective monotonic map, and $M \in \repf(\Zbb^2)$, then
    \[
    \Ncaltwo(M) \geq \Ncalone(\ell^*(M)).
    \]
\end{restatable}

\cref{theorem:slice-monotonicity} has several interesting consequences, such as the following.
In the result, a \emph{spread representation} is the indicator representation of a spread (\cref{figure:boundary-of-spread}), and a \emph{spread-decomposable representation} is one that decomposes as a direct sum of spread representations.

%, a main one being that $\Ncaltwo$ is strictly positive on non-trivial modules.
%Another interesting consequence
%such as the fact that
%is that $\Ncaltwo$ characterizes spread representations as those for which the count is one.

\begin{restatable}{corollaryx}{mainpropertiescount}
    \label{corollary:main-properties-count}
    Let $M \in \repf(\Zbb^2)$.
    \begin{enumerate}
        %\item $\Ncaltwo(M) \geq 0$.
        \item $\Ncaltwo(M) \geq \dim M_i$ for every $i \in \Zbb^2$.
        \item $\Ncaltwo(M) = 0$ if and only if $M \cong 0$.
        \item $\Ncaltwo(M) = 1$ if and only if $M$ is a spread representation.
        \item $\Ncaltwo(M) \geq \Ncaldec(M)$, and $\Ncaltwo(M) = \Ncaldec(M)$ if and only if $M$ is spread-decomposable.
    \end{enumerate}
\end{restatable}

%\luis{\cref{theorem:universal-property-count}
%    \cref{corollary:our-count-and-other-counts}
%    \cref{theorem:slice-monotonicity}
%    \cref{corollary:main-properties-count} proven}

\medskip

Our second set of results concerns the discriminating power and planar geometry of end-curves.
%Informally, \cref{theorem:curves-betti-tables} says that end-curves interpolate between the Betti tables (i.e., generators, relations, and syzygies) of the module, and \cref{theorem:boundary-tame} implies that there are two pairings between birth- and end-curves, and that, by gluing the curves along these, one obtains closed plane curves, see, e.g., \cref{figure:main-figure} \luis{revisit this sentence}.

\begin{definitionx}
    \label{definition:boundary-modules}
Let $M \in \rep(\Zbb^2)$.
Define $\cokerxy{M}, \kerxy{M}, \cokerx{M}, \kerx{M}, \cokery{M}, \kery{M} \in \repf(\Zbb^2)$ using the kernels and cokernels of multiplication by $\xbf, \ybf, \xbf\ybf \in \kbb[\xbf, \ybf]$, as follows:
\begin{alignat*}{5}
    0 & \to \kerxy{M}  &  & \hookrightarrow M[-1,-1]     &  & \xrightarrow{\xbf\ybf} M &  & \twoheadrightarrow \cokerxy{M}  &  & \to 0 \\
    0 & \to \kerx{M}   &  & \hookrightarrow \,\; M[-1,0] &  & \xrightarrow{\;\xbf\;} M &  & \twoheadrightarrow \cokerx{M}   &  & \to 0 \\
    0 & \to \,\kery{M} &  & \hookrightarrow \,\; M[0,-1] &  & \xrightarrow{\;\ybf\;} M &  & \twoheadrightarrow \,\cokery{M} &  & \to 0
\end{alignat*}%
Define also $\topleft{M}, \botright{M} \in \repf(\Zbb^2)$ as the following images:
\begin{align*}
    \topleft{M}  & \coloneqq \im\left(\,\kery{M} \hookrightarrow M[0,-1] \twoheadrightarrow \cokerx{M}[0,-1]\,\,\right)    \\
    \botright{M} & \coloneqq \im\left(\,\kerx{M} \hookrightarrow M [-1,0]\twoheadrightarrow \cokery{M} [-1,0]\,\,\right)\, .
\end{align*}
\end{definitionx}

As we observed above, the modules $\cokerxy{M}$ and $\kerxy{M}$ are ephemeral, and decompose as the birth-curves and the death-curves of $M$.
The other modules in \cref{definition:boundary-modules} are also ephemeral, with $\cokerx{M}$ and $\kerx{M}$ decomposing as vertical spread curves (since they are annihilated by $\xbf$), and $\cokery{M}$ and $\kery{M}$ decomposing as horizontal spread curves (since they are annihilated by $\ybf$).
The modules $\topleft{M}$ and $\botright{M}$ are annihilated by both $\xbf$ and $\ybf$, and are thus semisimple, so that $\dec\topleft{M}$ and $\dec\botright{M}$ are multisets of points of $\Zbb^2$; we refer to these multisets as the \emph{top-left corners} and \emph{bottom-right corners} of $M$, respectively.
See \cref{figure:boundary-of-spread} for an example.
Here, we let $\dec M$ denote the multiset of indecomposable summands of $M$, counted with multiplicity, and we identify a spread representation with its support (see \cref{notation:multisets}), so that, for example, we have $\births(M) = \dec \cokerxy{M}$ and $\deaths(M) = \dec \kerxy{M}$.

\begin{figure}
    \includegraphics[width=1\textwidth]{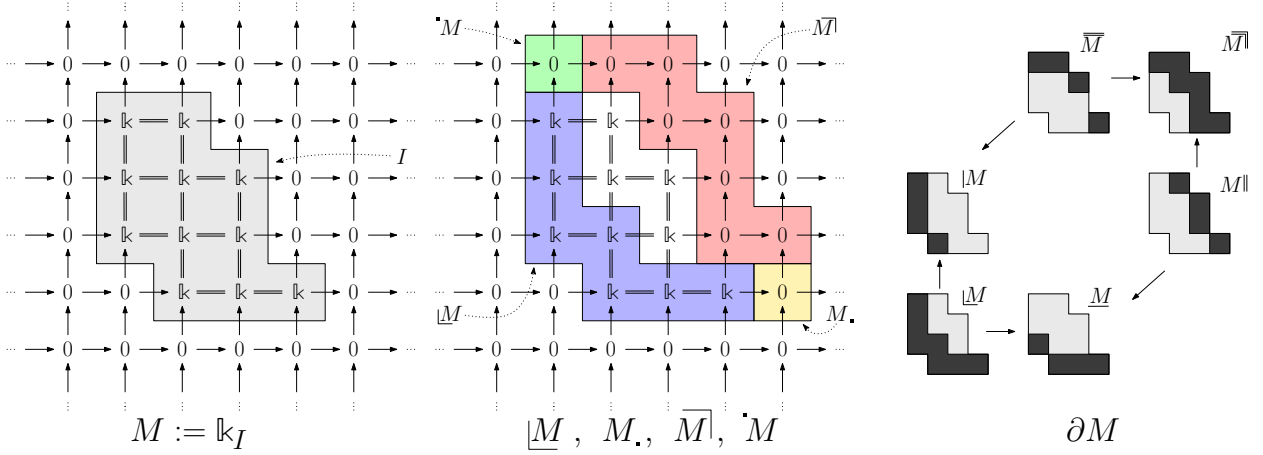}
    \caption{\emph{Left.} A spread $I \subseteq \Zbb^2$ and its corresponding spread representation $M = \kbb_I \in \rep(\Zbb^2)$.
        \emph{Center.} Birth-curves, death-curves, top-left corners, and bottom-right corners of $M$.
        \emph{Right.} The boundary of the module $M$.}
    \label{figure:boundary-of-spread}
\end{figure}

The next result implies that the Betti tables of $M$ can be read off from its end-curves; see \cref{section:betti-tables} for a standard description of Betti tables $\beta_i^M$
and \cref{figure:drawing-boundary}(\emph{right}) for an illustration.
% it says that birth-curves interpolate between generators and certain relations, while the death-curves interpolate between certain relations and syzygies.
%In words, the $0$th Betti table represents (bigrades of) generators, the $1$st Betti table represents (bigrades of) relations, and the $2$nd Betti table represents (bigrades of) relations between relations, or syzygies.
Note that a spread curve $I \subseteq \Zbb^2$ can be thought of as a one-dimensional path that interpolates between convex and concave corners; formally, we define $\conv(I)$ to be the set of minima of $I$ and $\conc(I)$ to be the set of maxima of $I$.
There are also the more restrictive notions of inner convex and inner concave corners, denoted $\inconv(I)$ and $\inconc(I)$, which are convex and concave corners, respectively, that are not endpoints of the curve (see \cref{definition:concave-convex-corners} for precise definitions).

\begin{restatable}{theoremx}{curvesbettitables}
    \label{theorem:curves-betti-tables}
    If $M \in \rep(\Zbb^2)$, then
    \begin{align*}
        %\beta_0(M) &= \sum_{B \in \births(M)} \sum_{x \in \conv(B)} \delta_x\\
        %\beta_1(M) &= \sum_{B \in \births(M)} \sum_{x \in \conc(B)} \delta_x
        %    + \sum_{D \in \deaths(M)} \sum_{y \in \conv(D)} \delta_y
        %    + \sum_{B \in \deaths(M)} \left(\delta_{\topleft(B)} + \delta_{\botright(B)}\right)\\
        %\beta_2(M) &= \sum_{D \in \deaths(M)} \sum_{y \in \conc(D)} \delta_y
        \beta_0^M \; & =\;
        \conv\left(\dec\cokerxy{M}\right) \\
        \beta_1^M \; & =\;
        \inconv\left( \dec\kerxy{M}\right)
        \,+\, \inconc \left( \dec\cokerxy{M}\right)
        \,+\, \dec\topleft{M}
        \,+\, \dec\botright{M}            \\
        \beta_2^M \; & =\;
        \conc\left( \dec\kerxy{M}\right)
    \end{align*}
    %\[
    %    \beta_0^M  =
    %    \conv\left(\dec\cokerxy{M}\right),\,
    %    \beta_1^M =
    %    \inconv\left( \dec\kerxy{M}\right)
    %    + \inconc \left( \dec\cokerxy{M}\right)
    %    + \dec\topleft{M}
    %    + \dec\botright{M} ,\,
    %    \beta_2^M  =
    %    \conc\left( \dec\kerxy{M}\right).
    %\]
    %Moreover, $\topleft{M} = \topleft{\left(\cokerxy{M}\right)}$ and $\botright{M} = \botright{\left(\cokerxy{M}\right)}$.
    Moreover $\topleft{M} \cong \topleft{\left(\cokerxy{M}\right)}$ and $\botright{M} \cong \botright{\left(\cokerxy{M}\right)}$ so the Betti tables of $M$ only depend on its end-curves.
\end{restatable}

%\luis{say that actually the complexity of the curves is in $|\beta_0^M|$, in the sense that there are as many corners as those}

As a consequence, we get that the two-parameter count of a finitely generated bigraded module is bounded above by the number of generators.
%, which in turn implies that the two-parameter count of the sublevel set persistence of a bifiltered simplicial complex is bounded above by the number of simplices.
This is remarkable, since, for instance, the size of the signed barcode is in general not linear \cite[Proposition~5.28]{botnan-oppermann-oudot-scoccola}, and the size of the generalized persistence diagram is not even polynomial \cite{kim-kim-lee}.
%, , and the interval approximation are not linear in the number of simplices \cite[Proposition~5.28]{botnan-oppermann-oudot-scoccola}, and sometimes not even polynomial \cite{kim-kim-lee}.
%In the result, if $\beta : \Zbb^2 \to \Zbb$ is a Betti table, its size is defined as $|\beta| = \sum_{i \in \Zbb^2} \beta(i) \in \Nbb$.

\begin{restatable}{corollaryx}{curveslinearsize}
    \label{corollary:curves-linear-size}
    If $M \in \rep(\Zbb^2)$ is finitely generated, then $\Ncaltwo(M) \leq |\beta^M_0|$.
    If $f : K \to \Zbb^2$ is a finite bifiltered simplicial complex, then $\Ncaltwo(H_i(f)) \leq |K_i|$ for every $i \in \Nbb$.
\end{restatable}

\medskip

Our last main result (\cref{theorem:boundary-tame}) is about the relationship between the birth-curves and death-curves.
The following notation\footnote{
The definitions of $\kerxy{M}$, $\kerx{M}$ and $\kery{M}$ follow the ``open convention'' in persistence theory, where a pair of integers $a \leq b \in \Zbb$ represents the half open interval $[a,b) \in \Zbb$; using the ``closed convention'', such a pair would represent the closed interval $[a,b]$.
The double-bar notation introduced here follows the closed convention.}
makes the statement simpler:
let
$\kerxyc{M} \coloneqq \kerxy{M}[1,1]$, $\kerxc{M} \coloneqq \kerx{M}[1,0]$, and 
$\keryc{M} \coloneqq \kery{M}[0,1]$.
Note that there is a canonical morphism $\delta M : \kerxyc{M} \to \cokerxy{M}$ given by the composite $\kerxyc{M} \hookrightarrow M \twoheadrightarrow \cokerxy{M}$.
We do not know if $\delta M$, in general, admits a simple geometric description,
but the following closely related and weaker construction does admit such a description, and this is what \cref{theorem:boundary-tame} states.
%The description is in terms of closed curves and (tame) algebraic data.

\begin{definitionx}
    \label{definition:boundary}
    Let $M \in \repf(\Zbb^2)$.
    The \emph{boundary} of $M$, denoted $\partial M$, is given by the following diagram, where the morphisms are the natural ones induced by the definitions of the objects:
    %the following diagram
    %of bigraded modules and bigraded morphisms, where the bidegree of the morphisms is indicated next to the head:
    %\begin{center}
    %    \includegraphics[width=0.3\textwidth]{figures/boundary-of-M.eps}
    %\end{center}
    \[
    \begin{tikzpicture}
        \matrix (m) [matrix of math nodes,row sep=1.5em,column sep=1.5em,minimum width=2em,nodes={text height=1.75ex,text depth=0.25ex}]
        {
             & \keryc{M} & \kerxyc{M}                            \\
            \cokerx{M}       &               & \kerxc{M}                             \\
            \cokerxy{M}      & \cokery{M}    & \\};
        \path[line width=0.75pt, ->]
        (m-1-2) edge (m-2-1)
        (m-1-2) edge (m-1-3)
        (m-3-1) edge (m-2-1)
        (m-3-1) edge (m-3-2)
        (m-2-3) edge (m-3-2)
        (m-2-3) edge (m-1-3)
        ;
    \end{tikzpicture}
%        \begin{tikzpicture}
%            \matrix (m) [matrix of math nodes,row sep=2em,column sep=2em,minimum width=2em,nodes={text height=1.75ex,text depth=0.25ex}]
%            {
%                \topleft{M} & \kery{M}   & \kerxy{M}                            \\
%                \cokerx{M}  &            & \kerx{M}                             \\
%                \cokerxy{M} & \cokery{M} & \botright{M}\\};
%            \path[line width=0.75pt, ->]
%            (m-1-2) edge [->>] (m-1-1)
%            (m-1-1) edge [right hook->] (m-2-1)
%            (m-1-2) edge [right hook->] (m-1-3)
%            (m-3-1) edge [->>] (m-2-1)
%            (m-3-1) edge [->>] (m-3-2)
%            (m-3-3) edge [left hook->] (m-3-2)
%            (m-2-3) edge [->>] (m-3-3)
%            (m-2-3) edge [right hook->] (m-1-3)
%            ;
%        \end{tikzpicture}
    \]
    %%    \;\;\;\;\;\;\;\;\;\;\;\;\;\;\;
    %%\begin{tikzpicture}
    %%            \matrix (m) [matrix of math nodes,row sep=2em,column sep=2em,minimum width=2em,nodes={text height=1.75ex,text depth=0.25ex},ampersand replacement=\&]
    %%            {
    %%                \bullet \& \bullet \& \bullet                         \\
    %%                \bullet \&   Q      \& \bullet                         \\
    %%                \bullet \& \bullet \& \bullet\\};
    %%            \path[line width=0.75pt, ->]
    %%            (m-1-2) edge (m-1-1)
    %%            (m-1-1) edge (m-2-1)
    %%            (m-1-2) edge (m-1-3)
    %%            (m-3-1) edge (m-2-1)
    %%            (m-3-1) edge (m-3-2)
    %%            (m-3-3) edge (m-3-2)
    %%            (m-2-3) edge (m-3-3)
    %%            (m-2-3) edge (m-1-3)
    %%            ;
    %%        \end{tikzpicture}
    %%\]
\end{definitionx}
Note that $\partial M$ encodes the birth- and death-curves of $M$, and relates them via the corners of $M$, since the corners are (shifts of) the images of the diagonal morphisms.
See \cref{figure:boundary-of-spread} for an example.

The boundary of $M$ belongs to a certain diagram category, which we denote $\boundaries$ for \emph{boundaries} (\cref{definition:ephemeral-pair}).
%There is a fully faithful embedding $\boundaries \hookrightarrow \rep(R)$, where $R$ is a string algebra, described in \cref{figure:large-gentle-quiver}, the upshot being that the pointwise finite dimensional, indecomposable representations of $R$ admit a classification.
%In particular, the finite dimensional band modules of $R$ can be classified in terms of boundary components.
%\cref{theorem:boundary-tame} states that the $\partial M$ decomposes as a 
The indecomposable objects of $\boundaries$ can be classified effectively, thanks to the fact that $\boundaries$ embeds into the category of representations of a string algebra (\cref{lemma:G-fully-faithful}).
The finite dimensional band modules of this string algebra can be classified in terms of boundary components, which essentially consist of an irreducible discrete closed plane curve and a matrix.
Here, an \emph{irreducible discrete closed plane curve} is a path in~$\Zbb^2$ that loops back to where it starts, and that is not the self-concatenation of a shorter path (\cref{definition:closed-curve}).

\begin{definitionx}
    \label{definition:boundary-component}
    %A \emph{discrete plane closed curve} is a sequence $\{v_0, \dots, v_{\ell-1}\}$ of elements of $\Zbb^2$ such that $\|v_i - v_{i+1}\|
    A \emph{boundary component} $(\gamma, \Tcal)$ consists of an irreducible discrete closed plane curve~$\gamma$ and an invertible square matrix $\Tcal$ over $\kbb$ that is not similar to a block diagonal matrix with more than one block.
    Two boundary components $(\gamma, \Tcal)$ and $(\gamma', \Tcal')$ are \emph{equivalent} if $\Tcal$ and $\Tcal'$ are similar, and the closed curves $\gamma$ and $\gamma'$ coincide up to possible rotation in their parametrizations.
\end{definitionx}

\cref{remark:boundary-component-module} associates,
to each boundary component $(\gamma, \Tcal)$, a \emph{boundary component representation}~$\kbb_{(\gamma, \Tcal)} \in \boundaries$, in such a way that
equivalent boundary components correspond with isomorphic boundary component representations.

%The eight modules
%$\cokerxy{M}, \kerxy{M}, \cokerx{M}, \kerx{M}, \cokery{M}, \kery{M}, \topleft{M}, \botright{M} \in \repf(\Zbb^2)$, naturally form a diagram, which we call the boundary of $M$, and which we introduce next.

%\noindent For a precise description of the morphisms in the diagram, see \cref{construction:boundary}.

%\luis{Introduce boundary components in words, and boundary component modules.}
%Say that their classification is tame.
%Say that the following result says that the image of $\partial$ is thus tame.}

\begin{restatable}{theoremx}{boundarytame}
    \label{theorem:boundary-tame}
    If $M \in \repf(\Zbb^2)$, then the boundary $\partial M \in \boundaries$ decomposes in an essentially unique way as a finite direct sum of boundary component representations, which are indecomposable.
\end{restatable}

The upshot is that, to each $M \in \repf(\Zbb^2)$, one can associate a finite collection of boundary components that completely encode $\partial M$; see \cref{figure:main-figure,figure:drawing-boundary} and \cref{section:examples} for examples.
Each boundary component consists of a geometric datum (the closed curve $\gamma$) and an algebraic datum (the matrix $\Tcal$), which is reminiscent of \cite{patel}.
There exist other interesting ways of encoding the geometric datum: as two pairings between the birth- and death-curves of $M$, as a graph structure on the Betti tables of $M$, or as a $\Zbb$-action on the Betti tables of $M$; see \cref{figure:drawing-boundary} for an example.

Our last main result says that the boundary is a complete invariant on spread-decomposable representations.
% (that is, representations that decompose as direct sums of indicator representations of spreads; see \cref{section:spreads}).

\begin{restatable}{corollaryx}{boundarycompletespreaddec}
    \label{corollary:boundary-complete-spread-dec}
    Two spread-decomposable representations $M,N \in \repf(\Zbb^2)$ are isomorphic if and only if $\partial{M} \cong \partial{N}$, which occurs if and only if the boundary components of $M$ counted with multiplicity are the same as those of $N$.
\end{restatable}

This implies that the boundary is not determined by the rank invariant, since the rank invariant is known to not be complete on the category of spread-decomposable representations; conversely, the boundary does also not determine the rank invariant (\cref{section:rank-and-boundary}), but we believe that there is a stronger connection between the two (\cref{section:boundary-and-rank-conjecture}).

%\luis{adjust intro to fit this result}
%
%\begin{restatable}{theoremx}{cubicalgorithm}
%    \label{theorem:cubic-algorithm}
%    The time complexity of computing the birth- and death-curves of $M \in \repf(\Zbb^2)$ is (at most) cubic in $|\beta_0(M)| + |\beta_1(M)|$.
%\end{restatable}
%
%\luis{add proof of result}

We conclude the paper by computing examples of end-curves and boundaries,
in \cref{section:examples},
and by proving that projective presentations of end-curves can be computed in linear time (\cref{proposition:main-computation-result}).

%\begin{restatable}{theoremx}{bandscompletespread}
%    \label{theorem:bands-complete-on-spreads}
%\end{restatable}

%\begin{restatable}{theoremx}{togmainresult}
%    Let $M,N : \Pscr \to \vect$ be TOG.
%    \begin{itemize}
%        \item If $\Pscr = \Gcal^2$, then $\Ccal(M) = \Ncaltwo(M)$.
%        \item If $M \hookrightarrow N$ or $N \twoheadrightarrow M$, then $\Ccal(M) \leq \Ccal(N)$.
%        \item (Monotonic-stability) If $\Pscr = \Rbb^n$ and $M$ and $N$ are $\epsilon$-interleaved, then $\Ccal(\smooth_{\delta + 2\epsilon}M) \leq \Ccal(\smooth_\delta N)$
%              and $\Ccal(\smooth_{\delta + 2\epsilon}N) \leq \Ccal(\smooth_\delta M)$ for all $\delta \geq 0$.
%    \end{itemize}
%\end{restatable}

\begin{figure}[p]
    \includegraphics[width=1\textwidth]{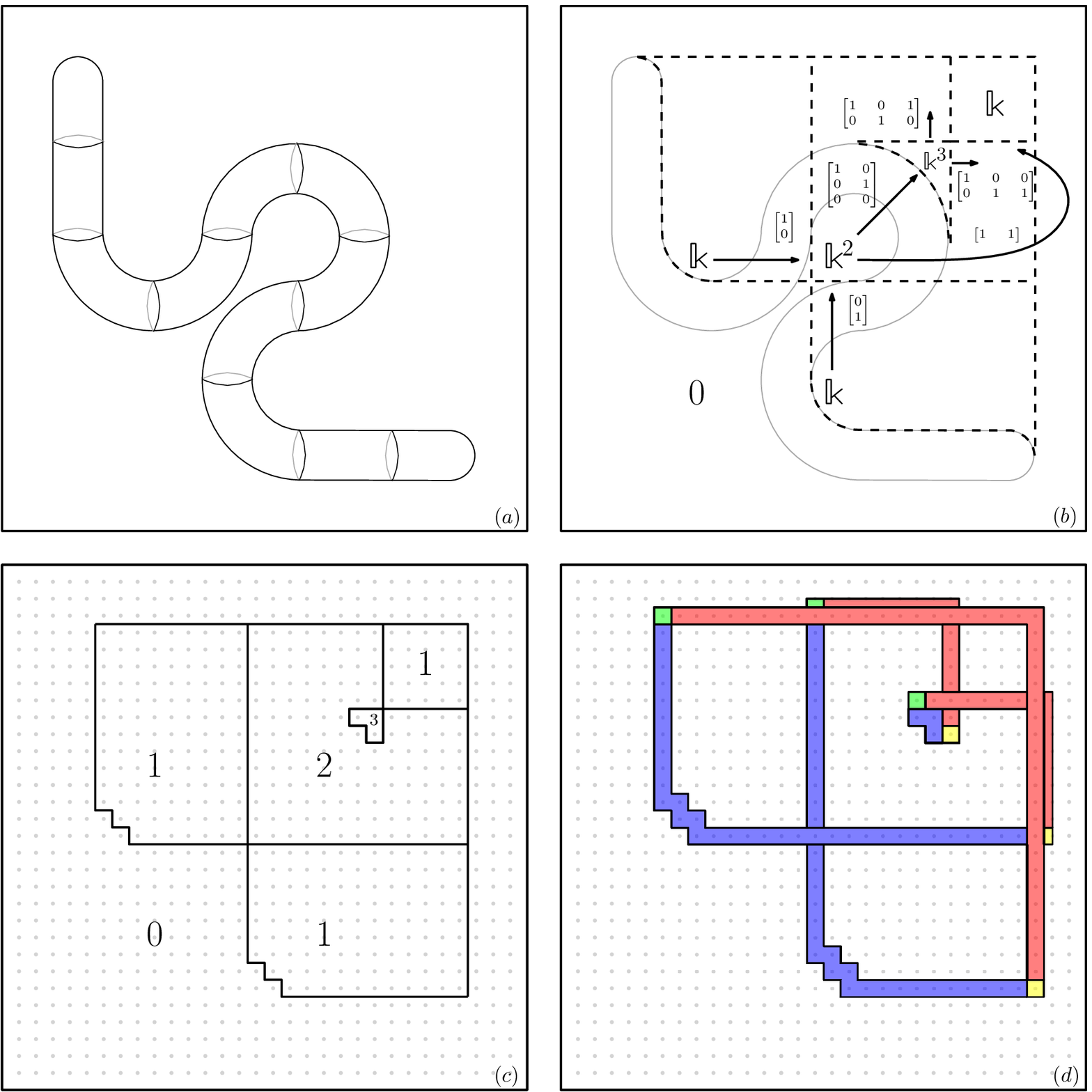}
    \caption{%
        (\textit{a}) A smooth function $f = (f_1, f_2) : S^2 \to \Rbb^2$ from the $2$-sphere to the plane.
        (\textit{b}) The one-dimensional sublevel set homology $H_1(f) : \Rbb^2 \to \vect$ of $f$ (i.e., the homology of the functor $(\Rbb, \leq)^2 \to \Top$ given by mapping $(x,y) \in \Rbb^2$ to $\{z \in S^2 : f_1(z) \leq x, f_2(z) \leq y\}$).
        Dashed lines indicate the boundary of regions in which the homology remains constant.
        (\textit{c})~The restriction of $H_1(f)$ to $\Zbb^2$ (in gray) resulting in a module $M \in \rep(\Zbb^2)$.
        Lines indicate the boundary of constant regions.
        (\textit{d}) The boundary $\partial M$ of $M$, consisting of birth-curves, death-curves, top-left corners, bottom-right corners.
        We slightly offset two sections of the boundary to avoid overlaps.
        In this case, the two-parameter count is $\Ncaltwo(M) = 3$, and $\partial M$ decomposes as a single boundary component.
        %\todo{Make sure that this example is correct.
        %And in particular that the boundary component modules have Jordan block of the form $J_{(1,1)}$.}
        }
    \label{figure:main-figure}
\end{figure}

\subsection{Discussion}
We discuss extensions, connections, and future work.

\subsubsection{Extension to continuous setting and two-dimensional Morse theory}
\label{section:extension-continuous}
Two-parameter representations coming from applied topology and Morse theory are often representations of a continuous poset $\Rbb^n$
(equivalently, graded modules over a real-exponent polynomial ring \cite{geist-miller}).
This prompts the question of whether our counts and end-curves can be extended to the continuous setting of~$\Rbb^2$.
It is not hard to extend the definition of end-curves to the case of finitely presented representations of $\Rbb^2$ (we omit the details for conciseness).
More generally, and more interestingly, we expect that such extensions are possible for particular kinds of tame representations in the sense of Miller~\cite{miller,miller2} and \cite{waas}.
%In the case of the sublevel set persistence of a two-dimensional Morse function,
We expect these suitably generalized end-curves to encode algebraically the 
Pareto grid of a two-dimensional Morse function \cite{cerri-ethier-frosini,budney-kaczynski,assif-baryshnikov}
and the Cerf diagram of a generic family of smooth functions
\cite{cerf,hatcher-wagoner,kirby}.

\subsubsection{Extension to higher number of parameters}
It is natural to wonder if there exist extensions of $\Ncalone$ and $\Ncaltwo$ to representations of $\Gcal^n$ or $\Zbb^n$, with $n \geq 3$.
Indeed, several of the equivalent descriptions of~$\Ncaltwo$ generalize readily;
this includes $\chigpd$, $\chirkdec$, and $\chispreadexact$, as well as the formula of \cref{definition:two-parameter-count},
which admits a natural inclusion-exclusion-type generalization $\Ninclexcl$ (\cref{definition:N-inclusion-exclusion}).

%$\Ninclexcl$ (\cref{definition:N-inclusion-exclusion}).
%those in \cref{remark:interpretation-of-counts}, whose definition makes sense in any number of parameters.
Although we have $\chirkdec = \Ninclexcl$ for every $n \geq 1$
(\cref{proposition:signed-barcode-equals-inclusion-exclusion}),
these is not a good candidate for a generalization of $\Ncalone$ and $\Ncaltwo$,
since it can fail to be positive (\cref{section:higher-dimensional-counts}).
This motivates the following questions.

\begin{questionx}
    When $n \geq 3$, are the counts $\chigpd$ and $\chispreadexact$ positive on $\rep(\Gcal^n)$?
\end{questionx}

\begin{questionx}
    When $n \geq 3$,
    is $\chigpd = \chispreadexact$ on $\rep(\Gcal^n)$?
\end{questionx}

%\begin{questionx}
%    \luis{Actually, this is easy to prove. Prove it.}
%    Is $\Ninclexcl = \chirkdec$ on $\rep(\Gcal^n)$?
%\end{questionx}

%In this context, \cref{theorem:curves-betti-tables} suggests that the end-curves of such continuous representations would serve as a well-behaved alternative to the Betti tables, which in that case are typically infinite or undefined (due to the lack of minimal projective resolutions).

%the end-curves o be intimately related
%Our end-curves seem very related to the Pareto grid from two-dimensional Morse theory and Cerf diagrams.
%The exact relationship between these notions is unclear at the moment,
%and, in order to establish a precise connection, one first needs to extend our curves to the continuous setting, as described above.

\subsubsection{Stability}
A central topic in persistence theory, not addressed in this paper, is stability with respect to perturbations of representations in the interleaving distance $d_I$; see \cite[Chapter~3]{oudot} and \cite[Section~6]{botnan-lesnick}.
Although the standard stability result for one-parameter persistence
\cite{cohen-steiner-et-al,chazal-et-al-2}
is stated in terms of the barcode, it does imply a stability result for the bar-count \cite[Theorem~6.1]{bujovsky-2}.
This result does not generalize readily to the two-parameter count, due to the fact that the two-parameter count is not monotonic with respect to inclusions (see \cref{section:monotonicity-fail}).
This issue is not specific to our count, since, as is shown in \cref{section:monotonicity-fail}, any additive count that takes the value $1$ on spreads must fail to be monotonic with respect to inclusions.
Relatedly, in \cref{section:perturbations} we give an example showing some of the difficulties one needs to overcome when understanding the behavior of end-curves with respect to perturbations.

Understanding the behavior of the two-parameter count, end-curves, and boundary with respect to perturbations is an important open problem, which we expect to be related to the problem of finding a good candidate for the persistence of the features counted by the two-parameter count.
A related problem is that of using end-curves to extend
the coherent matching distance \cite{cerri-ethier-frosini-2,cerri-ethier-frosini} to two-parameter representations that do not necessarily come from filtering functions.

\subsubsection{The boundary and the rank invariant}
\label{section:boundary-and-rank-conjecture}
As shown in \cref{section:rank-and-boundary}, the rank invariant of a two-parameter representation does not determine its boundary (or even its end-curves), nor does the boundary determine the rank invariant.
However, we conjecture that the boundary does determine a non-trivial lower bound for the rank invariant, which can be computed geometrically using the closed curves of the boundary components.
Similarly, it is of interest to relate the boundary to the support of the signed barcode \cite{botnan-oppermann-oudot}, the Betti tables relative to the rank exact structure \cite{botnan-oppermann-oudot-scoccola}, and the curves of \cite{morozov-patel}.

\subsubsection{Algorithmic computation}
In the context of applied topology, it is necessary to have efficient algorithms for computing the invariants of interest.
%The fact that the number of end-curves is linear in the number of simplices (\cref{corollary:curves-linear-size}), together with known algorithms for decomposing modules over string algebras such as
%, suggests that there exist efficient algorithms for computing the end-curves
As we show in \cref{section:algorithm}, projective presentations for end-curves can be computed in cubic time, which allows us to compute end-curves using any decomposition algorithm, such as that in \cite{dey_et_al:LIPIcs.SoCG.2025.41}.
The more difficult problem is computing boundary components, and the existence of algorithms for decomposition of representations of string algebras such as \cite{laubenbacher-sturmfels}
give hope that efficient solutions to this problem exist.
%and boundary components of the sublevel set persistence of bifiltered simplicial complexes.
An algorithm for computing the decomposition of the homology of filtrations over a certain string algebra appeared recently in \cite[Section~9]{dey-lapinski-soriano}, developed independently of this work.
However, their context is not multiparameter persistence, but rather bifurcations in time-dependent dynamical systems and Conley--Morse theory, and
their algorithm deals with string representations but not with band representations, since the quivers they consider do not admit band representations.

\subsection{Structure of the paper}
In \cref{section:background}, we give background and notation.
In \cref{section:counts}, we introduce counts derived from signed invariants and end-curves.
In \cref{section:properties-counts}, we prove the universal property of the two-parameter count, \cref{corollary:our-count-and-other-counts}, and \cref{theorem:slice-monotonicity}.
In \cref{section:end-curves-betti}, we prove \cref{theorem:curves-betti-tables}.
In \cref{section:bands}, we prove \cref{theorem:boundary-tame}.
In \cref{section:examples}, we compute examples of end-curves and boundaries.
In \cref{section:algorithm}, we prove a fundamental result concerning the computation of end-curves.

\subsection*{Acknowledgements}
TB was supported by Bishop's University, Universit\'e de Sherbrooke and NSERC Discovery Grant RGPIN/04465-2019.
LS thanks Vadim Lebovici and Vukašin Stojisavljević for clarifying discussions about topics in this paper, and specifically Vukašin for describing connections between persistence theory and symplectic topology.
LS was supported by a Centre de Recherches Mathématiques et Institut des Sciences Mathématiques fellowship. HT acknowledges support from NSERC (RGPIN-2022-03960) and the Canada Research Chairs program (CRC-2021-00120).

\section{Background and notation}
\label{section:background}

We assume familiarity with elementary notions from category theory \cite{maclane} and quiver representations \cite{derksen-weyman}.
% graded commutative algebra \cite{miller-sturmels},
% representation theory of finite dimensional algebras 
We fix a field $\kbb$, we let $\Vect$ be the category of $\kbb$-vector spaces, and $\vect$ be the full subcategory spanned by finite dimensional vector spaces.
If $\Lambda$ is a ring, by \emph{$\Lambda$-module} we mean right $\Lambda$-module.
The category of $\Lambda$-modules is denoted $\Mod_{\Lambda}$, and the full subcategory spanned by finitely generated $\Lambda$-modules is denoted by $\mod_{\Lambda}$.

\subsection{Posets of interest}
We endow the set of integers $\Zbb$ with its standard order.
If $a \geq 1 \in \Nbb$, we consider the finite totally ordered set $\{0, \dots, a-1\} \subseteq \Zbb$.
A finite \emph{$n$-dimensional grid poset} is a product poset of the form $\Gcal^n \coloneqq \prod_{i = 1}^k \{0, \dots, a_k-1\} \subseteq \Zbb^n$,
where $n \geq 1 \in \Nbb$ and $a_1, \dots, a_k \geq 1 \in \Nbb$.
Since the exact values for $a_1, \dots, a_n$ do not matter in this paper, we suppress them from the notation.
%The \emph{infinite $n$-dimensional grid} is the product poset $\Zbb^n$, where $\Zbb$ is endowed with its standard ordering.

\subsection{Poset representations}
Let $\Pscr$ be a poset.
A linear \emph{representation} of $\Pscr$ is a functor $\Pscr \to \Vect$, and such representation is \emph{pointwise finite dimensional} if it takes values in finite dimensional vector spaces.
The category of representations of $\Pscr$ is denoted by $\Rep(\Pscr)$, and that of pointwise finite dimensional presentations by $\rep(\Pscr)$.
Note that, in the literature, poset representations are also called persistence modules.
If $M : \Pscr \to \Vect$, and $p \leq q \in \Pscr$, we let $\phi^M_{p,q} : M_p \to M_q$ denote the corresponding structure morphism of $M$.

If $\Gcal^n \subseteq \Zbb^n$ is an $n$-dimensional grid poset, padding by zeros induces a fully faithful embedding $\rep(\Gcal^n) \to \repf(\Zbb^n)$.

\subsection{Poset representations as modules}
The \emph{poset algebra} $\kbb \Pscr$ of a poset $\Pscr$ is freely generated, as a vector space, by pairs $[i,j]$ where $i \leq j \in \Pscr$, with the multiplication $[i,j] \cdot [k,l]$ being $[i,l]$ if $j=k$ and zero otherwise.
There is a fully faithful embedding $\Rep(\Pscr) \to \Mod_{\kbb \Pscr}$, and when $\Pscr$ is finite, the algebra $\kbb\Pscr$ is finite dimensional, and the embedding restricts to an equivalence of categories
$\rep(\Pscr) \simeq \mod_{\kbb \Pscr}$
(see, e.g., \cite[Lemma~2.1]{botnan-oppermann-oudot-scoccola}).
%In particular, $\mod_{\kbb \Gcal^n} \simeq \rep(\Gcal^n)$.

\medskip

If $1 \leq i \leq n$, we let $\ebf_i \in \Zbb^n$ denote the standard basis vector whose $i$th coordinate is one, with all other coordinates being zero.
An \emph{$n$-graded $\kbb[\xbf_1, \dots, \xbf_n]$-module} is a $\kbb[\xbf_1, \dots, \xbf_n]$-module of the form $M = \bigoplus_{i \in \Zbb^n} M_i$ such that $\xbf_j M_i \subseteq M_{i + \ebf_j}$.
We let $\grmod{n}_{\kbb[\xbf_1, \dots, \xbf_n]}$ denote the category of $n$-graded $\kbb[\xbf_1, \dots, \xbf_n]$-modules.
There is an equivalence of categories $\Rep(\Zbb^n) \simeq \grmod{n}_{\kbb[\xbf_1, \dots, \xbf_n]}$ given by mapping a functor $M : \Zbb^n \to \Vect$ to the graded module $\bigoplus_{i \in \Zbb^n} M_i$, with action $(- \cdot \xbf_j) : M_i \to M_{i+\ebf_j}$ given by $\phi^M_{i,i+\ebf_j}$.
This equivalence restricts to an equivalence of categories between $\repf(\Zbb^n)$
and the category of finite dimensional, $n$-graded $\kbb[\xbf_1, \dots, \xbf_n]$-modules.

%This also allows us to see any representation of $\Gcal^n$ as a finite dimensional $n$-graded $\kbb[\xbf_1, \dots, \xbf_n]$-module.

\noindent
\begin{center}
    \fbox{%
        \begin{minipage}{0.98\linewidth}
            \begin{notation}
                \label{remark:persistence-modules-as-other-things}
                We often use the following equivalences and fully faithful functors implicitly:
                %allow us to see a representation of a grid $\Gcal^n$ as a $\kbb \Gcal^n$-module, as a representation of $\Zbb^n$ with finite support, and as a finite dimensional, $n$-graded $\kbb[\xbf_1, \dots, \xbf_n]$-module:
                \begin{align*}
                    \mod_{\kbb \Pscr} &\simeq \rep(\Pscr), \text{ for every finite poset $\Pscr$,}\\
                    \mod_{\kbb \Gcal^n} &\simeq \rep(\Gcal^n) \hookrightarrow \repf(\Zbb^n) \hookrightarrow \rep(\Zbb^n) \hookrightarrow \Rep(\Zbb^n) \simeq \grmod{n}_{\kbb[\xbf_1, \dots, \xbf_n]}, \text{ for every $n \geq 1 \in \Nbb$}.
                \end{align*}
            \end{notation}
        \end{minipage}}
\end{center}

%\luis{Do we need the following notation in generality?}
%\begin{notation}
%    \label{notation:ker-coker-notation}
%    The \emph{degree} of a homogeneous element $\eta = \xbf_1^{d_1} \cdots \xbf_n^{d_n} \in \kbb[\xbf_1, \dots, \xbf_n]$, is the vector $\deg(\eta) = (d_1, \dots, d_n)$.
%    If $M$ is an $n$-graded $\kbb[\xbf_1, \dots, \xbf_n]$-module, we consider the exact sequence
%    \[
%        0 \to M^\eta \to M[-\deg(\eta)] \xrightarrow{- \cdot \eta} M \to M_{\eta} \to 0 \,,
%    \]
%    where $M[a]$ for $a \in \Zbb^n$ denotes the shift of $M$, so that $M[a]_i = M_{i+a}$.
%\end{notation}

%\luis{generalize the following construction from $n=2$ to arbitrary $n$}
%%Recall from \cref{remark:persistence-modules-as-other-things} that there is a fully faithful embedding
%%\[
%%    \rep(\kbb\Gcal^2) \simeq \fintwoparam \hookrightarrow n\text{-}\mathrm{gr.mod}_{\kbb[\xbf_1, \dots, \xbf_n]}
%%\]
%%given by mapping a representation $M \in \fintwoparam$ to the bigraded $\kbb[\xbf, \ybf]$-module that at bigrade $(a,b) \in \Zbb^2$ has $M(a,b)$ if $(a,b) \in \Gcal^2$, and $0$ otherwise.
%The action of $\xbf,\ybf \in \kbb[\xbf, \ybf]$ is given, respectively, by the action of $\xbf,\ybf \in \kbb \Gcal^2$, defined as follows:
%\begin{equation}
%    \label{equation:definition-of-x-and-y-for-grid}
%    \xbf \coloneqq \sum_{\substack{p \in \Gcal^2 \,\,\,\text{s.t.} \\ p+(1,0) \in \Gcal^2}} \big[p,p+(1,0)\big]
%    \;\;\;\; \text{and} \;\;\;\;
%    \ybf \coloneqq \sum_{\substack{p \in \Gcal^2 \,\,\,\text{s.t.} \\ p+(0,1) \in \Gcal^2}} \big[p,p+(0,1)\big]
%\end{equation}

\subsection{Betti tables}
\label{section:betti-tables}
The Betti tables (also known as Betti diagrams) of a finitely generated, $n$-graded $\kbb[\xbf_1, \dots, \xbf_n]$-module have several equivalent definitions.
One definition is as the grades of generators in a minimal projective resolution of the module (e.g., \cite[Definition~1.9]{miller-sturmels}).
We instead use the definition based on the Koszul complex; it is standard that these two definitions coincide; for a proof, see Proposition~1.28 and Lemma~1.32 of \cite{miller-sturmels}, and note that, although they are working with $\Nbb^n$-gradings, the proof is exactly the same for $\Zbb^n$-gradings.
The Koszul complex definition, which we now instantiate to the case $n=2$, also has the advantage that it makes sense for any pointwise finite dimensional module.

The \emph{Betti tables} $\beta_0^M, \beta_1^M, \beta_2^M : \Zbb^2 \to \Nbb$ of a pointwise finite dimensional, bigraded $\kbb[\xbf,\ybf]$-module $M$ are the multisets of elements of $\Zbb^2$ given by the dimension vectors of the homology of the Koszul complex $\beta_i^M(x,y) \coloneqq \dim H_i(K(M))(x,y)$, where
\[
    K(M)\;\; \coloneqq\;\; 0 \to M[-1,-1] \xrightarrow{\,\,\,\,[\ybf, -\xbf]^T\,\,} M[-1,0] \oplus M[0,-1] \xrightarrow{ \,\,[\xbf, \ybf] \,\,} M \to 0\,\,.
\]

\subsection{Spreads and spread representations}
\label{section:spreads}
Let $\Pscr$ be a poset.
If $x,y \in \Pscr$, a \emph{path} from $x$ to $y$ consists of a sequence $x_1, \dots, x_k$ such that $x = x_1$, $y=x_k$, and $x_i$ and $x_{i+1}$ are comparable in~$\Pscr$ for all $1 \leq i < k$.
A set $S \subseteq \Pscr$ is \emph{poset-connected} if it is non-empty, and for every $x,y \in S$ there exist a path from $x$ to $y$ entirely contained in $S$.
A set $S \subseteq \Pscr$ is \emph{poset-convex} if whenever $x \leq y \leq z \in \Pscr$ with $x,z \in S$, then $y \in S$.
A \emph{spread} of $\Pscr$ is a subset $I \subseteq \Pscr$ that is poset-connected and poset-convex.

If $I \subseteq \Pscr$ is a spread, there exists a unique functor $\kbb_I : \Pscr \to \vect$, called a \emph{spread representation}, that takes the value $\kbb$ on $I$ and the value $0$ otherwise, with all morphisms that are not forced to be zero being the identity $\kbb \to \kbb$.
%By an abuse of language, we sometimes identify spreads of $\Pscr$ and spread representations $\Pscr \to \vect$.

The \emph{segment} corresponding to $x \leq y \in \Pscr$ is the spread $[x,y] \coloneqq \{z \in \Pscr : x \leq z \leq y\}$, and the corresponding \emph{segment representation} is $\kbb_{[x,y]}$.
The \emph{simple} representation at $x \in \Pscr$ is $\kbb_x \coloneqq \kbb_{[x,x]}$.
The \emph{hook} corresponding to $x < y \in \Pscr \cup \{\infty\}$ is the spread $\langle x, y \langle\,\, \coloneqq \{z \in \Pscr : x \leq z \ngeq y\}$, where $\Pscr \cup \{\infty\}$ is the poset obtained by adding an element $\infty$ to $\Pscr$ which is larger than all other elements; the corresponding \emph{hook representation} is $\kbb_{\langle a , b \langle}$.

%\begin{remark*}
%    In the persistence literature, spreads are often called ``intervals''.
%    We prefer the term ``spread'' as it leaves no room for confusion with the established notion of interval in the theory of posets.
%    We also use the term ``segment'' to refer to what is often known as a closed interval, again to avoid confusion.
%\end{remark*}

\subsection{Multisets}
\label{notation:multisets}
A finite \emph{multiset} of elements of a set $X$ consists of a function $c : X \to \Nbb$ of finite support.
Then, the union of finite multisets is simply addition.
An \emph{indexed multiset} of elements of $X$ consists of a set $I$ and a function $m : I \to X$; and such multiset is finite if $I$ is finite.
Two indexed multisets $(I,m)$ and $(I',m')$ are \emph{isomorphic} if there exists a bijection $f : I \to J$ such that $m' \circ f = m$.
It is then clear that there exists a bijection between
finite indexed multisets of elements of $X$ up to isomorphism
and
finite multisets of elements of $X$, given by mapping $(I,m)$ to the function $X \to \Nbb$ which maps $x$ to $|\{i \in I : m(i) = x\}|$.
The \emph{cardinality} of a finite multiset $c : X \to \Nbb$, denoted $|c|$, is $\sum_{x \in X} c(x)$.

By decomposing a spread-decomposable representation $M$ we get a multiset of spread representations $\dec(M)$, which is uniquely characterized by the multiset of spreads given by the support of the representations in $\dec(M)$.
By a standard abuse of notation, we denote this multiset of spreads also by $\dec(M)$.

\subsection{Duality}
\label{section:duality}
Let $\Pscr$ be a poset.
Dualization of vector spaces induces an equivalence of categories $\rep(\Pscr^\op) \simeq \rep(\Pscr)^\op$.
If $\Pscr$ is self-dual (i.e., $\Pscr \cong \Pscr^\op$), then, composing these two equivalence, we get an equivalence $\rep(\Pscr) \simeq \rep(\Pscr)^\op$.
This, coupled with the fact that the notion of spread is self-dual (i.e., if $I \subseteq \Pscr$ is a spread, then the same set is a spread of $\Pscr^\op$), allows us to simplify some proofs by only considering one of two cases, the other one being dual.
Note that both $\Gcal^n$ and $\Zbb^n$ are self-dual.

\subsection{Finite dimensional algebras}
For references, see \cite{auslander-reiten-smalo, assem-simson-skowronski}.
Let $\Lambda$ be a finite dimensional $\kbb$-algebra.
We let $\indproj_\Lambda = \{P_i\}_{i \in I}$ be a complete set of non-isomorphic indecomposable projective $\Lambda$-modules, and $\indsimpl_\Lambda = \{S_i\}_{i \in I}$ with $S_i = P_i/\mathrm{rad}(P_i)$ be the corresponding set of simple $\Lambda$-modules.

\subsection{Modules over quiver algebras.}
Let $\Lambda = \kbb Q / (\rho)$ be the quotient of a quiver algebra.
A right $\Lambda$-module $M$ is \emph{unital} if $M \Lambda = M$.
Equivalently, a unital $\Lambda$-module can be viewed as a representation of $Q$ that satisfies the zero relations encoded by $\rho$, which justifies denoting the category of unital $\Lambda$-modules by $\Rep(Q,\rho)$.
A $\Lambda$-module is \emph{finitely controlled} (resp.~\emph{pointwise finite dimensional}) if for every vertex $v$ of $Q$, the set $e_v M$ is contained in a finitely generated submodule of $M$ (resp.~it is finite dimensional).
Clearly, pointwise finite dimensional implies finitely controlled.
We denote the category of unital, pointwise finite dimensional $\Lambda$-modules by $\rep(Q,\rho)$.

\smallskip

\subsection{String algebras}
\label{section:string-algebras}

We now recall the classification result for indecomposable modules over a string algebra; for this, we follow \cite{crawley-boevey}.

\smallskip
\noindent\textit{String algebras.}
A \emph{string algebra} in the sense of Crawley-Boevey (a generalization of the notion originally introduced by Butler and Ringel \cite{butler-ringel}) is an algebra of the form $\kbb Q / (\rho)$ where $Q$ is a (not necessarily finite) quiver, and $\rho$ is a set of paths of $Q$ of length at least two, such that
\begin{itemize}
    \item every vertex of $Q$ has in-degree at most two, and out-degree at most two;
    \item given an arrow $y \in Q$, there is at most one path $x y \notin \rho$, and at most one path $y z \notin \rho$.
\end{itemize}

See \cref{figure:kxy-gentle} for an example.

\smallskip
\noindent\textit{Words.}
Let $\Lambda = \kbb Q/(\rho)$ be a string algebra.
A \emph{letter} of $Q$ is either an arrow $x$ of $Q$
%(a \emph{direct letter})
or the formal inverse $x^{-1}$ of an arrow;
%(an \emph{inverse letter}).
the source and target of $x^{-1}$ are the target and source of $x$, respectively.
If $I$ is one of the four sets $\{0, 1, \dots, n\}$, $\Nbb$, $-\Nbb \coloneqq \{0, -1, -2, \dots\}$, or $\Zbb$, we define an \emph{$I$-word} $C$ as follows:
\begin{itemize}
    \item If $I \neq \{0\}$, then $C$ is a sequence of letters $C_i$ of $Q$, for each $i \in I$ with $i-1 \in I$, such that
          \begin{itemize}
              \item the target of $C_i$ is equal to the source of $C_{i+1}$;
              \item $C_i^{-1} \neq C_{i+1}$;
              \item no path in $\rho$, nor its formal inverse, appears as a sequence of consecutive letters in $C$.
          \end{itemize}
    \item If $I = \{0\}$, there are two
          %\emph{trivial}
          $I$-words $1_{v,\epsilon}$ for each vertex $v$ of $Q$, where $\epsilon \in \{-1, 1\}$.
\end{itemize}

\smallskip
\noindent\textit{Operations on words.}
The \emph{inverse} $C^{-1}$ of a word $C$ is given by inverting the letters of $C$ (with the convention that $(x^{-1})^{-1} = x$), and reversing their order.
By convention, $(1_{v,\epsilon})^{-1} = 1_{v, -\epsilon}$.
If $C$ is a $\Zbb$-word, and $n \in \Zbb$, its \emph{$n$-shift} $C[n]$ is the $\Zbb$-word mapping $i$ to $C_{i+n}$.
If $C$ is an $I$-word and $i \in I$, the associated vertex $v_i(C) \in Q_0$ is the target of $C_i$ (equivalently the source of $C_{i+1}$) if $I \neq \{0\}$, and it is $v$ if $C = 1_{v,\epsilon}$.

A word is \emph{direct} if it does not contain formal inverses, and it is \emph{inverse} if it only contains formal inverses.
A $\Zbb$-word is \emph{periodic} if $C = C[n]$ for some $n > 0 \in \Nbb$, and the minimal such $n$ is called the \emph{period} of $C$.
By convention, the $n$-shift $C$ of an $I$-word with $I \neq \Zbb$ is $C$ itself.

Define an equivalence relation $\sim$ on the set of all words by $C \sim D$ if and only if either $D$ is a shift of $C$, or $D$ is a shift of the inverse of $C$.

\smallskip
\noindent\textit{String and band modules.}
Given an $I$-word $C$, define a $\Lambda$-module $M(C)$ freely generated, as a vector space, by symbols $\{b_i\}_{i \in I}$, with the action of $\Lambda$ given by%
\footnote{Note that, as opposed to \cite{crawley-boevey}, we work with right $\Lambda$-modules throughout.}:
\begin{itemize}
    \item For every $v$ vertex of $Q$, we have $b_i e_v = b_i$ if $v_i(C) = v$ and $b_i e_v = 0$ otherwise.
    \item For every $x$ arrow of $Q$, we have $b_i x = b_{i-1}$ if $i-1 \in I$ and $C_i = x$;
          $b_i x = b_{i+1}$ if $i+1 \in I$ and $C_{i+1} = x^{-1}$;
          and $b_i x = 0$ otherwise.
\end{itemize}
A \emph{string module} is a module of the form $M(C)$, with $C$ a non-periodic word.

If $C$ is a $\Zbb$-word and $n \in \Nbb$, there is an isomorphism $i_{C,n} : M(C) \to M(C[n])$ given by $t_{C,n}(b_i) = b_{i-n}$.
If $C$ is a periodic word of period $n$, then $M(C)$ is a $\left(\kbb[T,T^{-1}],\Lambda\right)$-bimodule with the action of $T$ given by $t_{C,n}$, so, for any $\kbb[T,T^{-1}]$-module $V$, we define $M(C,V) \coloneqq V \otimes_{\kbb[T,T^{-1}]} M(C)$.
A \emph{band module} is a module of the form $M(C,V)$, with $C$ a periodic word and $V$ an indecomposable $\kbb[T,T^{-1}]$-module.
A \emph{primitive injective} band module is a module of the form $M(C,V)$ with $C$ a direct or inverse periodic word, and $V$ is the injective envelope of a simple $\kbb[T,T^{-1}]$-module.

%It is easily checked that, if $C \sim D$, then $M(C) \cong M(D)$, and if $C$ and $D$ are periodic, then $M(C,V) \cong M(D,V)$ for every $\kbb[T,T^{-1}]$-module $V$.

\begin{theorem}[{\cite[Theorems~1.1,~1.2,~and~1.4]{crawley-boevey}}]
    \label{theorem:string-algebra-module-classification}
    Let $\kbb Q / (\rho)$ be a string algebra.
    \begin{itemize}
        \item
              String modules, finite dimensional band modules, and primitive injective band modules are indecomposable.
              Two of these modules are isomorphic if and only if they belong to the same class, their associated words are equivalent, and, in the case of band modules, their associated $\kbb[T,T^{-1}]$-modules are isomorphic.
        \item Every unital, finitely controlled $\Lambda$-module decomposes in an essentially unique way as a direct sum of string modules, finite dimensional band modules, and primitive injective band modules.
    \end{itemize}
\end{theorem}

\begin{remark}
    Given an invertible square matrix $\Tcal \in \kbb^{n \times n}$ that is not similar to a block diagonal matrix with more than one block,
    %with no non-trivial eigenspaces,
    we obtain a finite dimensional, indecomposable $\kbb[T,T^{-1}]$-module with underlying vector space $\kbb^n$ and action $T \cdot v = \Tcal v$,
    and any finite dimensional, indecomposable $\kbb[T,T^{-1}]$-module is of this form.
    This allows us to use canonical forms, such as the rational canonical form, or the Jordan canonical form (when $\kbb$ is algebraically closed), to classify finite dimensional, indecomposable $\kbb[T,T^{-1}]$-modules up to isomorphism.
\end{remark}

\subsection{Additive invariants and bases}
\label{section:additive-invariants}

The following definitions are based on \cite{blanchette-brustle-hanson,amiot-brustle-hanson}; see \cite{escolar-kim} for an essentially equivalent setup.

\begin{definition}
    Let $\Acal$ be an additive category.
    An \emph{additive invariant} on $\Acal$ is a function $\alpha : \Acal \to G$ mapping objects of $\Acal$ to an Abelian group $G$, which is isomorphism invariant (i.e., $\alpha(X) = \alpha(Y)$ whenever $X \cong Y \in \Acal$) and additive (i.e., $\alpha(X \oplus Y) = \alpha(X) + \alpha(Y)$ for all $X,Y \in \Acal$).
\end{definition}

\begin{definition}
    If $\alpha$ and $\beta$ are additive invariants on $\Acal$, we say that $\alpha$ \emph{is finer} than $\beta$, and write $\alpha \succcurlyeq \beta$, if $\alpha(X) = \alpha(Y)$ implies $\beta(X) = \beta(Y)$ for all $X,Y \in \Acal$.
    If $\alpha \succcurlyeq \beta$ and $\beta \succcurlyeq \alpha$, we say that $\alpha$ and $\beta$ are \emph{equivalent}, and write $\alpha \approx \beta$.
\end{definition}

\begin{definition}
    A \emph{basis} for an additive invariant $\alpha : \Acal \to G$ consists of an indexed collection of objects $\{M_i \in \Acal\}_{i \in I}$, such that for every $M \in \Acal$ there exists a unique indexed collection of integers $\{c^M_i \in \Zbb\}_{i \in I}$ such that $c_i \neq 0$ for finitely many $i \in I$, and such that
    \[
        \alpha(M) = \sum_{i \in I} c^M_i \cdot \alpha(M_i).
    \]
\end{definition}

\section{Counts for poset representations}
\label{section:counts}

By a \emph{count} we mean an additive invariant (\cref{section:additive-invariants}) taking values in $\Zbb$.

\subsection{An inclusion-exclusion type count}
Our first count is a natural generalization of $\Ncaltwo$.

\begin{definition}
    \label{definition:N-inclusion-exclusion}
    Let $n \geq 1 \in \Nbb$.
    The \emph{inclusion-exclusion count} $\Ninclexcl : \repf(\Zbb^n) \to \Zbb$ is
    \[
        \Ninclexcl(M) \coloneqq \sum_{S \subseteq \{1, \dots, n\}} (-1)^{|S|}\, \dim \left(\xbf_S\, M \right)\,,
    \]
    where, if $S = \{s_1, \dots s_\ell\}$, we let $\xbf_S = \xbf_{s_1} \cdots \xbf_{s_\ell}$.
\end{definition}

Note that, when $n = 2$, the count $\Ninclexcl$ specializes to $\Ncaltwo$.

\subsection{Counts from end-curves}
We define end-curves and their associated count.

%Recall that we have a fully faithful embedding $\rep(\Zbb^2) \hookrightarrow 2\text{-}\mathrm{gr.mod}_{\kbb[\xbf, \ybf]}$.

\begin{definition}
    A representation $M \in \Rep(\Zbb^2)$ is \emph{ephemeral} if $\xbf\ybf M = 0$.
\end{definition}

\begin{definition}
    A \emph{spread curve} is a spread $I \subseteq \Zbb^2$ such that $p \in I$ implies $p+(1,1) \notin I$.
    A \emph{spread curve representation} is a representation of the form $\kbb_I$ with $I$ a spread curve.
\end{definition}

Equivalently, a spread curve is a spread $I$ such that $\kbb_I$ is ephemeral.

\begin{figure}
    \begin{center}
        \includegraphics[width=1\linewidth]{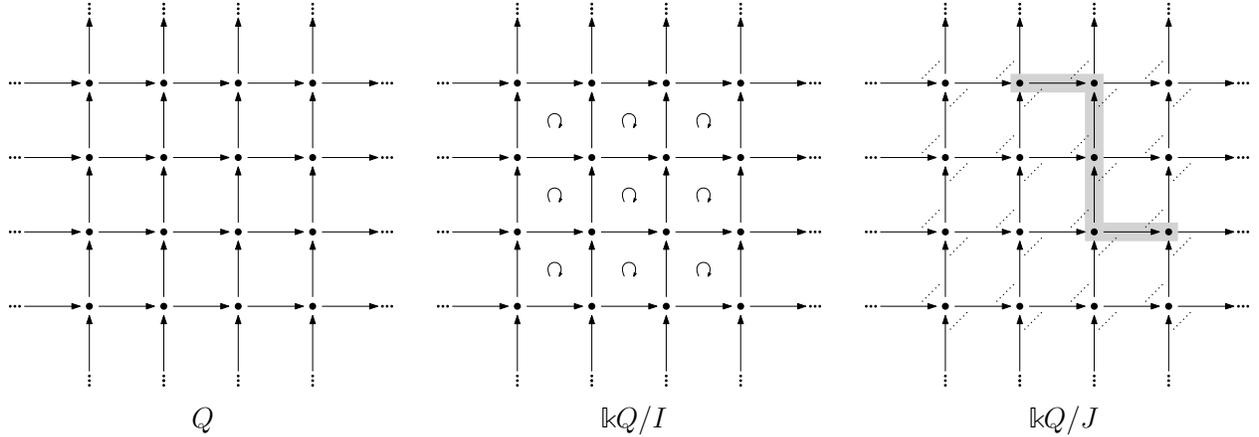}
    \end{center}
    \caption{\emph{Left.} A quiver $Q$.
    \emph{Center.} The quiver $Q$ with the ideal $I \subseteq \kbb Q$ of all commutativity relations.
    By identifying the vertices of the quiver $Q$ with the elements of $\Zbb^2$, we see that $\rep(Q,I) \simeq \rep(\Zbb^2)$.
    \emph{Right.} The quiver $Q$ with the ideal $J \subseteq \kbb Q$ generated by length-two paths, which contain both a vertical and a horizontal arrow.
    The algebra $\kbb Q/J$ is a string algebra, and using the identification between vertices of $Q$ and elements of $\Zbb^2$, the equivalence classes of words in this algebra correspond to spread curves.
    In gray is a word $C$, whose associated module $M(C)$ is a string module.
    %The string algebra is isomorphic to $\kbb\Zbb^2/(\xbf\ybf)$, where $\xbf, \ybf \in \kbb\Zbb^2$ are defined in \cref{equation:definition-of-x-and-y-for-grid}, and are, informally, the elements which act by moving up and to the right, respectively.
    As it can be observed, there are no periodic words in $\kbb Q/J$.
    }
    \label{figure:kxy-gentle}
\end{figure}

\begin{theorem}
    \label{theorem:decomposition-theorem-ephemeral}
    Every ephemeral representation in $\rep(\Zbb^2)$ decomposes in an essentially unique way as a direct sum of spread curve representations, which are indecomposable.
\end{theorem}
\begin{proof}
    We use the theory of string algebras recalled in \cref{section:string-algebras}.

    Let $Q$ be the quiver of \cref{figure:kxy-gentle}, and let $I$ and $J$ be the ideals of $\kbb Q$ also in \cref{figure:kxy-gentle}.
    The algebra $\kbb Q/J$ is a string algebra.
    By identifying the vertices of $Q$ with the elements of $\Zbb^2$, we get an equivalence of categories
    $\rep(Q,I) \simeq \rep(\Zbb^2)$.

    Note that $I \subseteq J$, which induces a fully faithful functor $F : \rep(Q,J) \to \rep(Q,I) \simeq \rep(\Zbb^2)$,
    and that $M \in \rep(\Zbb^2)$ is in the essential image of $F$ if and only if $M$ is ephemeral, since the ideal~$J$ consists of all paths in $Q$ that strictly increase both coordinates.
    %Let $J$ be the ideal of \cref{figure:kxy-gentle}, so that we get a quotient morphism $\kbb \Zbb^2 \cong \kbb Q/I \to \kbb Q/J$.

    We can parametrize any word in $\kbb Q/J$ in such a way that the $x$-coordinate is non-decreasing and the $y$-coordinate is non-increasing.
    This implies that the string algebra $\kbb Q/J$ admits no periodic words, so, by \cref{theorem:string-algebra-module-classification},
    all unital, finitely controlled $\kbb Q/J$-modules decompose in an essentially unique way as direct sums of string modules, which are indecomposable, and which in fact belong to $\rep(Q,J)$.
    %We get induced fully faithful functors $\Mod_{\kbb Q/J} \hookrightarrow \Mod_{\kbb Q/I} \simeq \Mod_{\kbb\Zbb^2} \hookleftarrow \rep(\Zbb^2)$.
    %It is straightforward to check that a module in $\Mod_{\kbb\Zbb^2}$ is in the essential image of the right-most functor restricted to the ephemeral representations if and only if it is in the essential image of the left-most functor restricted to the finitely controlled $\kbb Q/J$-modules.
    Finally, using the identification between the vertices of $Q$ and the elements of $\Zbb^2$, we see that equivalence classes of words in $\kbb Q/J$ correspond to spread curves, so that string modules in $\rep(Q,J)$ correspond, under~$F$, to spread curve representations.
    %To conclude, note that the string modules over $\kbb Q/J$ correspond, under the functors above, to spread curve modules $\Zbb^2 \to \vect$, since equivalence classes of words correspond to spread curves.
    %    Moreover, the string modules over $\Zbb^2 / (\xbf\ybf)$ map to spread curves
    %    Thus, it is sufficient to prove that all $\Zbb^2 / (\xbf\ybf)$-modules decompose as direct sums of spread curves.
    %
    %    We observe that $\Gcal^2 / (\xbf\ybf)$ is isomorphic to the path algebra of the quiver with relations in \cref{figure:kxy-gentle}.
    %    This algebra is clearly gentle, strings correspond precisely with spread curves, and there are no bands.
    %    Thus, the result thus follows from \cref{theorem:gentle-algebra-module-classification}.
    The result follows.
\end{proof}

%If $\Lambda$ is a commutative ring, $N$ is a $\Lambda$-module, and $a \in \Lambda$, there is an exact sequence of $\Lambda$-modules
%\[
%    0 \to \kerelem{N}{a} \to N \xrightarrow{- \cdot a} N \to N/aN \to 0.
%\]
%Taking $\Lambda = \kbb[\xbf,\ybf]$ and using \cref{remark:persistence-modules-as-other-things}, allows us to define the birth and death of a representation of $\Gcal^2$.

Recall from the introduction that, given $M \in \repf(\Zbb)$, we define the representations $\cokerxy{M}, \kerxy{M} \in \repf(\Zbb^2)$ using the cokernel and kernel of multiplication by $\xbf\ybf \in \kbb[\xbf, \ybf]$:
\[
    0 \to \kerxy{M}  \to M[-1,-1]  \xrightarrow{\xbf\ybf} M  \to \cokerxy{M}   \to 0\, .
\]

%\begin{definition}
%    Let $M \in \Rep(\Zbb^2)$.
%    The \emph{birth} of $M$ is the module $\cokerxy{M} \in \Rep(\Zbb^2)$.
%    The \emph{death} of $M$ is the module $\kerxy{M} \in \Rep(\Zbb^2)$.
%\end{definition}

\begin{corollary}
    \label{corollary:birth-death-curves}
    Let $M \in \rep(\Zbb^2)$.
    Then $\cokerxy{M}$ and $\kerxy{M}$ decompose in an essentially unique way as direct sums of spread curve representations.
\end{corollary}
\begin{proof}
    By construction, $\cokerxy{M}$ and $\kerxy{M}$ and are ephemeral,
    so the result follows from \cref{theorem:decomposition-theorem-ephemeral}.
\end{proof}

%By \cref{corollary:birth-death-curves}, if $M \in \rep(\Zbb^2)$, then $\cokerxy{M}$ and $\kerxy{M}$ are completely determined, up to isomorphism, by the multiplicity of the spread curve modules in their indecomposable decomposition.
We now restate \cref{definition:endcurves} more formally;
%The multiplicities of the spread curves in the indecomposable decomposition of $\cokerxy{M}$ and $\kerxy{M}$ are what we define to be the birth- and death-curves of $M$, respectively; see \cref{figure:main-figure}(d) for an example.
it is well-defined thanks to \cref{corollary:birth-death-curves}.

\begin{definition}
    \label{definition:births-deaths}
    Let $\spreadcurves$ denote the set of spread curves of $\Zbb^2$.
    We define the additive invariants $\births,\deaths : \rep(\Zbb^2) \to \Zbb^{\spreadcurves}$ to be the unique ones satisfying
    \[
        \cokerxy{M} \cong \bigoplus_{I \in \spreadcurves} (\kbb_I)^{\births(M)(I)} \;\;\; \text{and} \;\;\;
        \kerxy{M} \cong \bigoplus_{I \in \spreadcurves} (\kbb_I)^{\deaths(M)(I)},
    \]
    for every $M \in \rep(\Zbb^2)$.
\end{definition}

The following is immediate.

\begin{lemma}
    \label{lemma:spread-curves-basis-for-birth-death}
    Let $\spreadcurvesf$ be the set of finite curves of $\Zbb^2$.
    The set $\{\kbb_I\}_{I \in \spreadcurvesf}$ forms a basis for both additive invariants $\births$ and $\deaths$ on $\repf(\Zbb^2)$.
    \qed
\end{lemma}

We can then define two counts for a representation $M\in \repf(\Zbb^2)$, which simply count the number of birth- and death-curves of $M$.

\begin{definition}
    We define two counts $\repf(\Zbb^2) \to \Zbb$ as
    \[
        \Ncal^{\births} \coloneqq \Ncal^{\left(\births,\, \{\kbb_I\}_{I \in \spreadcurvesf}\right)}
        \;\;\; \text{and} \;\;\;
        \Ncal^{\deaths} \coloneqq \Ncal^{\left(\deaths,\, \{\kbb_I\}_{I \in \spreadcurvesf}\right)}.
    \]
\end{definition}

%\begin{lemma}
%    \label{lemma:births-of-birth}
%    Let $M \in \fintwoparam$.
%    Then, $\births(M) = \births(M^B)$ and $\deaths(M) = \deaths(M^D)$.
%\end{lemma}
%\begin{proof}
%    \todo{}
%\end{proof}

\subsection{Counts from additive invariants and bases}
The following is a natural way to produce a count from an additive invariant together with a basis; we prove \cref{corollary:standard-way-of-producing-count} that any count satisfying mild assumptions is of this form.

\begin{definition}
    \label{definition:counting-function-from-basis}
    Let $\Acal$ be an additive category, let $\Bcal = \{M_i \in \Acal\}_{i \in I}$ be a basis for an additive invariant $\alpha : \Acal \to G$, and let $f : \Bcal \to \Zbb$ be any function.
    The \emph{count associated} to $\alpha$, $\Bcal$, and $f$ is given by
    \[
        \Ncal^{(\alpha,\Bcal,f)}(M) \coloneqq \sum_{i \in I} c^M_i \cdot f(M_i).
    \]
    If $f$ is constantly $1$, then we simplify notation and write
    $\Ncal^{(\alpha,\Bcal)}(M) \coloneqq \sum_{i \in I} c^M_i$.
\end{definition}

This count satisfies the following useful characterization.

%Refinement is a transitive and reflexive relation on additive invariants on $\Acal$.

%\begin{notation}
%    If the rank invariant is determined by an additive invariant $\alpha$, we say that $\alpha$ is \emph{rank-determining}.
%\end{notation}

\begin{proposition}
    \label{lemma:basic-results-counts-basis}
    Let $\Acal$ be an additive category, let $\Bcal = \{M_i \in \Acal\}_{i \in I}$ be a basis for an additive invariant $\alpha : \Acal \to G$, and let $f : \Bcal \to \Zbb$.
    Then, the count $\Ncal^{(\alpha,\Bcal,f)}$ is the unique additive invariant $\Ncal  : \Acal \to \Zbb$ satisfying
    \begin{enumerate}
        \item $\alpha \succcurlyeq \Ncal$.
        \item $\Ncal(M_i) = f(M_i)$ for all $i \in I$.
    \end{enumerate}
\end{proposition}
\begin{proof}
    It is clear that $\Ncal^{(\alpha,\Bcal,f)}$ satisfies (1) and (2).
    Let $\Ncal$ also satisfy (1) and (2).
    Given $M \in \Acal$, let $\alpha(M) = \sum_{i \in I} c_i^M \cdot \alpha(M_i)$.
    By \cref{lemma:additivity-finer-invariant}, and the fact that $\Ncal$ equals $f$ on the elements of $\Bcal$, we have $\Ncal(M) = \sum_{i \in I} c_i^M \cdot \Ncal(M_i)= \sum_{i \in I} c_i^M \cdot f(M_i)$, which equals $\Ncal^{(\alpha, \Bcal,f)}(M)$, by \cref{definition:counting-function-from-basis}.
\end{proof}

%\begin{lemma}
%    Let $\Acal$ be an additive category, let $\alpha : \Acal \to G$ and $\Ncal : \Acal \to \Zbb$ be additive invariants, and let $\Bcal = \{M_i \in \Acal\}_{i \in I}$ be basis for $\alpha$.
%    If $\alpha \succcurlyeq \Ncal$, and $\Ncal(M_i) = 1$ for every $i \in I$, then $\Ncal = \Ncal^{(\alpha,\Bcal)}$.
%\end{lemma}
%\begin{proof}
%\end{proof}

\cref{lemma:basic-results-counts-basis} implies, in particular, that any count $\Ncal$ that can be computed from an invariant admitting a basis must be of the form of \cref{definition:counting-function-from-basis},
in the following sense:

\begin{corollary}
    \label{corollary:standard-way-of-producing-count}
    Let $\Acal$ be an additive category and let $\Ncal : \Acal \to \Zbb$ be a count.
    Suppose that there exists an additive invariant $\alpha$ on $\Acal$ that admits a basis $\Bcal = \{M_i\}_{i \in I}$, and such that $\alpha \succcurlyeq \Ncal$.
    Let $f : \Bcal \to \Zbb$ be defined by $f(M_i) = \Ncal(M_i)$.
    Then $\Ncal = \Ncal^{(\alpha, \Bcal, f)}$.\qed
\end{corollary}

\subsection{Counts from invariants in the literature}
\label{section:counts-from-known-invariants}

We first recall the definitions of several additive invariants and basis from persistence theory, and then use \cref{definition:counting-function-from-basis} to derive counts from these.

\begin{definition}
    \label{definition:three-basic-invariants-tda}
    %Let $\Pscr$ be a poset and let $I \subseteq \Pscr$ be a spread.
    %Given $M : \Pscr \to \vect$, let
    %\[
    %    \Rk_I(M) \coloneqq \rk\big(\lim M|_I \to \colim M|_I\big).
    %\]
    We define three families of additive invariants.
    \begin{itemize}
        \item
              Let $\Ical$ be a set of spreads of a poset $\Pscr$.
              The \emph{$\Ical$-rank invariant} \cite{botnan-oppermann-oudot,clause-kim-memoli}, denoted $\Rk_\Ical : \rep(\Pscr) \to \Zbb^\Ical$, is given by
              \[
                  \Rk_\Ical(M)(I) \coloneqq \mathrm{rank}\big(\lim M|_I \to \colim M|_I\big),
              \]
              for every spread $I \in \Ical$.
              The \emph{rank invariant} \cite{carlsson-zomorodian}, denoted $\Rk : \rep(\Pscr) \to \Zbb^{\rects}$ is the $\Ical$-invariant with $\Ical$ the set $\rects$ of all segments in $\Pscr$.
              Equivalently, the rank invariant is defined by $\Rk(M)([a,b]) = \rk(M_a \to M_b)$, for $a \leq b \in \Pscr$.
        \item Let $\Lambda$ be a finite dimensional $\kbb$-algebra and let $\Xcal$ be a set of pairwise non-isomorphic finite dimensional, indecomposable $\Lambda$-modules.
              The \emph{$\Xcal$-dimhom invariant} \cite{blanchette-brustle-hanson}, denoted $\dimhom_\Xcal : \mod_{\Lambda} \to \Zbb^{\Xcal}$, is given by
              \[
                  \dimhom_{\Xcal}(M)(X) \coloneqq \dim\left(\hom_{\Lambda}(X,M)\right).
              \]
              %If $\Lambda$ is a finite dimensional $\kbb$-algebra,
              The \emph{dimension vector} is the additive invariant $\dimhom_{\proj_\Lambda} : \mod_\Lambda \to \Zbb^{\proj_\Lambda}$.
        \item Let $\Lambda$ be a finite dimensional $\kbb$-algebra and let $\Xcal$ be a set of pairwise non-isomorphic finite dimensional, indecomposable $\Lambda$-modules.
              A sequence $A \to B \to C$ in $\mod_\Lambda$ is \emph{$\Xcal$-exact} if $\hom(X,A) \to \hom(X,B) \to \hom(X,C)$ is exact for every $X \in \Xcal$.
              The \emph{$\Xcal$-Grothendieck group invariant} \cite{blanchette-brustle-hanson}, denoted $\Ksf_\Xcal : \mod_\Lambda \to \Ksf(\Lambda, \Xcal)$, is given by
              \[
                  \Ksf_\Xcal(M) \coloneqq [M] \in \Ksf(\Lambda, \Xcal)
              \]
              mapping a module to its class in $\Ksf(\Lambda, \Xcal)$, the Grothendieck group relative to $\Xcal$, that is, the free Abelian group generated by isomorphism classes of $\Lambda$-modules, modulo relations $[B] = [A] + [C]$ for every short $\Xcal$-exact sequence $0 \to A \to B \to C \to 0$.
    \end{itemize}
\end{definition}

The following results give ways of producing bases for additive invariants; proofs are in \cref{section:appendix-additive-invariants-bases} since they are standard.

\begin{proposition}[\cite{botnan-oppermann-oudot,clause-kim-memoli}]
    \label{proposition:I-is-basis-of-rk-I}
    Let $\Pscr$ be a finite poset, and let $\Ical$ be a set of spreads.
    Then, $\{\kbb_{I}\}_{I \in \Ical}$ is a basis for $\Rk_\Ical : \rep(\Pscr) \to \Zbb^{\Ical}$.
\end{proposition}
\begin{proof}
    See \cref{section:appendix-additive-invariants-bases}.
\end{proof}

\begin{proposition}[{cf.~\cite{blanchette-brustle-hanson}}]
    \label{proposition:projectives-are-basis}
    Let $\Lambda$ be a finite dimensional $\kbb$-algebra, and let $\Xcal$ be a set of pairwise non-isomorphic finite dimensional, indecomposable $\Lambda$-modules.
    We have $\dimhom_{\Xcal} \approx \Ksf_{\Xcal}$, and if every finite dimensional $\Lambda$-module admits a finite $\Xcal$-projective resolution (\cref{definition:relative-projective-resolution}), then $\Xcal$ is a basis for $\dimhom_{\Xcal}$ and $\Ksf_\Xcal$.
\end{proposition}
\begin{proof}
    See \cref{proposition:projectives-are-basis-2}.
\end{proof}

\begin{corollary}
    \label{remark:basis-for-known-invariants}
    Let $\Pscr$ be a finite poset, and let $\spreads$, $\rects$, and $\hooks$ be the sets of spreads, segments, and hooks of $\Pscr$ (\cref{section:spreads}).
    \begin{enumerate}
        \item The set $\{\kbb_I\}_{I \in \spreads}$ is a basis for $\Rk_{\spreads}$.
        \item The set $\{\kbb_I\}_{I \in \rects}$ is a basis for $\Rk$.
        \item The set $\{\kbb_I\}_{I \in \hooks}$ is a basis for $\dimhom_\hooks \approx \Ksf_\hooks$.
        \item The set $\{\kbb_I\}_{I \in \spreads}$ is a basis for $\dimhom_\spreads \approx \Ksf_\spreads$.
    \end{enumerate}
\end{corollary}
\begin{proof}
    See \cref{section:appendix-additive-invariants-bases}.
\end{proof}

\begin{definition}
    Let $\Pscr$ be a finite poset.
    \begin{itemize}
    \item The function mapping $M \in \rep(\Pscr)$ to the coefficients of $\Rk_{\spreads}(M)$ in the basis $\{\kbb_I\}_{I \in \spreads}$ is the \emph{generalized persistence diagram} \cite{kim-memoli}%
    \footnote{A variation of the generalized persistence diagram was first introduced in \cite{kim-memoli}.
    We use the more standard definition~\cite{dey-kim-memoli,clause-kim-memoli},
    also called minimal generalized rank decomposition in \cite{botnan-oppermann-oudot}.}.
    \item The function mapping $M \in \rep(\Pscr)$ to the coefficients of $\Rk(M)$ in the basis $\{\kbb_I\}_{I \in \rects}$ is the \emph{signed barcode}
    %or \emph{minimal rank decomposition by rectangles}
    \cite{botnan-oppermann-oudot}.
    \item
    The function mapping $M \in \rep(\Pscr)$ to the coefficients of $\dimhom_\hooks(M)$ in the basis $\{\kbb_I\}_{I \in \hooks}$ is the \emph{minimal rank decomposition by hooks} \cite{botnan-oppermann-oudot-scoccola}.
    \item The function mapping $M \in \rep(\Pscr)$ to the coefficients of $\dimhom_{\spreads}(M)$, equivalently $\Ksf_\spreads(M)$, in the basis $\{\kbb_I\}_{I \in \spreads}$ is the \emph{interval Euler characteristic} \cite{escolar-kim}%
    \footnote{This was considered in a special case in \cite{asashiba-escolar-nakashima-yoshiwaki-2} and extended to general finite posets in \cite[Example~2.39]{escolar-kim}.}.
    \end{itemize}
\end{definition}

\begin{definition}
    \label{definition:known-counts}
    Let $\Pscr$ be a finite poset.
    We define four counts $\rep(\Pscr) \to \Zbb$ as
    \begin{align*}
        \chigpd         & \coloneqq \Ncal^{\left(\Rk_\spreads,\, \{\kbb_{I}\}_{I\in\spreads}\right)}   \\
        \chirkdec       & \coloneqq \Ncal^{\left(\Rk,\, \{\kbb_{I}\}_{I\in\rects}\right)}              \\
        \chirkexact     & \coloneqq \Ncal^{\left(\dimhom_\hooks,\, \{\kbb_I\}_{I \in \hooks}\right)} = \Ncal^{\left(\Ksf_\hooks,\, \{\kbb_I\}_{I \in \hooks}\right)} \\
        \chispreadexact & \coloneqq  \Ncal^{\left(\dimhom_\spreads,\, \{\kbb_I\}_{I \in \spreads}\right)} = \Ncal^{\left(\Ksf_\spreads,\, \{\kbb_I\}_{I \in \spreads}\right)} . 
    \end{align*}
\end{definition}

We also define a family of counts using the signed invariants of \cite{asashiba-escolar-nakashima-yoshiwaki};
we do not give more details since only abstract properties of these invariants are required to prove the corresponding part of \cref{corollary:our-count-and-other-counts}.

\begin{proposition}[{\cite{asashiba-escolar-nakashima-yoshiwaki}}]
    \label{proposition:compressed-multiplicities-properties}
    We use the notation of \cite{asashiba-escolar-nakashima-yoshiwaki}, and fix $\ast = \mathrm{ss}$, $\mathrm{cc}$, or $\mathrm{tot}$.
    The compressed multiplicities $\underline{d}^*$ form an additive invariant on $\rep(\Gcal^2)$.
    This invariant is finer than the rank invariant, and admits $\{\kbb_I\}_{I \in \spreads}$ as a basis.
\end{proposition}
\begin{proof}
    Additivity is proven in \cite[Lemma~4.19]{asashiba-escolar-nakashima-yoshiwaki}.
    The fact that the invariant is finer than the rank invariant is proven in \cite[Proposition~4.16]{asashiba-escolar-nakashima-yoshiwaki}.
    The fact that spreads form a basis is implied by \cite[Theorem~5.10]{asashiba-escolar-nakashima-yoshiwaki}.
\end{proof}

\begin{definition}
    We use the notation of \cite{asashiba-escolar-nakashima-yoshiwaki}, and fix $\ast = \mathrm{ss}$, $\mathrm{cc}$, or $\mathrm{tot}$.
    We define a count $\Ncal^{\mathrm{int.dec.repl.}} \coloneqq  \Ncal^{\left(\underline{d}^*,\, \{\kbb_I\}_{I \in \spreads}\right)}$. 
\end{definition}

%We could also define another count $\chirkexact \coloneqq \Ncal^{(\dimhom_\hooks, \hooks)}$, but this gives nothing new, as the next result shows.

\begin{remark}[Interpretation of counts]
    \label{remark:interpretation-of-counts}
    %\luis{adjust. we are saying this to conclude that we are taking poisitive - negative in what people have already considered}
    %\luis{also mention the Euler char interpretation in the case of exact structures}
    Let $\Pscr$ be a finite poset.
    \begin{enumerate}
        \item
             The count $\chigpd$ equals the number of positive bars minus the number of negative bars in the generalized persistence diagram.
        \item
              The count $\chirkdec$ equals the number of positive bars minus the number of negative bars in the signed barcode.
        \item
              The count $\chirkexact$ equals the number of positive bars minus the number of negative bars in the minimal hook decomposition.
              %, as well as the Euler characteristic of $M$ relative to the \emph{rank exact structure} of \cite{botnan-oppermann-oudot}.
        \item
              The count $\chispreadexact$ equals the number of positive bars minus the number of negative bars in the interval Euler characteristic.
              %, as well as the Euler characteristic relative to the \emph{spread exact structure}.
        \item
              The count $\Ncal^{\mathrm{int.dec.repl.}}$ equals the number of indecomposable summands in $\delta^\ast_+$ minus the number of indecomposable summands in $\delta^\ast_-$; see \cite{asashiba-escolar-nakashima-yoshiwaki}.
    \end{enumerate}
\end{remark}

\begin{remark}
    Although $\Rk_{\spreads}$ and $\dimhom_\spreads \approx \Ksf_\spreads$ seem very related, since, for example, they admit the same basis $\spreads$, these invariants are not equivalent for a general poset $\Pscr$ (see \cite[Section~7.3]{blanchette-brustle-hanson} and \cite[Theorem~4.1]{escolar-kim}).
    Perhaps surprisingly, the counts $\chigpd$ and $\chispreadexact$ are equal in the case where $\Pscr$ is a finite two-dimensional grid poset, thanks to \cref{corollary:our-count-and-other-counts}.
\end{remark}

The following remark also allows us to interpret $\chirkexact$ and $\chispreadexact$ as relative Euler characteristics.

\begin{remark}
    \label{remark:counts-as-euler}
    Let $\Lambda$ be a finite dimensional algebra, and let $\Xcal$ be a set of pairwise non-isomorphic, finite dimensional $\Lambda$-modules, such that every finite dimensional $\Lambda$-module admits a finite $\Xcal$ resolution.
    By \cref{proposition:projectives-are-basis}, we get counts $\Ncal^{(\dimhom_\Xcal, \Xcal)} = \Ncal^{(\Ksf_\Xcal, \Xcal)}$.
    If $M \in \mod_{\Lambda}$, these can be interpreted as the Euler characteristic of $M$ relative to $\Xcal$, in the sense that if $0 \to A_k \to \cdots \to A_0 \to M$ is an $\Xcal$-projective resolution of $M$, then the counts are equal to $\sum_{i \in \Nbb} (-1)^i \cdot \Ncaldec(A_i)$, where $\Ncaldec$ counts indecomposable summands (with multiplicity).
    In other words, the counts $\Ncal^{(\dimhom_\Xcal, \Xcal)} = \Ncal^{(\Ksf_\Xcal, \Xcal)}$ equal the alternating sum of the total multiplicities of the Betti tables relative to the $\Xcal$-exact structure in all homological degrees.
\end{remark}

%\luis{introduce some of these (we don't need all) as invariants with a basis}
%\begin{itemize}
%    \item $\chigpd(M)$, the (signed) sum of the multiplicities in the generalized persistence diagram of $M$ in the sense of \cite[Definition~2.9]{dey-kim-memoli}.
%    \item $\chirkdec(M)$, the number of positive bars minus the number of negative bars in the signed barcode of $M$ in the sense of \cite[Section~6.1]{botnan-oppermann-oudot}, and, more generally, $\chigrkdec{I}(M)$ is the number of positive spreads minus the number of negative spreads in the minimal $\Rk_\Ical$-decomposition of $M$, for $\Ical$ any set of spreads containing the rectangles of $\Gcal^2$, in the sense of \cite[Section~2]{botnan-oppermann-oudot}.
%    {\luis{define rank decompositions}}
%    \item $\chirkexact(M)$, the Euler characteristic (i.e., the number of indecomposables in even homological degrees minus the number of indecomposables in odd homological degrees in a resolution) relative to the exact structure on $\fintwoparam$ whose exact sequences are those that are additive with respect to the rank invariant \cite{botnan-oppermann-oudot,botnan-oppermann-oudot-scoccola}.
%    \item $\chispreadexact(M)$, the Euler characteristic relative to the exact structure on $\fintwoparam$ whose projective modules are the spread-decomposable modules \cite{asashiba-escolar-nakashima-yoshiwaki-2}.
%\end{itemize}

\begin{remark}
    \label{remark:hilbert-decomposition-count}
    Let $\Pscr$ be a finite poset.
    For $p \in \Pscr$, let ${\uparrow}p = \{q \in \Pscr : p \leq q\} \subseteq \Pscr$ be the corresponding principal upset, which is in particular a spread.
    Let $\mathsf{p.upset}$ be the set of principal upsets of $\Pscr$.
    It is standard \cite[p.~157]{gabriel} that the indecomposable projective objects of $\rep(\Pscr)$ are precisely those of the form $\kbb_{{\uparrow}p}$.
    The \emph{minimal Hilbert decomposition} of \cite{oudot-scoccola} can be equivalently described as the function mapping $M \in \rep(\Pscr)$ to the coefficients of $\dimhom_{\mathsf{p.upset}}(M)$, or equivalently $\Ksf_{\mathsf{p.upset}}(M)$, in the basis given by $\{\kbb_{{\uparrow}p}\}_{p \in \Pscr}$ (the fact that this is indeed a basis is a direct consequence of \cref{proposition:projectives-are-basis}).
    Note that $\dimhom_{\mathsf{p.upset}}$ simply records the pointwise dimension of the representation, and is also known as the Hilbert function.
    Thus, we get a count
    \[
        \Ncal^{\mathrm{Hilb.dec.}} \coloneqq 
    \Ncal^{
        \left(\dimhom_{\mathsf{p.upset}},
              \{\kbb_{{\uparrow}p}\}_{p \in \Pscr}
        \right)}=
    \Ncal^{
        \left(\Ksf_{\mathsf{p.upset}},
              \{\kbb_{{\uparrow}p}\}_{p \in \Pscr}
        \right)},
    \]
    which equals the number of positive bars minus the number of negative bars in the minimal Hilbert decomposition.
    In view of \cref{remark:counts-as-euler}, this count can also be interpreted as the (standard, non-relative) Euler characteristic.
    This count is, however, quite weak as an invariant:
    for example, if~$\Pscr$ admits a maximum $p \in \Pscr$, then $\Ncal^{\mathrm{Hilb.dec.}}(M) = \dim M_p$ for every $M \in \rep(\Pscr)$.
    In particular, this count is not equal to $\Ncaltwo$ on $\rep(\Gcal^2)$.
\end{remark}

\section{Properties of counts}
\label{section:properties-counts}

We now prove fundamental properties of the counts introduced in \cref{section:counts}.

\subsection{The two-parameter count}

We first prove that the two-parameter count is always positive.

\begin{proposition}
    \label{proposition:positivity}
    If $M \in \repf(\Zbb^2)$, then $\dim M - \dim \xbf M - \dim \ybf M + \dim \xbf \ybf M \geq 0$.
\end{proposition}
\begin{proof}
    By identifying the ungraded vector space underlying $M$ with $\kbb^n$, the actions of $\xbf$ and $\ybf$ are given by $n\times n$ matrices $A$ and $B$.
    Then, $\dim M = n$, $\dim(\xbf M) = \rank(A)$, $\dim(\ybf M) = \rank(B)$, and $\dim(\xbf\ybf M) = \rank(AB)$.
    The result follows from the Sylvester rank inequality, which states that, if $A$ is an $m \times n$ matrix and $B$ is an $n \times k$ matrix, then $\rank(AB) \geq \rank(A) + \rank(B) - n$.
\end{proof}

The two-paramter count only depends on the end-curves of the representation, in the following sense.

\begin{lemma}
    \label{equation:count-same-as-birth-death}
    If $M \in \repf(\Zbb^2)$, then $\Ncaltwo(M) = \Ncaltwo(\cokerxy{M}) = \Ncaltwo(\kerxy{M})$.
\end{lemma}
\begin{proof}
    We prove that $\Ncaltwo(M) = \Ncaltwo(\cokerxy{M})$; the other equality is analogous.
    Note that $\cokerx{(\cokerxy{M})} \cong \cokerx{M}$,
    $\cokery{(\cokerxy{M})} \cong \cokery{M}$,
    and $\cokerxy{\left(\cokerxy{M}\right)} \cong \cokerxy{M}$.
    The result then follows from the equality
    \[
        \Ncaltwo(M) = \dim \cokerx{M} + \dim \cokery{M} - \dim \cokerxy{M},
        %= \dim \kerx{M} + \dim \kery{M} - \dim \kerxy{M},
    \]
    which is a simple application of rank-nullity.
    %Simple computations show that
    %the constructions $M \mapsto \cokerx{M}$, $M \mapsto \cokery{M}$, and $M \mapsto \cokerxy{M}$ are idempotent
    %so, using \cref{lemma:alternative-count}(1), we get
    %$\Ncaltwo(\cokerxy{M}) = \dim \cokerx{M}+ \dim \cokery{M} - \dim \cokerxy{M} = \Ncaltwo(M)$.
    %The case of $\kerxy{M}$ is analogous.
\end{proof}

%\begin{lemma}
%    \label{lemma:alternative-count}
%    Let $M \in \repf(\Zbb^2)$.
%    We have
%    \begin{align}
%        \Ncaltwo(M) & = \dim \cokerx{M} + \dim \cokery{M} - \dim \cokerxy{M} \\
%                    & =
%        \dim \kerx{M} + \dim \kery{M} - \dim \kerxy{M}.
%    \end{align}
%\end{lemma}
%\begin{proof}
%    This follows directly from rank nullity.
%\end{proof}

\subsection{The signed barcode count equals the inclusion-exclusion count}

\begin{proposition}
    \label{proposition:signed-barcode-equals-inclusion-exclusion}
    Let $n \geq 1 \in \Nbb$.
    We have $\chirkdec = \Ninclexcl$ on $\rep(\Gcal^n)$.
\end{proposition}
\begin{proof}
    Applying \cref{lemma:basic-results-counts-basis}, it is enough to check that (1) 
    we have $\Rk \succcurlyeq \Ninclexcl$ as additive invariants on $\repf(\Zbb^n)$, and that (2)
    $\Ninclexcl(\kbb_{[a,b]}) = 1$ for every segment representation $\kbb_{[a,b]}$, $a\leq b \in \Zbb^n$.

    \smallskip

    \noindent\emph{(1)}
    It is sufficient to prove that, for every $S = \{s_1, \dots, s_\ell\} \subseteq \{1, \dots, n\}$, we can compute $\dim(\xbf_S M)$ using the ranks of the morphisms of $M$, where $\xbf_S = \xbf_{s_1} \cdots \xbf_{s_\ell}$.
    For this, let $b \in \Zbb^n$ be defined by $b_i = 1$ if $\xbf_i \in S$ and $b_i = 0$, otherwise, and note that
    \[
        \dim(\xbf_S M) = \sum_{a \in \Zbb^n} \mathrm{rank}\left(\phi^M_{a,a+b} : M_a \to M_{a + b}\right).
    \]

    \smallskip

    \noindent\emph{(2)}
    We proceed by induction on $n$.
    If $n = 1$, this is clear.
    For the inductive case, let $\kbb_{[a,b]}$ be a segment representation, with $a = (a_1, \dots, a_{n+1})$ and $b = (b_1, \dots, b_{n+1})$.
    Then, $\kbb_{[a,b]} \cong \kbb_{[a_1,b_1]} \otimes_{\kbb} \cdots \otimes_{\kbb} \kbb_{[a_{n+1},b_{n+1}]} = \kbb_{[a',b']} \otimes_{\kbb} \kbb_{[a_{n+1},b_{n+1}]}$, where the action of $\xbf_i \in \kbb[\xbf_1, \dots, \xbf_{n+1}]$ on the $i$th tensor factor.
    Unfolding the definition of $\Ninclexcl$, we get
    \begin{align*}
        \Ninclexcl\left(\kbb_{[a,b]}\right)
            &= \sum_{S \subseteq \{1, \dots, n+1\}} (-1)^{|S|}\,\, \dim(\xbf_S M) \\
            &=
             %\left(
                \sum_{S \subseteq \{1, \dots, n\}} (-1)^{|S|}\,\, \dim\left(\xbf_S M\right)
             %\right)
            \;+\;
            %\left(
                \sum_{T \subseteq \{1, \dots, n\}} (-1)^{|T|+1}\,\, \dim\left(\xbf_{T \cup \{n+1\}} M\right)
            %\right)
            \\
            &= \Ninclexcl(\kbb_{[a',b']}) \cdot \dim\left(\kbb_{[a_{n+1},b_{n+1}]}\right) - \Ninclexcl(\kbb_{[a',b']}) \cdot \dim\left(\kbb_{[a_{n+1} + 1,b_{n+1}]} \right)\\
            &= 1 \cdot (b_{n+1} - a_{n+1}) - 1 \cdot (b_{n+1} - (a_{n+1} + 1)) = 1,
    \end{align*}
    where in the fourth equality we used the inductive hypothesis.
\end{proof}

\subsection{The signed barcode count and the topological Euler characteristic}
The main result in this subsection is 
\cref{theorem:count-is-euler},
which says that, over finite lattices, the count $\chirkdec$ of a spread-representation $\kbb_I$ equals the topological Euler characteristic of $\Delta(I)$, the \emph{order complex} \cite[Section~9.3]{bjorner} of the spread.
We deduce from this an interesting universal property for $\chirkdec$
(\cref{theorem:universal-property-signed-barc}).

It is relevant to note that all spreads in $\Zbb^2$ are contractible, and thus their Euler characteristic is one (\cref{lemma:spread-contractible}), but this is not true already in $\Zbb^3$, where the Euler characteristic can even be negative (\cref{figure:three-parameter-negative}).

\begin{theorem}
    \label{theorem:count-is-euler}
    If $\Pscr$ is a finite lattice and $I \subseteq \Pscr$ is a spread, then $\chirkdec(\kbb_I) = \chi(\Delta(I))$.
\end{theorem}

\begin{proof}
    We will make use of the following easy claims;
    the first one is standard and follows from, e.g., \cite[Corollary~10.12]{bjorner},
    the second one is a consequence of \cref{lemma:any-decomposition-gives-count},
    and the third one is a straightforward check.
    \smallskip

    \noindent \emph{Claim 1.}
    %\label{lemma:rectangle-contractible}
    If $\Pscr$ is a finite poset with a minimum, then $\Delta(\Pscr)$ is contractible.

    \smallskip

    \noindent \emph{Claim 2.}
    %\label{lemma:rectangle-decomposition-gives-count}
    If $\Pscr$ be a finite poset, $M : \Pscr \to \vect$, and $A_1, \dots, A_k \subseteq \Pscr$ and $B_1, \dots, B_\ell \subseteq \Pscr$ are segments such that
    $\Rk(M) = \sum_{i = 1}^k \Rk(\kbb_{A_i}) - \sum_{j = 1}^\ell \Rk(\kbb_{B_j})$,
    then $\chirkdec(M) = k - \ell$.

    \smallskip

    \noindent \emph{Claim 3.}
    %\label{lemma:intersection-of-segments}
    Let $\Pscr$ be a lattice, and let $[x,y],[x',y'] \subseteq \Pscr$ be segments.
    Then $[x,y] \cap [x',y']$ is either empty or the segment $[x\vee x', y\wedge y']$.

    \medskip

    Let $\Ssf$ be the set of maximal segments included in $I$.
    Every chain in $I$ must lie entirely in a segment $R \in \Ssf$, so $C \coloneqq \{\Delta(R)\}_{R \in \Ssf}$ is a cover of $\Delta(I)$ by subcomplexes.
    Let $K$ be the nerve of $C$, so that the $i$th simplices $K$ are given by the subsets $T \subseteq \Ssf$ such that $|T|=i+1$ and $\cap T \neq \emptyset$.

    If $T \subseteq \Ssf$ is such that $\cap T \neq \emptyset$, then $\cap T$ is a segment (by Claim 3), so it has a minimum, and thus $\Delta(\cap T)$ is contractible, by Claim 1.
    This implies that all intersections of elements of $C$ are either empty or contractible, and thus, by the nerve lemma for simplicial complexes \cite[Theorem~10.6]{bjorner}, we get $K \simeq \Delta(I)$, which in turn implies that $\chi(K) = \chi(\Delta(I))$.

    It is thus sufficient to show that $\chirkdec(\kbb_I) = \chi(K)$.
    %\smallskip
    %\noindent\textit{Claim.}
    We claim that
    \begin{equation*}
        \label{equation:inclusion-exclusion-rank}
        \Rk(\kbb_I) = \sum_{i \in \Nbb}\; (-1)^i\; \sum_{T \subseteq \Ssf,\, |T|=i+1,\, \cap T \neq \emptyset}\; \Rk(\kbb_{\cap T}).
    \end{equation*}
    From this equation and Claim 2, it follows that
    \[
        \chirkdec(\kbb_I) =
        \big|\big\{ T \subseteq \Ssf : |T| \text{ odd},\; \cap T \neq \emptyset\big\}\big| -
        \big|\big\{ T \subseteq \Ssf : |T| \text{ even},\; \cap T \neq \emptyset\big\}\big|
        = \chi(K),
    \]
    so let us prove the equation.
    We do this by evaluating the left- and right-hand sides on each $a \leq b \in \Pscr$.
    If either $a$ or $b$ is not in $I$, then both sides are zero, so let us assume that $a, b \in I$.
    In this case, the left-hand side is one, so we need to show that the right-hand side is also one.

    Let $\Ssf_{a,b} \subseteq \Ssf$ be the set of segments in $\Ssf$ that contain both $a$ and $b$.
    All intersections of sets in $\Ssf_{a,b}$ must also contain $a$ and $b$, so we need to prove that
    \[
        \big|\big\{ T \subseteq \Ssf_{a,b} : |T| \text{ odd}\big\}\big| - \big|\big\{ T \subseteq \Ssf_{a,b} : |T| \text{ even}\big\}\big| = 1,
    \]
    which is true for any finite set $X$ in the place of $\Ssf_{a,b}$, since it follows from the fact that the alternating sum of binomial coefficients is $0$.
\end{proof}

\begin{corollary}
    \label{lemma:signed-barcode-hooks}
    If $\Pcal$ is a finite lattice, then $\chirkdec = \chirkexact\; :\; \rep(\Pcal) \to \Zbb$.
\end{corollary}
\begin{proof}
    It is clear that $\dimhom(\kbb_{\langle a, \infty \langle}, M) = \Rk(M)(a,a)$
    and we have
    $\Rk(M)(a,b) = \dimhom(\kbb_{\langle a, \infty \langle}, M) -  \dimhom(\kbb_{\langle a, b \langle},M)$ by rank-nullity (see \cite[Lemma~4.4(2)]{botnan-oppermann-oudot}).
    Thus, $\Rk$ is equivalent to $\dimhom_{\hooks}$.
    By \cref{lemma:basic-results-counts-basis}, it is thus enough to show that $\chirkdec(\kbb_{\langle a, b \langle}) = 1$ for every hook $\langle a, b \langle$ of $\Pcal$, which is true by \cref{theorem:count-is-euler}
    since $\Delta \langle a, b \langle$ is contractible, since $\langle a, b \langle$ has a minimum.
    %    , and thus is it enough to prove that $\Delta(\langle a, b \langle)$ is contractible.
    %
    %    Let $m_1, \dots, m_k$ be the set of maximal elements of $\langle a, b \langle$.
    %    Since very chain in $\langle a, b \langle$ must lie in $[a,m_i]$ for some $1 \leq i \leq k$, the family of simplicial complexes $\{\Delta([a,m_i])\}_{1 \leq i \leq k}$ forms a cover of of $\Delta(\langle a, b \langle)$.
    %    The intersection of all of these simplicial complexes includes the zero-simplex corresponding to $a$, and thus it is non-empty.
    %    Thus, the nerve of the cover is contractible, and $\Delta(\langle a, b \langle)$ is contractible as well, by the nerve lemma \cite[Theorem~10.6]{bjorner}.
\end{proof}

\begin{corollary}
    \label{corollary:two-parameter-normalized}
    If $I \subseteq \Zbb^2$ is a finite spread, then $\Ncaltwo(I) = 1$.
\end{corollary}
\begin{proof}
    By translating $I$ if necessary, we can assume that $I$ belongs to a finite grid $I \subseteq \Gcal^2 \subseteq \Zbb^2$.
    Over this finite, two-dimensional grid we have $\Ncaltwo = \chirkdec$, by \cref{proposition:signed-barcode-equals-inclusion-exclusion}, so the result follows from \cref{lemma:spread-contractible}.
\end{proof}

We conclude by deriving universal properties for $\chirkdec$ and $\Ncaltwo$.

%In order to state these, we require the following definition.
%\begin{definition}
%    \label{definition:generated-invariant}
%    Let $\alpha : \Acal \to G$ be an additive invariant, and let $S$ be a collection of objects of~$\Acal$.
%    The invariant $\alpha$ is \emph{$S$-generated} if for every $M \in \Acal$ there exist $M_1, \dots, M_k \in S$ and integers $c_1, \dots, c_k \in \Zbb$, such that $\alpha(M) = \sum_{1 \leq i \leq k} c_i \cdot \alpha(M_i)$.
%\end{definition}
%
%For example, if $\Bcal$ is a basis for an additive invariant $\alpha$, then $\alpha$ is $\Bcal$-generated.
%
%\todo{can we make this better by using $\alpha$ and a basis instead of an abstract $\Ncal$?}
%\todo{If we change it, make sure that when we use it properly (eg when we say that the count is 1 on spreads)}

\begin{corollary}
    \label{theorem:universal-property-signed-barc}
    Let $\Pscr$ be a finite lattice, let
    $\alpha$ be an additive invariant on $\rep(\Pscr)$, and let $\Bcal$ be a basis for $\alpha$.
    Assume that:
    \begin{enumerate}
    \item The basis $\Bcal$ only contains spread representations.
    \item As additive invariants, we have $\alpha \succcurlyeq \chirkdec$.
    \end{enumerate}
    Then $\Ncal^{(\alpha, \Bcal, f)} = \chirkdec$, where $f : \Bcal \to \Zbb$ is given by $f(\kbb_I) = \chi(\Delta(I))$.
\end{corollary}
\begin{proof}
    We use \cref{lemma:basic-results-counts-basis} for the count $\Ncal^{(\alpha, \Bcal, f )}$.
    Condition (1) follows from \cref{theorem:count-is-euler} and the definition of $f$, while condition (2) is just assumption (2).
\end{proof}

\begin{corollary}
    \label{theorem:universal-property-count}
    Let $\alpha$ be an additive invariant on $\rep(\Gcal^2)$, and let $\Bcal$ be a basis for $\alpha$.
    Assume that:
    \begin{enumerate}
        \item The basis $\Bcal$ only contains spread representations.
        \item As additive invariants, we have $\alpha \succcurlyeq \Ncaltwo$.
    \end{enumerate}
    Then $\Ncal^{(\alpha, \Bcal)} = \Ncaltwo$.
\end{corollary}
\begin{proof}
    On $\fintwoparam$, we have $\Ncaltwo = \chirkdec$ by \cref{proposition:signed-barcode-equals-inclusion-exclusion},
    so the result is a consequence of \cref{theorem:universal-property-signed-barc,corollary:two-parameter-normalized}.
\end{proof}

%\begin{proposition}
%    \label{proposition:signed-barcode-normalized}
%    If $I \subseteq \Gcal^2$ is a spread, then $\chirkdec(\kbb_I) = 1$.
%\end{proposition}
%\begin{proof}
%    This is a direct consequence of \cref{theorem:count-is-euler,lemma:spread-contractible}.
%\end{proof}

\subsection{Proof of \cref{corollary:our-count-and-other-counts}}

%\begin{corollary}
%    \label{theorem:universal-property-count}
%    The invariant $\Ncaltwo$ is the unique additive invariant $\Ncal : \fintwoparam \to \Zbb$ satisfying:
%    \begin{enumerate}
%        \item $\Ncal(\kbb_I) = 1$ for every spread $I \subseteq \Gcal^2$.
%        \item There exists a spread-generated invariant $\alpha$ on $\fintwoparam$, with $\alpha \succcurlyeq \Ncal$ and $\alpha \succcurlyeq \Ncaltwo$.\qed
%    \end{enumerate}
%\end{corollary}
%
%If $[x,y]$ is a segment of a poset $\Pscr$, and $M : \Pscr \to \vect$,
%then $\Rk_{[x,y]}(M) = \mathsf{rank}(M_x \to M_y)$, so the rank invariant simply records the rank of all the structure morphisms of $M$.

\ourcountandothercounts*
\begin{proof}
    We know that $\Ncaltwo = \chirkdec$, by \cref{proposition:signed-barcode-equals-inclusion-exclusion},
    and $\chirkdec = \chirkexact$, by \cref{lemma:signed-barcode-hooks}.
    %We use \cref{theorem:universal-property-count} to prove that all counts equal $\Ncaltwo$ .
    %\chirkdec(M) = \chirkexact(M) = |\births(M)| = |\deaths(M)|.
    To prove that $\chirkdec$ equals the other counts, we use the universal property of $\Ncaltwo$ (\cref{theorem:universal-property-count}).

    %\medskip
    %\noindent\textit{The case of $\chirkdec$.}
    %%By \cref{lemma:signed-barcode-hooks}, it is enough to prove that $\chirkdec$ satisfies the two conditions of \cref{theorem:universal-property-count}.
    %Condition (1) is satisfied thanks to \cref{proposition:signed-barcode-normalized}.
    %For condition (2), we take $\alpha = \Rk$.
    %The rank invariant $\Rk$ is generated by spreads since segments form a basis for it (\cref{remark:basis-for-known-invariants}(3)), it satisfies $\Rk \succcurlyeq \chirkdec$ by \cref{definition:known-counts}, and it satisfies $\Rk \succcurlyeq \Ncaltwo$ by \cref{lemma:rank-determines-count}.

    \medskip
    \noindent\textit{The case of $\chigpd$.}
    Here $\alpha = \Rk_\spreads$, and the basis $\Bcal$ consists of all spread representations, so condition (1) of \cref{theorem:universal-property-count} is satisfied.
    Condition (2) follows from the fact that $\Rk_\spreads \succcurlyeq \Rk \succcurlyeq \chirkdec \succcurlyeq \Ncal^2$.
    %We have $\chigpd = \chirkexact$ by \cref{lemma:signed-barcode-hooks}, so we prove $\chigpd = \Ncaltwo$.
    %Condition $(1)$ of \cref{theorem:universal-property-count} follows from \cref{lemma:basic-results-counts-basis}(2).
    %and $\chirkdec(\kbb_I) = \chi(\Delta(I)) = 1$, by \cref{lemma:spread-contractible}.
    %To check condition $(2)$, we let $\alpha = \Rk_\spreads$.
    %To check condition (2), note that 
    %Then $\alpha \succcurlyeq \chigpd$ by \cref{lemma:basic-results-counts-basis}(1), and $\alpha \succcurlyeq \Rk$.
    %The invariant $\alpha$ is also spread-generated, since spreads form a basis for $\Rk_\spreads$, by \cref{remark:basis-for-known-invariants}(1).

    \medskip
    \noindent\textit{The case of $\chispreadexact$.}
    Here we can take $\alpha = \dimhom_\spreads$ or $\alpha = \Ksf_\spreads$ and the basis $\Bcal$ to consist of all spreads.
    Since we already know that $\Rk \succcurlyeq \Ncaltwo$, it is sufficient to check that $\dimhom_\spreads \succcurlyeq \Rk$, which is standard \cite[Lemma~4.4]{botnan-oppermann-oudot}.

    \medskip
    \noindent\textit{The case of $\Ncal^{\mathrm{int.dec.repl.}}$.}
    Taking $\alpha = \underline{d}^*$, the proof of this case is analogous to that of the previous two cases, using \cref{proposition:compressed-multiplicities-properties} to satisfy the conditions.

    \medskip
    \noindent\textit{The cases of $\Ncal^\births$ and $\Ncal^\deaths$.}
    We do the case of $\Ncal^\births$; the case of $\Ncal^\deaths$ is similar.
    Here $\alpha = \births$, and the basis $\Bcal$ consists of the spread curves of $\Zbb^2$ that are contained in $\Gcal^2$.
    Condition (1) then follows from the fact that spread curves are in particular spreads, while condition (2) follows from the fact that $\deaths \succcurlyeq \Ncaltwo$, thanks to \cref{equation:count-same-as-birth-death}.
    %To check $(2)$, we take $\alpha = \births$.
    %We have $\alpha \succcurlyeq \Ncal^\births$, by \cref{lemma:basic-results-counts-basis}(1), and $ \alpha \succcurlyeq \Ncaltwo$
    %since, for every $M \in \fintwoparam$ we have that $\Ncaltwo(M) = \Ncaltwo(\cokerxy{M})$ by \cref{equation:count-same-as-birth-death}, and that $\births(M)$ determines the entire representation $\cokerxy{M}$ (up to isomorphism).
    %The invariant $\births$ is spread-generated since the spread curves form a basis for it (\cref{lemma:spread-curves-basis-for-birth-death}).
\end{proof}

\subsection{Slice-monotonicity of the two-parameter count}

If $f : \Pscr \to \Qscr$ a monotonic function between posets, and $M : \Qscr \to \Vect$, we let $f^*(M) : \Pscr \to \Vect$ be defined by pre-composition with $f$, that is $f^*(M) \coloneqq M \circ f$.
The following is standard from the representation theory of quivers of type $A$.

\begin{lemma}
    \label{lemma:indecomposable-one-parameter}
    %If $\Lcal$ is a finite linear order, then $M : \Lcal \to \vect$ is indecomposable if and only if it is isomorphic to $\kbb_{[a,b]}$ for $a \leq b \in \Lcal$.
    A representation $M \in \repf(\Zbb)$ is indecomposable if and only if it is isomorphic to $\kbb_{[a,b)}$ for $a <b \in \Zbb$.
    \qed
\end{lemma}

Recall that the simple representation at $i \in \Zbb$ is $\kbb_i \coloneqq \kbb_{[i,i+1)}$.

\begin{lemma}
    \label{lemma:one-parameter-count-as-hom}
    If $M \in \repf(\Zbb)$, then $\Ncalone(M) = \dim \hom\left(M,\bigoplus_{i \in \Zbb} \kbb_i\right) = \dim \Hom\left(\bigoplus_{i \in \Zbb} \kbb_i, M\right)$.
\end{lemma}
\begin{proof}
    By \cref{lemma:indecomposable-one-parameter} and additivity, it is sufficient to prove this for $M = \kbb_{[a,b)}$, with $a < b \in \Zbb$.
    Then $\hom\left(\kbb_{[a,b)},\bigoplus_{i \in \Zbb} \kbb_i\right) \cong \hom(\kbb_{[a,b)}, \kbb_a) \cong \kbb$, and
    $\hom\left(\bigoplus_{i \in \Zbb} \kbb_i, \kbb_{[a,b)}\right) \cong \hom(\kbb_{b-1}, \kbb_{[a,b)}) \cong \kbb$, and both have dimension $1$, as required.
\end{proof}

The following is standard, see, e.g., \cite{bauer-lesnick}.

\begin{lemma}
    \label{lemma:one-parameter-monotonicity-inclusion}
    Let $M, M' \in \repf(\Zbb)$.
    If there exists a monomorphism $M \to M'$ or an epimorphism $M' \to M$, then $\Ncalone(M) \leq \Ncalone(M')$.
    %Let $\Lcal$ be a finite linearly ordered set, and let $M, M' : \Lcal \to \vect$.
    %If there exists monomorphism $M \to M'$ or an epimorphism $M' \to M$, then $\Ncalone(M) \leq \Ncalone(M')$.
\end{lemma}
\begin{proof}
    By duality (\cref{section:duality}), it is enough to do the case of a monomorphism, so let $f : M \to M'$ be a monomorphism.
    By \cref{lemma:one-parameter-count-as-hom}, we need to prove that
    $\dim \hom\left(\bigoplus_{i \in \Zbb} \kbb_i, M\right) \leq \dim \hom\left(\bigoplus_{i \in \Zbb} \kbb_i, M'\right)$, which follows from the fact that the morphism
    $f_* : \hom\left(\bigoplus_{i \in \Zbb} \kbb_i, M\right) \to \dim \hom\left(\bigoplus_{i \in \Zbb} \kbb_i, M'\right)$ induced by $f$ is a monomorphism.
\end{proof}

%To motivate \cref{definition:counting-function-discrete} briefly, first recall that the \emph{one-parameter counting function} $\Ncalone : \rep(\Gcal) \to \Zbb$ is given by counting the number of indecomposable summands $\Ncalone(M) \coloneqq |\dec(M)|$.
%Since every indecomposable in $\rep(\Gcal)$ consits of a spread, it follows that $\Ncalone(M) = \dim M/\rsf M$, where $\rsf \coloneqq \sum_{p \in \Gcal \text{ s.t.~} p+1 \in \Gcal} [p, p+1]$.

%The count $\Ncaltwo$ has the following remarkable properties.

For the rest of this section, we identify representations of $\Zbb$ with singly-graded $\kbb[\zbf]$-modules,
%\luis{in the following, use the convention for graded algebras?}
%\luis{make sure that the shifts work as expected}
%If $\Lcal = [n] = \{0, \dots, n-1\}$ is a finite linear order for $n \geq 1 \in \Nbb$, let $\zbf \in \kbb \Lcal$ be defined as $\zbf \coloneqq \sum_{0 \leq i < n} [i,i+1]$.
and if $M \in \rep(\Zbb)$, we consider the following exact sequence, which we use to define $\kerelem{M}{\zbf}$ and $\cokerelem{M}{\zbf}$:
%, as in \cref{notation:ker-coker-notation}:
\[
    0 \to \kerelem{M}{\zbf} \to M[-1] \xrightarrow{\,\,\zbf\,\,} M \to \cokerelem{M}{\zbf} \to 0\,.
\]

\begin{lemma}
    \label{lemma:one-parameter-count-as-dimension-birth-death}
    If $M \in \repf(\Zbb)$, then $\Ncalone(M) = \dim \cokerelem{M}{\zbf} = \dim \kerelem{M}{\zbf}$.
\end{lemma}
\begin{proof}
    By additivity and \cref{lemma:indecomposable-one-parameter}, it is sufficient to prove this for the case of $M = \kbb_{[a,b)}$, where we have $\cokerelem{M}{\zbf} \cong \kbb_a$ and $\kerelem{M}{\zbf} \cong \kbb_b$, both of dimension $1$.
\end{proof}

\begin{definition}
    A monotonic function between finite posets $f : \Pscr \to \Qscr$ is \emph{contiguous} if whenever $x \leq y \in \Pscr$ is a cover relation, we have that $f(x) \leq f(y) \in \Qscr$ is a cover relation as well.
\end{definition}

\begin{lemma}
    \label{lemma:contiguous-death-relation}
    Let $\ell : \Zbb \to \Zbb^2$ be a contiguous monotonic function.
    Let $M \in \repf(\Zbb^2)$, let $i \in \Zbb$ and let $x \in M_{\ell(i)} = \ell^*(M)_i$.
    If $\zbf \cdot x = 0 \in \ell^*(M)$, then $\xbf\ybf \cdot x = 0 \in M$.
    It follows that there is a monomorphism of representations of $\Zbb$:
    \[
        \kerelem{\ell^*(M)}{\zbf} \hookrightarrow \ell^*(\kerxy{M}).
    \]
    %    Let $\Lcal = [n]$, and let $\ell : \Lcal \to \Gcal^2$ be a contiguous monotonic function mapping the maximum of $\Lcal$ to the maximum of $\Gcal^2$.
    %    Let $M : \Gcal^2 \to \vect$, let $i \in \Lcal$ and let $x \in M_{\ell(i)} = \ell^*(M)_i$.
    %    If $\zbf(x) = 0 \in \ell^*(M)$, then $\xbf\ybf(x) = 0 \in M$.
    %    It follows that there is a monomorphism of modules over $\Lcal$:
    %    \[
    %        \ell^*(M)[\zbf] \hookrightarrow \ell^*(M[\xbf\ybf]).
    %    \]
\end{lemma}
\begin{proof}
    We must have $\zbf \cdot x = \phi^{\ell^*(\kerxy{M})}_{i, i+1}(x) = 0$.
    The cover relations in $\Zbb^2$ are those of the form $(a,b) \leq (a+1, b)$ and $(a,b) \leq (a, b+1)$.
    So $\ell(i+1)$ is either $\ell(i) + (1,0)$ or $\ell(i) + (0,1)$; in either case, $\ell(i+1) \leq \ell(i) + (1,1)$.
    And thus, $\xbf\ybf \cdot x = \phi^{\kerxy{M}}_{\ell(i), \ell(i) + (1,1)}(x) = 0$, as required.
\end{proof}

\begin{lemma}
    \label{lemma:preimage-of-spread-in-linear-order}
    %Let $\Lcal$ be a finite linearly ordered set, and let $\ell : \Lcal \to \Pscr$ be a monotonic function into a poset $\Pscr$.
    Let $\Zbb \to \Pscr$ be a monotonic function into a poset $\Pscr$.
    If $I \subseteq \Pscr$ is a spread, then $\Ncalone(\ell^*(\kbb_I)) \leq 1$.
\end{lemma}
\begin{proof}
    We claim that $\ell^{-1}(I) \subseteq \Zbb$ must be either empty or an interval of $\Zbb$.
    The result follows from this, since $\Ncalone$ of an interval is $1$, and $\Ncalone$ of the zero representation is $0$.

    So let us prove that, if $\ell^{-1}(I)$ is not empty, then it is an interval.
    Let $a \leq b \leq c \in \Zbb$ and assume that $a,c \in \ell^{-1}(I)$.
    Then $\ell(a) \leq \ell(b) \leq \ell(c) \in \Pscr$ and $\ell(a), \ell(c) \in I$, so $\ell(b) \in I$ and $b \in \ell^{-1}(I)$, as required.
\end{proof}

\begin{lemma}
    \label{lemma:one-parameter-monotonicity-restriction}
    Let $f : \Zbb \hookrightarrow \Zbb$ be an injective monotonic function, and let $M \in \repf(\Zbb)$.
    We have $\Ncalone(f^*(M)) \leq \Ncalone(M)$.
    %    Let $\Lcal$ and $\Lcal'$ be finite linearly ordered sets, let $f : \Lcal \to \Lcal'$ be a monotonic function, and let $M : \Lcal' \to \vect$.
    %    \begin{enumerate}
    %        \item If $f$ is injective, then $\Ncalone(f^*(M)) \leq \Ncalone(M)$.
    %        \item If $f$ is surjective, then $\Ncalone(f^*(M)) = \Ncalone(M)$.
    %    \end{enumerate}
\end{lemma}
\begin{proof}
    Recall that $\Ncalone$ simply counts the number of indecomposable summands, so it is sufficient to prove the result for $M$ indecomposable.
    So, using \cref{lemma:indecomposable-one-parameter}, let $M = \kbb_{[a,b)}$ for $a< b \in \Zbb$.
    Now, if $S \coloneqq f^{-1}([a,b)) = \emptyset$, then $f^*(M) = 0$, and otherwise $M \cong \kbb_{[\min S, \max S]}$, by the same argument used in the proof of \cref{lemma:preimage-of-spread-in-linear-order}.
    The result follows.
\end{proof}

\slicemonotonicity*
\begin{proof}
    %If $\Lcal$ does not map the maximum of $\Lcal$ to the maximum of $\Gcal^2$, then it factors as $\ell : \Lcal \xrightarrow{\iota} \Lcal' \xrightarrow{\ell'} \Gcal^2$, with $\Lcal'$ a finite linear order, $\iota$ injective, and $\ell'$ mapping the maximum of $\Lcal'$ to that of $\Gcal^2$.
    %If $\Lcal$ is not injective, then it factors as $\ell : \Lcal \xrightarrow{f} \Lcal' \xrightarrow{\ell'} \Gcal^2$, with $\Lcal'$ a finite linear order, $f$ surjective, and $\ell'$ injective.
    %Thus, by \cref{lemma:one-parameter-monotonicity-restriction}(2), we may assume that $\ell$ is injective.
    If $\ell$ is not contiguous, then it factors as $\ell : \Zbb \xrightarrow{\iota} \Zbb \xrightarrow{\ell'} \Zbb^2$, with $\ell'$ contiguous and injective, and $\iota$ injective.
    Thus, by \cref{lemma:one-parameter-monotonicity-restriction}, we may assume that $\ell$ is contiguous.
    %We thus assume that $\ell$ is injective and contiguous, and that it maps the maximum of $\Lcal$ to that of $\Gcal^2$.
    %Without loss of generality, we also assume that $\Lcal = [n]$.

    On the one hand, using \cref{lemma:one-parameter-count-as-dimension-birth-death}, we see that
    $\Ncalone(\ell^*(M)) = \Ncalone( \kerelem{\ell^*(M)}{\zbf} )$, and using \cref{lemma:contiguous-death-relation,lemma:one-parameter-monotonicity-inclusion}, we have $\Ncalone( \kerelem{\ell^*(M)}{\zbf} ) \leq \Ncalone(\ell^*(\kerxy{M}))$, so that
    $\Ncalone(\ell^*(M)) \leq \Ncalone(\ell^*(\kerxy{M}))$.
    On the other hand, we have $\Ncaltwo(\kerxy{M}) = \Ncaltwo(M)$, by \cref{equation:count-same-as-birth-death}.

    Since $\kerxy{M}$ is ephemeral, by additivity and \cref{theorem:decomposition-theorem-ephemeral}, it is sufficient to prove $\Ncalone(\ell^*(M)) \leq \Ncaltwo(M)$ for $M$ a spread curve representation.
    In this case, the right-hand side is one, by \cref{corollary:two-parameter-normalized}, and the left-hand since is at most one, by \cref{lemma:preimage-of-spread-in-linear-order}, concluding the proof.
\end{proof}

\mainpropertiescount*
\begin{proof}
    We start with (1).
    By taking any monotonic map $\ell : \Zbb \hookrightarrow \Zbb^2$ with $\ell(0) = i$, we get $\Ncaltwo(M) \geq \Ncalone(\ell^*(M)) \geq \dim(\ell^*(M)_0) = \dim M_i$ by \cref{theorem:slice-monotonicity} and the fact that the number of summands bounds the pointwise dimension for one-parameter representations.
    Statement (2) follows from (1).

    Since this is required for statements (3) and (4), let us prove: ($*$) we have $\Ncaltwo(M) \geq \Ncaldec(M)$.
    Using additivity, we can assume that $M$ is indecomposable, and in this case we have $\Ncaldec(M) = 1$ and $\Ncaltwo(M) \geq 1$, by \cref{proposition:positivity} and (2).

    Let us now prove (3).
    If $M$ is a spread representation, then $\Ncaltwo(M) = 1$, by \cref{corollary:two-parameter-normalized}.
    To prove the forward direction, assume that $\Ncaltwo(M) = 1$.
    Then, by ($\ast$), the representation $M$ must be indecomposable, and by (1), we must have $\dim M_i \in \{0,1\}$, for every $i \in \Zbb^2$.
    The result now follows from \cref{theorem:indecomposable-think-is-spread}, a restatement of \cite[Theorem~24]{asashiba-buchet-escolar-nakashima-yoshiwaki}, which says that an indecomposable two-parameter representation of pointwise dimension at most one must be a spread representation.

    We conclude by proving (4).
    The first statement is ($\ast$).
    The converse direction of the equivalence follows from additivity and (4).
    We conclude by proving the contrapositive of the forward direction.
    If $M$ is not spread-decomposable, then it admits at least one indecomposable summand in which the inequality of (2) is strict, and the rest of the indecomposable summands still need to satisfy~(2).
    Then, we have $\Ncaltwo(M) > \Ncaldec(M)$, so that $\Ncaltwo(M) \neq \Ncaldec(M)$, as required.
\end{proof}

\section{End-curves and Betti tables}
\label{section:end-curves-betti}

We need the following standard characterization of the indecomposable representation of the commutative square; see, e.g., \cite{escolar-hiraoka} for a proof.

\begin{lemma}
    \label{lemma:indecomposables-over-square}
    Let $N \in \rep(\{0 < 1\}^2)$ be indecomposable.
    Then $N$ is isomorphic to one of the following representations:
    \[
        \resizebox{\linewidth}{!}{%
            \begin{tikzpicture}
                \matrix (m) [matrix of math nodes,ampersand replacement=\&,row sep=0em,column sep=0em,minimum width=0em,nodes={text height=1.75ex,text depth=0.25ex}]
                {
                    0 \&  \& 0                      \\
                    \& {\scriptstyle (a)} \& \\
                    \kbb \&  \& 0 \\};
                \path[line width=0.75pt, ->]
                (m-1-1) edge (m-1-3)
                (m-3-1) edge (m-1-1)
                (m-3-1) edge (m-3-3)
                (m-3-3) edge (m-1-3)
                ;
            \end{tikzpicture}
            \begin{tikzpicture}
                \matrix (m) [matrix of math nodes,ampersand replacement=\&,row sep=0em,column sep=0em,minimum width=0em,nodes={text height=1.75ex,text depth=0.25ex}]
                {
                    0 \&  \& 0                      \\
                    \& {\scriptstyle (b)} \& \\
                    0 \&  \& \kbb \\};
                \path[line width=0.75pt, ->]
                (m-1-1) edge (m-1-3)
                (m-3-1) edge (m-1-1)
                (m-3-1) edge (m-3-3)
                (m-3-3) edge (m-1-3)
                ;
            \end{tikzpicture}
            \begin{tikzpicture}
                \matrix (m) [matrix of math nodes,ampersand replacement=\&,row sep=0em,column sep=0em,minimum width=0em,nodes={text height=1.75ex,text depth=0.25ex}]
                {
                    \kbb \&  \& 0                   \\
                    \& {\scriptstyle (c)} \& \\
                    0 \&  \& 0 \\};
                \path[line width=0.75pt, ->]
                (m-1-1) edge (m-1-3)
                (m-3-1) edge (m-1-1)
                (m-3-1) edge (m-3-3)
                (m-3-3) edge (m-1-3)
                ;
            \end{tikzpicture}
            \begin{tikzpicture}
                \matrix (m) [matrix of math nodes,ampersand replacement=\&,row sep=0em,column sep=0em,minimum width=0em,nodes={text height=1.75ex,text depth=0.25ex}]
                {
                    0 \&  \& \kbb                   \\
                    \& {\scriptstyle (d)} \& \\
                    0 \&  \& 0 \\};
                \path[line width=0.75pt, ->]
                (m-1-1) edge (m-1-3)
                (m-3-1) edge (m-1-1)
                (m-3-1) edge (m-3-3)
                (m-3-3) edge (m-1-3)
                ;
            \end{tikzpicture}
            \begin{tikzpicture}
                \matrix (m) [matrix of math nodes,ampersand replacement=\&,row sep=0em,column sep=0em,minimum width=0em,nodes={text height=1.75ex,text depth=0.25ex}]
                {
                    \kbb \&  \& 0                   \\
                    \& {\scriptstyle (e)} \& \\
                    \kbb \&  \& 0 \\};
                \path[line width=0.75pt, ->]
                (m-1-1) edge (m-1-3)
                (m-3-1) edge [-,double equal sign distance] (m-1-1)
                (m-3-1) edge (m-3-3)
                (m-3-3) edge (m-1-3)
                ;
            \end{tikzpicture}
            \begin{tikzpicture}
                \matrix (m) [matrix of math nodes,ampersand replacement=\&,row sep=0em,column sep=0em,minimum width=0em,nodes={text height=1.75ex,text depth=0.25ex}]
                {
                    0 \&  \& 0                      \\
                    \& {\scriptstyle (f)} \& \\
                    \kbb \&  \& \kbb \\};
                \path[line width=0.75pt, ->]
                (m-1-1) edge (m-1-3)
                (m-3-1) edge (m-1-1)
                (m-3-1) edge [-,double equal sign distance] (m-3-3)
                (m-3-3) edge (m-1-3)
                ;
            \end{tikzpicture}
            \begin{tikzpicture}
                \matrix (m) [matrix of math nodes,ampersand replacement=\&,row sep=0em,column sep=0em,minimum width=0em,nodes={text height=1.75ex,text depth=0.25ex}]
                {
                    \kbb \&  \& \kbb                \\
                    \& {\scriptstyle (g)} \& \\
                    0 \&  \& 0 \\};
                \path[line width=0.75pt, ->]
                (m-1-1) edge [-,double equal sign distance] (m-1-3)
                (m-3-1) edge (m-1-1)
                (m-3-1) edge (m-3-3)
                (m-3-3) edge (m-1-3)
                ;
            \end{tikzpicture}
            \begin{tikzpicture}
                \matrix (m) [matrix of math nodes,ampersand replacement=\&,row sep=0em,column sep=0em,minimum width=0em,nodes={text height=1.75ex,text depth=0.25ex}]
                {
                    0 \&  \& \kbb                   \\
                    \& {\scriptstyle (h)} \& \\
                    0 \&  \& \kbb \\};
                \path[line width=0.75pt, ->]
                (m-1-1) edge (m-1-3)
                (m-3-1) edge (m-1-1)
                (m-3-1) edge (m-3-3)
                (m-3-3) edge [-,double equal sign distance] (m-1-3)
                ;
            \end{tikzpicture}
            \begin{tikzpicture}
                \matrix (m) [matrix of math nodes,ampersand replacement=\&,row sep=0em,column sep=0em,minimum width=0em,nodes={text height=1.75ex,text depth=0.25ex}]
                {
                    \kbb \&  \& 0                   \\
                    \& {\scriptstyle (i)} \& \\
                    \kbb \&  \& \kbb \\};
                \path[line width=0.75pt, ->]
                (m-1-1) edge (m-1-3)
                (m-3-1) edge [-,double equal sign distance] (m-1-1)
                (m-3-1) edge [-,double equal sign distance] (m-3-3)
                (m-3-3) edge (m-1-3)
                ;
            \end{tikzpicture}
            \begin{tikzpicture}
                \matrix (m) [matrix of math nodes,ampersand replacement=\&,row sep=0em,column sep=0em,minimum width=0em,nodes={text height=1.75ex,text depth=0.25ex}]
                {
                    \kbb \&  \& \kbb                \\
                    \& {\scriptstyle (j)} \& \\
                    0 \&  \& \kbb \\};
                \path[line width=0.75pt, ->]
                (m-1-1) edge [-,double equal sign distance] (m-1-3)
                (m-3-1) edge (m-1-1)
                (m-3-1) edge (m-3-3)
                (m-3-3) edge [-,double equal sign distance] (m-1-3)
                ;
            \end{tikzpicture}
            \begin{tikzpicture}
                \matrix (m) [matrix of math nodes,ampersand replacement=\&,row sep=0em,column sep=0em,minimum width=0em,nodes={text height=1.75ex,text depth=0.25ex}]
                {
                    \kbb \&  \& \kbb                \\
                    \& {\scriptstyle (k)} \& \\
                    \kbb \&  \& \kbb \\};
                \path[line width=0.75pt, ->]
                (m-1-1) edge [-,double equal sign distance] (m-1-3)
                (m-3-1) edge [-,double equal sign distance] (m-1-1)
                (m-3-1) edge [-,double equal sign distance] (m-3-3)
                (m-3-3) edge [-,double equal sign distance] (m-1-3)
                ;
            \end{tikzpicture}}
    \]
\end{lemma}

Given $\ell \in \Zbb^2$ let $\sq(\ell) \subseteq \Zbb^2$ be the full four-element subposet spanned by $\{\ell, \ell-(0,1), \ell-(1,0), \ell-(1,1)\}$.
Given $M \in \rep(\Zbb^2)$, we can decompose the restriction $M|_{\sq(\ell)} \in \rep(\sq(\ell))$ as a direct sum of representations of the form of those in \cref{lemma:indecomposables-over-square}.
In that case, if $n \in \{a, b, \dots, j, k\}$, let $M^n_\ell$ be the number of indecomposables of type $(n)$.

\begin{lemma}
    \label{lemma:betti-tables-and-decomposition-on-square}
    Let $M \in \rep(\Zbb^2)$, and let $\ell \in \Zbb^2$.
    We have the following characterization of the Betti tables of $M$
    \[
        \beta_0^M(\ell) = M^d_\ell\, , \;\;\;\;\;\;\;\;\;
        \beta_1^M(\ell) = M^i_\ell + M^j_\ell + M^b_\ell + M^c_\ell\, , \;\;\;\;\;\;\;\;\;
        \beta_2^M(\ell) = M^a_\ell\, ,
    \]
    and the following set of equalities
    \begin{align*}
        M^a_\ell & = \dim \left\{z \in M_{\ell-(1,1)} : \xbf \cdot z = \ybf \cdot z = 0 \right\} \\
        M^b_\ell & = \dim \im\left(\,
        \left\{z \in M_{\ell-(0,1)} : \ybf \cdot z = 0\right\} \hookrightarrow M_{\ell-(0,1)} \twoheadrightarrow M_{\ell-(0,1)}/\xbf M_{\ell-(1,1)}\,
        \right)                                                                                  \\
        M^c_\ell & = \dim \im\left(\,
        \left\{z \in M_{\ell-(1,0)} : \xbf \cdot z = 0\right\} \hookrightarrow M_{\ell-(1,0)} \twoheadrightarrow M_{\ell-(1,0)}/\ybf M_{\ell-(1,1)}\,
        \right)                                                                                  \\
        M^d_\ell & = \dim M_\ell \, /\left(\xbf M_{\ell-(1,0)} + \ybf M_{\ell-(0,1)}\right)    \\
        M^i_\ell & = \dim \ker\left( M_{\ell-(1,1)} \xrightarrow{\xbf\ybf} M_{\ell} \right)\, /
        \left(
        \ker \left( M_{\ell-(1,1)} \xrightarrow{\xbf} M_{\ell - (0,1)} \right)
        +
        \ker \left( M_{\ell-(1,1)} \xrightarrow{\ybf} M_{\ell - (1,0)} \right)
        \right)
        \\
        M^j_\ell & = \dim
        \left\{(w,z) \in
        \left( M_{\ell-(1,0)}/ \ybf M_{\ell-(1,1)} \right) \oplus
        \left( M_{\ell-(0,1)}/ \xbf M_{\ell-(1,1)} \right) : \xbf \cdot w = \ybf \cdot z \in M_{\ell}/\xbf\ybf M_{\ell-(1,1)}
        \right\}
    \end{align*}
\end{lemma}
\begin{proof}
    The Koszul complex for $M$ at $\ell$, which is what we use to define $\beta_\bullet^M(\ell)$ (\cref{section:betti-tables}) only depends on $M|_{\sq(\ell)}$, that is, we have $\beta_\bullet^M(\ell) = \beta_\bullet^{M|_{\sq(\ell)}}(\ell)$.
    The second set of equalities also only depends on $M|_{\sq(\ell)}$.
    Thus, it is enough to prove the result for $M \in \rep(\{0 < 1\}^2)$, $\ell = (1,1)$, and $M$ one of the indecomposables of \cref{lemma:indecomposables-over-square}.
    This is then a routine check.
    %So it is enough to prove the result for $M \in \rep(\{0,1\}^2)$, $i = (1,1)$, and $M$ one of the indecomposables of \cref{lemma:indecomposables-over-square}, which is straightforward.
\end{proof}

%\begin{corollary}
%    Let $M \in \rep(\Zbb^2)$ and $\ell \in \Zbb^2$.
%    Then
%    \[
%        M^a_\ell = \left(\cokerxy{M}\right)^a_\ell
%        \;\;\;\;
%        M^b_\ell = \left(\kerxy{M}\right)^b_\ell
%    \]
%\end{corollary}

\begin{definition}
    \label{definition:concave-convex-corners}
    Let $I \subseteq \Zbb^2$ be a spread curve.
    A point $i \in I$ is
    \begin{itemize}
        \item a \emph{convex corner} of $I$ if $i-(1,0) \notin I$ and $i-(0,1) \notin I$;
        \item an \emph{inner convex corner} of $I$ if $i+(1,0) \in I$ and $i+(0,1) \in I$;
        \item a \emph{concave corner} of $I$ if $i+(1,0) \notin I$ and $i+(0,1) \notin I$;
        \item an \emph{inner concave corner} of $I$ if $i-(1,0) \in I$ and $i-(0,1) \in I$.
    \end{itemize}
\end{definition}

The sets of convex, inner convex, concave, and inner concave corners of a spread curve $I \subseteq \Zbb^2$ are donted by
$\conv(I)$, $\inconv(I)$, $\conc(I)$, and $\inconc(I)$, respectively.

%Note that the set of convex (resp.~concave) corners of a spread curve coincides with its set of minimal (resp.~maximal) elements.

The proof of the following result is straightforward.

\begin{lemma}
    \label{lemma:convex-corner-algebraic}
    Let $I \subseteq \Zbb^2$ be a spread curve, let $M \coloneqq \kbb_I \in \rep(\Zbb^2)$, and let $\ell \in \Zbb^2$.
    \begin{enumerate}
        \item The point $\ell$ is a convex corner of $I$ if and only if $M^d_\ell = 1$.
        \item The point $\ell$ is an inner convex corner of $I$ if and only if $M^i_{\ell+1} = 1$.
              \qed
    \end{enumerate}
\end{lemma}

\begin{lemma}
    \label{lemma:coker-fixes-d-and-i}
    Let $M \in \rep(\Zbb^2)$, and let $\ell \in \Zbb^2$.
    \begin{enumerate}
        \item $\left(\cokerxy{M}\right)^d_\ell = M^d_\ell$;
        \item $\left(\kerxy{M}\right)^i_{\ell} = M^i_{\ell}$.
    \end{enumerate}
\end{lemma}
\begin{proof}
    We give the proof for claim (1); claim (2) follows from a similar argument.
    Note that $\left(\cokerxy{M}\right)^d_\ell \geq M^d_\ell$, since any summand of type $d$ in $M|_{\sq(\ell)}$ induces a summand of type $d$ in $\cokerxy{M}|_{\sq(\ell)}$.
    The converse inequality follows by checking that only summands of type $d$ in $M|_{\sq(\ell)}$ can induce a summand of type $d$ in $\cokerxy{M}|_{\sq(\ell)}$.
\end{proof}

\curvesbettitables*
\begin{proof}
    We start with the equalities computing the Betti tables.
    Let $\ell \in \Zbb^2$.
    Using \cref{lemma:betti-tables-and-decomposition-on-square} and the definitions of $\topleft{M}$ and $\botright{M}$, we get $M^b_\ell = \dim (\topleft{M})_\ell$ and
    $M^c_\ell = \dim (\botright{M})_\ell$.

    To simplify notation, if $C$ is a finite multiset of spread curves, let $\conv(C)_\ell$, $\inconv(C)_\ell$, $\conc(C)_\ell$, $\inconc(C)_\ell \in \Nbb$ be the multiplicities of $\ell$ as a convex, inner convex, concave, and inner concave corner of the elements of $C$, respectively.
    It is then enough to prove
    \begin{align*}
        \conv\left(\dec \cokerxy{M}\right)_\ell &= M^d_\ell,\\
        \inconv\left(\dec\kerxy{M}\right)_\ell &= M^i_{\ell+1},\\
        \conc\left(\dec\kerxy{M}\right)_\ell &= M^a_{\ell+1},\\
        \inconc\left(\dec\cokerxy{M}\right)_\ell &= M^j_\ell.
    \end{align*}
    We do the first two cases, that is the cases of $d$ ($\conv$) and $i$ ($\inconv$), since $a$ is dual to $d$, and $j$ is dual to $i$ (in the sense of \cref{section:duality}).

    For the case of $d$, we have
    $\conv(\dec \cokerxy{M})_\ell
        = \left(\cokerxy{M}\right)^d_\ell = M^d_\ell$,
    where in the first equality we used
    \cref{lemma:convex-corner-algebraic}(1), and in the second we used \cref{lemma:coker-fixes-d-and-i}(1).
    For the case of $i$ we have
    $\inconv(\dec \kerxy{M})_\ell
        = \left(\kerxy{M}\right)^i_{\ell+1} = M^i_{\ell+1}$,
    where in the first equality we used
    \cref{lemma:convex-corner-algebraic}(2)
    and in the second we used
    \cref{lemma:coker-fixes-d-and-i}(2).

    \smallskip

    We now prove the isomorphisms relating the corners of $M$ to the corners of $\cokerxy{M}$.
    We prove $\topleft{M} \cong \topleft{\left(\cokerxy{M}\right)}$, the other isomorphism being analogous.
    Let $i \in \Zbb^2$.
    Let $\pi : M \to \cokerxy{M}$ be the quotient morphism, and consider the following diagram
    \[
    \begin{tikzpicture}
        \matrix (m) [matrix of math nodes,row sep=2.5em,column sep=3em,minimum width=1em,nodes={text height=1.75ex,text depth=0.25ex}]
        {
            \kery{M}[0,1] &  M & \cokerx{M} \\
             \kery{\left(\cokerxy{M}\right)}[0,1] &  \cokerxy{M} & \cokerx{\left(\cokerxy{M}\right)},\\};
        \path[line width=0.75pt, ->]
        (m-1-1) edge [right hook->,above] node {$f$} (m-1-2)
        (m-1-2) edge [above] node {$g$} (m-1-3)
        (m-2-1) edge [right hook->,above] node {$f'$} (m-2-2)
        (m-2-2) edge [above] node {$g'$} (m-2-3)
        (m-1-1) edge [left] node {$\kery{\pi}[0,1]$} (m-2-1)
        (m-1-2) edge [right,->>] node {$\pi$} (m-2-2)
        (m-1-3) edge [right] node {$\cokerx{\pi}$} node [left] {$\cong$} (m-2-3)
        ;
    \end{tikzpicture}
    \]
    so that $\topleft{M} = \im(g \circ f)$ and $\topleft{\left(\cokerxy{M}\right)} = \im(g' \circ f')$,
    and we need to prove that the image of the top row is isomorphic to the image of the bottom row.

    Note that the vertical morphism $\cokerx{\pi}$ is an isomorphism, since taking the quotient by the action of $\xbf\ybf$ and then by that of $\xbf$ is equivalent to directly taking the quotient by the action of $\xbf$.
    %$\pi : M \to \cokerxy(M)$ is given by taking the quotient by the image of the action $\xbf\ybf$, and the morphism $\cokerx{\pi}$ is the one induced by taking the quotient of both the domain and codomain by the image of the action of $\xbf$.
    So it is enough to show that, given $\ell \in \Zbb^2$ and $z' \in \kery{\left(\cokerxy{M}\right)}_{\ell-(0,1)}$, there exists $z \in \kery{M}_{\ell-(0,1)}$ such that $(\cokerx{\pi} \circ g \circ f)(z) = (g' \circ f')(z')$.
    By definition, the element $z'$ is an element $z' \in (\cokerxy{M})_{\ell-(0,1)}$ such that $\ybf z' = 0 \in (\cokerxy{M})_{\ell}$.
    Let $z'' \in M_{\ell-(0,1)}$ be a representative of $z'$ (i.e., $\pi(z'') = z$), so that 
    $\ybf z'' = \xbf\ybf w$ for some $w \in M_{\ell- (1,1)}$.
    Let $z \coloneqq z'' - \xbf w \in M_{\ell-(0,1)}$, and note that $\ybf z = \ybf z'' - \ybf\xbf w = \xbf\ybf w - \ybf\xbf w = 0$, so that $z \in \kery{M}_{\ell-(0,1)}$.

    To conclude, note that we have $(g' \circ f')(z') = (g' \circ \pi)(z'') = (\cokerx{\pi} \circ g)(z'')$, so it is sufficient to show that $g(z'') = g(z) \in \left(\cokerx{M}\right)_{\ell}$, which is clear since $z'' - z = \xbf w$.
\end{proof}

\curveslinearsize*
\begin{proof}
    Let us start with the first statement.
    Since $M$ is finitely generated, so is $\cokerxy{M}$, and thus so are each of the indecomposable summands of $\cokerxy{M}$, which are the birth-curves of $M$, by definition.
    It follows that each one of the birth-curves must have at least one convex corner.
    The result then follows from the first equality in \cref{theorem:curves-betti-tables}, and the fact that, for finitely generated representations of $\Zbb^2$, we defined $\Ncaltwo$ as the number of birth-curves.

    Let us now prove the second statement.
    A finite $\Zbb^2$-filtered simplicial complex $(K,f)$ consists of a finite simplicial complex $K$ and a function $f : K \to \Zbb^2$ from the simplices of $K$ to $\Zbb^2$ that is monotonic, in the sense that $\sigma \subseteq \tau \in K$ implies $f(\sigma) \leq f(\tau) \in \Zbb^2$.
    For any $i \in \Nbb$, this defines a representation $H_i(f) \in \rep(\Zbb^2)$ by letting $H_i(f)_\ell = H_i(\{ \sigma \in K : f(\sigma) \leq \ell\})$.
    Equivalently, we can define a chain complex $C_\bullet$ of projective objects of $\rep(\Zbb^2)$, and $H_i(f) = \ker( C_i \to C_{i-1} ) / \im(C_{i+1} \to C_{i})$.
    Here $C_i$ is the projective representation generated by the $i$th simplices of $K$, so that $\beta_0(C_i) = \sum_{\sigma \in K_i} \delta_{f(\sigma)}$.
    We have
    \[
        \Ncaltwo\left(H_i(f)\right)
        \leq
        \left|\beta_0^{H_i(f)}\right|
        =
        \left|\beta_0^{\ker( C_i \to C_{i-1}) / \im(C_{i+1} \to C_i)}\right|
        \leq
        \left|\beta_0^{\ker( C_i \to C_{i-1})}\right|
        \leq
        \left| \beta_0^{C_i} \right| = |K_i|,
    \]
    where in the first inequality we used the first statement,
    in the second inequality we used \cref{lemma:inequalities-betti-tables-1},
    and in the third inequality we used \cref{lemma:inequalities-betti-tables-2}.
\end{proof}

\section{Relating birth- and death-curves with the boundary}
\label{section:bands}

\subsection{The boundary}
We start by formally describing the category $\boundaries$, and explaining how the boundary (\cref{definition:boundary}) is naturally an object of this category.

%\begin{remark}
%    \luis{do we need this remark?}
%    The definitions of $\kerxy{M}$, $\kerx{M}$ and $\kery{M}$ follow the ``open convention'' in persistence theory, where, for example, a pair of integers $a \leq b \in \Zbb$ represents the half open interval $[a,b) \in \Zbb$; using the ``closed convention'' such pair represents the closed interval $[a,b]$.
%    In this paper, the open convention has several notational advantages such as minimizing overlaps between birth and death curves (e.g., \cref{figure:boundary-of-spread}), and avoiding shifts in some results (e.g., \cref{theorem:curves-betti-tables}).
%    For other results, however, the closed convention is more convenient, where we define
%    $\kerxyc{M} \coloneqq \kerxy{M}[1,1]$, $\kerxc{M} = \kerx{M}[1,0]$, and $\keryc{M}[0,1] \coloneqq \kery{M}$.
%\end{remark}

\begin{definition}
    \label{definition:ephemeral-pair}
    A \emph{boundary} consists of a pair of ephemeral representations $B,D \in \rep(\Zbb^2)$ 
    together with morphisms $f : \keryc{D} \to \cokerx{B}$ and $g : \kerxc{D} \to \cokery{B}$.
    A morphism of boundaries from $(B,D,f,g)$ to $(B',D',f',g')$ consists of a pair of morphisms $\phi : B \to B'$ and $\psi : D \to D'$ such that the following diagrams commute:
    \[
    \begin{tikzpicture}
        \matrix (m) [matrix of math nodes,ampersand replacement=\&,row sep=1em,column sep=1em,minimum width=1em,nodes={text height=1.75ex,text depth=0.25ex}]
        {
            \keryc{D} \&  \& \cokerx{B}                \\
            \&     \& \\
            \keryc{D}' \&  \& \cokerx{B}' \\};
        \path[line width=0.75pt, ->]
        (m-1-1) edge [above] node {$f$} (m-1-3)
        (m-1-1) edge [left] node {$\keryc{\psi}$} (m-3-1)
        (m-3-1) edge [above] node {$f'$} (m-3-3)
        (m-1-3) edge [right] node {$\cokerx{\phi}$} (m-3-3)
        ;
    \end{tikzpicture}
    \;\;\;\;\;
    \begin{tikzpicture}
        \matrix (m) [matrix of math nodes,ampersand replacement=\&,row sep=1em,column sep=1em,minimum width=1em,nodes={text height=1.75ex,text depth=0.25ex}]
        {
            \kerxc{D} \&  \& \cokery{B}                \\
            \&     \& \\
            \kerxc{D'} \&  \& \cokery{B}' \\};
        \path[line width=0.75pt, ->]
        (m-1-1) edge [above] node {$g$} (m-1-3)
        (m-1-1) edge [left] node {$\kerxc{\psi}$} (m-3-1)
        (m-3-1) edge [above] node {$g'$} (m-3-3)
        (m-1-3) edge [right] node {$\cokery{\phi}$} (m-3-3)
        ;
    \end{tikzpicture}
    \]
    We denote the category of boundaries by $\boundaries$.
\end{definition}

\begin{construction}
    Define a functor $\partial : \rep(\Zbb^2) \to \boundaries$ by mapping $M \in \rep(\Zbb^2)$ to the boundary $(\cokerxy{M}, \kerxyc{M}, f,g)$, where $f$ and $g$ are the following composites, respectively:
    \begin{align*}
        \keryc{\left(\kerxyc{M}\right)} = \keryc{M} &\hookrightarrow M \twoheadrightarrow \cokerx{M} = \cokerx{\left(\cokerxy{M}\right)}\\
        \kerxc{\left(\kerxyc{M}\right)} = \kerxc{M} &\hookrightarrow M \twoheadrightarrow \cokery{M} = \cokery{\left(\cokerxy{M}\right)}
    \end{align*}
    The action of $\partial$ on morphisms is the evident one.
\end{construction}

\subsection{The boundary as a representation of a string algebra}
We now describe a fully faithful embedding of $\boundaries$ into the category of representations of the string algebra $R$ of \cref{figure:large-gentle-quiver}.

\begin{construction}
    Define a functor $\iota : \boundaries \to \rep(R)$ by mapping a boundary $(B,D,f,g)$ to the representation $X\in \rep(R)$ given as follows:
    \begin{itemize}
        \item On the down-left, left, and down vertices, we let $X_{d\ell_i} = B_i = \cokerxy{B_i}$, $X_{\ell_i} = \cokerx{B}_i$, and $X_{d_i} = \cokery{B}_i$.
        \item On the up-right, right, and up vertices, we let $X_{ur_i} = D_i = \kerxyc{D_i}$, $X_{r_i} = \kerxc{D}_i$, and $X_{u_i} = \keryc{D}_i$.
        \item The horizontal morphisms are induced by the natural morphisms $\keryc{D} \hookrightarrow D \xrightarrow{\xbf} \keryc{D}[1,0]$ and $\cokery{B}[-1,0] \xrightarrow{\xbf} B \twoheadrightarrow \cokery{B}$.
        \item The vertical morphisms are induced by the natural morphisms $\kerxc{D} \hookrightarrow D \xrightarrow{\ybf} \kerxc{D}[1,0]$ and $\cokerx{B}[-1,0] \xrightarrow{\ybf} B \twoheadrightarrow \cokerx{B}$.
        \item For the diagonal morphism, we let $X_{u_i,\ell_i} = f_i$ and $X_{r_i,d_i} = g_i$.
    \end{itemize}
    The action of $\iota$ on morphisms is the evident one.
\end{construction}

It is clear that, if $(B,D,f,g)$ is a boundary, then $\iota(B,D,f,g)$ contains the exact same information as $(B,D,f,g)$, that is:

\begin{lemma}
    \label{lemma:G-fully-faithful}
    The functor $\iota : \boundaries \to \rep(R)$ is fully faithful.
    \qed
\end{lemma}

\begin{notation}
    By identifying $\boundaries$ with its essential image $\iota(\boundaries) \subseteq \rep(R)$, we treat $\boundaries$ as a full subcategory of $\rep(R)$.
\end{notation}

\begin{figure}
    \includegraphics[width=1\textwidth]{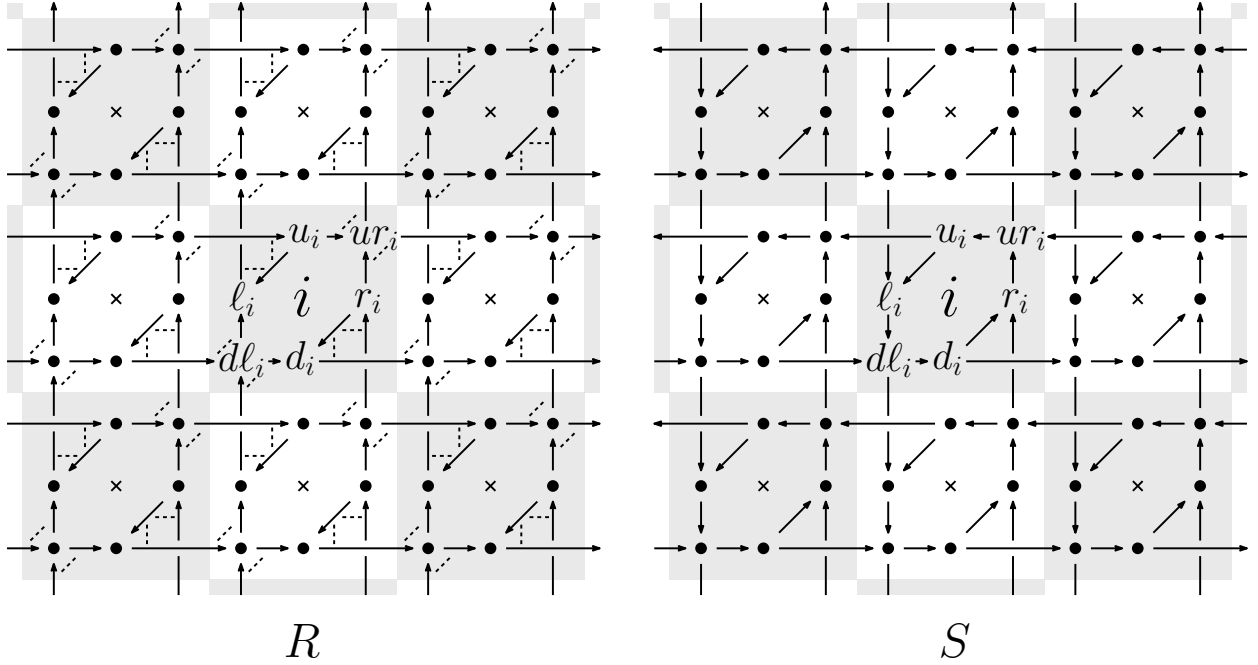}
    \caption{\emph{Left.} An infinite quiver $R$ with zero relations represented by dotted lines.
    The vertices of $R$ are parametrized as follows:
    for each element $i \in \Zbb^2$, there are six associated vertices of $R$, which we denote by $r_i$ (right), $ur_i$ (up-right), $u_i$ (up), $\ell_i$ (left), and $d\ell_i$ (down-left), $d_i$ (down).
    The arrows are as in the diagram.
    Note that we only name the vertices associated with the bigrade $i \in \Zbb^2$, for simplicity.
    The checker pattern is for visualization purposes, to help group the vertices associated to each bigrade.
    \emph{Right.} A directed graph $S$.
    The vertices of $S$ are the same as those of $R$.
    The arrows are the same, except that the edges in half of the rows (those corresponding to $u$ and $ur$) and in half of the columns (those corresponding to $\ell$ and $d\ell$) have the opposite orientation.}
    \label{figure:large-gentle-quiver}
\end{figure}

\subsection{Classifying band representations of $R$}

Towards the goal of providing a classification of the band representations of $R$,
we first provide a correspondence between $\Zbb$-words in $R$ and oriented cycles in the direct graph $S$ of \cref{figure:large-gentle-quiver}.
The idea of characterizing words in a string algebra using oriented paths in an associated oriented graph that is obtained by changing some orientations appears in, e.g., \cite{bell-et-al}.

\begin{figure}
    \includegraphics[width=1\textwidth]{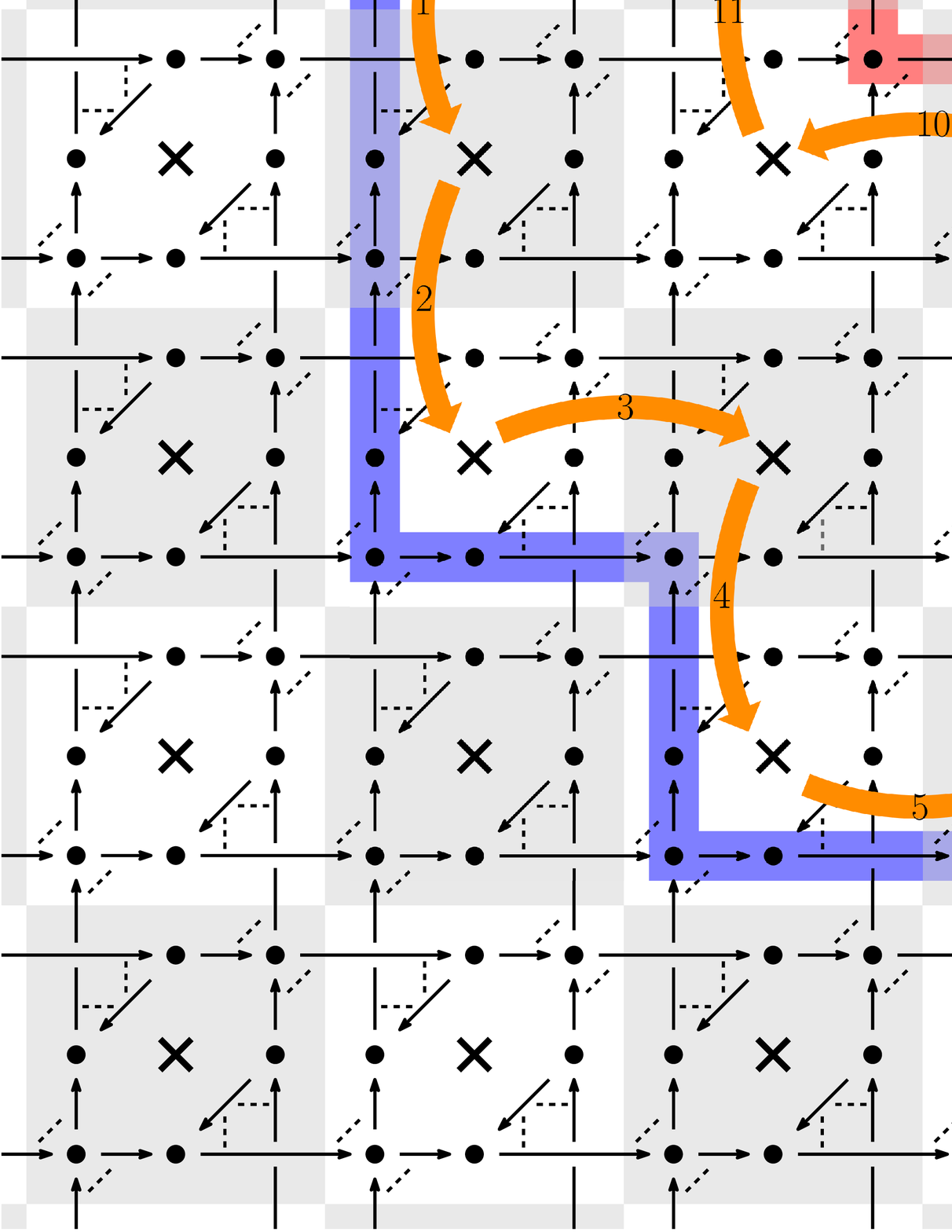}
    \caption{We let $M = \kbb_I$ the spread representation in \cref{figure:boundary-of-spread}.
    \emph{Left.}
    The irreducible discrete closed curve $\gamma$ (orange) representing the (unique) boundary component of $\partial M$, and its corresponding $\Zbb$-word in $R$.
    This illustrates \cref{construction:closed-curve-to-oriented-cycle,construction:closed-curve-to-oriented-cycle}.
    Note that, as we follow $\gamma$, the $\Zbb$-word in $R$ traces the birth- and death-curves of the representation.
    This induces two canonical pairings $\partial_{u\ell}$ and $\partial_{dr}$ between the birth- and death-curves of the representation.
    \emph{Right.}
    The birth- and death-curves, corners, and bigraded Betti tables of $M$ (generators represented by circles, relations by triangles, and syzygies by squares).
    Note that, following $\gamma$, we get an action of the group $\Zbb$ on the union of all Betti tables of $M$, which characterizes $\gamma$, and, which can also be characterized by a directed graph structure on the union of the Betti tables of $M$.}
    \label{figure:drawing-boundary}
\end{figure}

\begin{definition}
    Let $G$ be an oriented graph.
    \begin{itemize}
    \item An \emph{oriented cycle} in $G$ consists of a sequence of edges $e = \{e_0, \dots, e_{k-1}\}$ of $G$ such that $\ssf(e_{i+1}) = \tsf(e_i)$ for all $i \in \Zbb/k\Zbb$.
    \item If $n \geq 2 \in \Nbb$, the concatenation of an oriented cycle $e$ with itself $n$ times is again an oriented cycle, denoted by $e^n$.
    \item An oriented cycle is \emph{irreducible} if it cannot be written as $e^n$ for some oriented cycle $e$ and $n \geq 2$.
    \item Two oriented cycles $\{e_0, \dots, e_{k-1}\}$ and $\{e'_0, \dots, e'_{k-1}\}$ are \emph{rotationally equivalent} if there exists $j \in \Zbb$ such that $e_i = e'_{i+j}$ for all $i \in \Zbb/k\Zbb$.
    \end{itemize}
\end{definition}

\begin{construction}
    \label{construction:oriented-cycle-to-periodic-word}
    %We define a function $F$ mapping cycles in $S$ to periodic words in $R$.
    Given an oriented cycle $e = \{e_0, \dots, e_{k-1}\}$ in $S$, we define a $\Zbb$-word $F(e) \in R$ by concatenating $e$ with itself countably many times, as follows
    \[
    \dots, e^{s_{k-1}}_{k-1}, e^{s_0}_0, e^{s_1}_1, \dots, e^{s_{k-2}}_{k-2}, e^{s_{k-1}}_{k-1}, e^{s_0}_0, e^{s_1}_1, \dots, e^{s_{k-2}}_{k-2}, e^{s_{k-1}}_{k-1}, e^{s_0}_0, \dots,
    \]
    where $s_i = +1$ if the edge $e_i$ has the same orientation in $S$ and in $R$, and $s_i = -1$ otherwise.
\end{construction}

The proof of the next result is then straightforward.

\begin{lemma}
    \label{lemma:oriented-cycle-to-periodic-word}
    \cref{construction:oriented-cycle-to-periodic-word} induces a bijection between the set of irreducible oriented cycles in $S$ up to rotational equivalence, and the set of equivalence classes of periodic $\Zbb$-words in $R$.
    \qed
\end{lemma}

Next, we provide a correspondence between oriented cycles in $S$ and discrete closed plane curves, in the following sense.

\begin{definition}
    \label{definition:closed-curve}
    \begin{itemize}
    \item A \emph{discrete closed plane curve} consists of a sequence $\gamma = \{\gamma_0, \dots, \gamma_{k-1}\}$ with $\gamma_i \in \Zbb^2$, such that $\|\gamma_{i+1} - \gamma_i\|_1 \leq 1$ for all $i \in \Zbb/k\Zbb$.
    \item If $n \geq 2 \in \Nbb$, the concatenation of a closed curve $\gamma$ with itself $n$ times is again a closed curve, denoted by $\gamma^n$.
    \item A closed curve is \emph{irreducible} if it cannot be written as $\gamma^n$ for some closed curve $\gamma$ and $n \geq 2$.
    \item Two closed curves $\gamma$ and $\gamma'$ are \emph{rotationally equivalent} if there exists $j \in \Zbb$ such that $\gamma_i = \gamma'_{i+j}$ for all $i \in \Zbb/k\Zbb$.
    \end{itemize}
\end{definition}

Note that, in a discrete closed plane curve, the same vertex can appear consecutively.
In order to relate discrete closed plane curves and cycles in $S$, it is convenient to define the following.

%The next construction relies on the directed graph $S$, of \cref{figure:large-gentle-quiver}.
An \emph{up-right vertex} of $R$ (and of $S$) is a vertex of the form $ur_i$ for $i \in \Zbb^2$, while a \emph{down-left vertex} is a vertex of the form $d\ell_i$ for $i \in \Zbb^2$; a \emph{diagonal vertex} is either an up-right vertex or a down-left vertex.
If $v = ur_i$ or $v = d\ell_i$ is a diagonal vertex, its corresponding \emph{grade} is $\mathrm{idx}(v) \coloneqq i \in \Zbb^2$.
\emph{A full turn around} $i$ is an oriented cycle that is rotationally equivalent to $\{(ur_i, u_i), (u_i,\ell_i), (\ell_i,d\ell_i), (d\ell_i,d_i), (d_i, r_i), (r_i,ur_i)\}$ in $S$.

We observe that, if $e$ is an oriented cycle in $S$, then it must visit at least one (in fact at least two) diagonal vertices.
We also make the following simple observation.

\begin{lemma}
    If $v, w \in S$ are such that $\|\idx(v) - \idx(w)\|_1 = 1$, then there exists a unique shortest oriented path from $v$ to $w$ in $S$.
    \qed
\end{lemma}

\begin{construction}
    \label{construction:closed-curve-to-oriented-cycle}
    Given a discrete closed plane curve $\gamma = \{\gamma_0, \dots, \gamma_{k-1}\}$, we construct an oriented cycle $e$ in $S$ as follows.
    \begin{itemize}
    \item If $\gamma$ is constantly $i \in \Zbb$, then let $e$ be the concatenation of $k$ full turns around $i$.
    \item Otherwise, by grouping consecutive appearances of the same grade, we can write $\gamma = i_1^{t_1+1} \cdot i_2^{t_2+1} \cdots i_{j}^{t_j+1}$, so that consecutive grades satisfy $i_m \neq i_{m+1} \in \Zbb^2$, and such that $i_1 \neq i_j \in \Zbb^2$ as well.
    \begin{itemize}
        \item Pick the unique pair of diagonal vertices $v_1, v_2$ at shortest path distance in $S$, such that $\idx(v_1) = i_1$ and $\idx(v_2) = i_2$; let $p_1$ be the unique shortest path from $v_1$ to $v_2$ in $S$.
        \item Let $p_2$ be the unique shortest oriented path from $v_2$ to the closest diagonal vertex $v_3$ that satisfies $\idx(v_3) = i_3$.
        \item Let $p_3$ be the unique shortest oriented path from $v_3$ to the closest diagonal vertex $v_4$ that satisfies $\idx(v_4) = i_4$.
        \item Etc., until we reach $v_1$ again through a shortest path $p_j$.
    \end{itemize}
    \end{itemize}
    Construct the oriented cycle $e \coloneqq p_1 \cdots p_j$.
\end{construction}

See \cref{figure:drawing-boundary} for an example illustrating \cref{construction:closed-curve-to-oriented-cycle}.

\begin{construction}
    \label{construction:oriented-cycle-to-closed-curve}
    Given $e = \{e_0, \dots, e_{k-1}\}$ an oriented cycle in $S$, we construct a discrete closed plane curve $\gamma$ as follows.
    \begin{itemize}
        \item If all that $e$ is doing is $k$ full turns around some $i \in \Zbb^2$, then we construct the discrete closed plane curve $\gamma \coloneqq i^k$ given by $i$ repeated $k$ times.
        \item Otherwise, rotating $e$ if necessary, we can assume that $e$ starts at a diagonal vertex $v$, which has a different grade than the last diagonal vertex that it visits.
        Note that $e$ behaves as follows:
        \begin{itemize}
            \item Starting from $w_1 \coloneqq v$, it first does a number $t_1 \in \Nbb$ of full turns around $\idx(w_1) \in \Zbb^2$.
            \item It then gets to the first diagonal vertex $w_2$ with the property that $\idx(w_2) \neq \idx(w_1)$, and proceeds to do some number $t_2 \in \Nbb$ of full turns around $\idx(w_2) \in \Nbb$.
            \item It then gets to the first subsequent diagonal vertex $w_3$ such that $\idx(w_3) \neq \idx(w_2)$, and proceeds to do some number $t_3$ of full turns around $\idx(w_3) \in \Nbb$.
            \item Etc., until we reach $w_v$ again.
        \end{itemize}
        From this observation, we get a sequence of grades $\idx(w_1), \dots, \idx(w_{j})$ and a sequence of numbers $t_1, \dots, t_{j}$ (which are allowed to be zero), and we construct the concatenation
        \[
            \gamma \coloneqq \idx(w_1)^{t_1+1} \cdot \idx(w_2)^{t_2+1} \cdots \idx(w_{j})^{t_{j}+1}.
        \]
    \end{itemize}
\end{construction}

        %\item By rotating $e$ if necessary, we can assume that $v_0$ is a diagonal vertex.
        %\item Let $V' = \{v_0, v'_1, \dots, v'_{m'}\}$ be the sublist of $V$ containing only the diagonal vertices.
%        \item
%        If $I$ is not constant, by rotating $e$ if necessary, we can assume that $i_0 \neq i_{m}$.
%        The path $e$ then behaves as follows:
%        It starts at a diagonal vertex $w_0 \coloneqq v_0$.
%        Then it does an even (and potentially zero) number $2t_0$ of full turns around $i_0$, and either returns to $v_0$ or to the other diagonal vertex with index $\idx(v_0)$.
%        It then goes to the first vertex $w_1$ with $\idx(w_1) \neq \idx(w_0)$.
%        It then does an even (and potentially zero) number $2t_1$ of full turns around $\idx(w_1)$, and either returns to $w_1$ or to the other diagonal vertex with index $\idx(w_1)$ (by doing and extra half-turn).
%        By following $e$ once in this way, we get a list of indices $\idx(w_0), \idx(w_1), \dots, \idx(w_{j})$ and half lengths.

The following proof is straightforward, but tedious, so we omit it.

\begin{lemma}
    \label{lemma:closed-curve-to-oriented-cycle}
    \cref{construction:closed-curve-to-oriented-cycle,construction:oriented-cycle-to-closed-curve}
    induce inverse bijections between the set of discrete closed curves up to rotational equivalence and the set of oriented cycles in $S$ up to rotational equivalence.
    The bijections respect irreducibility.
    \qed
\end{lemma}

\begin{proposition}
    \label{proposition:classification-of-bands}
    %The indecomposable band modules in $\rep(R)$ are classified by boundary components.
    %More specifically,
    There is a bijection between isomorphism classes of band modules over $R$ and equivalence classes of boundary components.
    The bijection is given by mapping a discrete closed plane curve to a periodic $\Zbb$-word in $R$ using \cref{construction:oriented-cycle-to-periodic-word,construction:closed-curve-to-oriented-cycle}.
\end{proposition}
\begin{proof}
    We apply the classification result for indecomposable representations of string algebras (\cref{theorem:string-algebra-module-classification}), which is in terms of pairs consisting of an equivalence class of periodic $\Zbb$-words and an isomorphism class of indecomposable $\kbb[T,T^{-1}]$-modules.
    The correspondence between equivalence classes of periodic $\Zbb$-words and irreducible discrete closed curves is obtained using the bijections of \cref{lemma:closed-curve-to-oriented-cycle,lemma:oriented-cycle-to-periodic-word}.
    The correspondence between isomorphism classes of finite dimensional indecomposable $\kbb[T,T^{-1}]$-modules and similarity classes of matrices as in \cref{definition:boundary-component} (invertible square matrices that are not similar to a block diagonal matrix with more than one block) is standard
    and is given by mapping such a matrix $\Tcal$ to the $\kbb[T,T^{-1}]$-module $M_\Tcal$ with underlying vector space $\kbb^\ell$ and action of $T$ given by $\Tcal$.
\end{proof}

\subsection{Decomposition of the boundary}
We conclude this section by proving the decomposition theorem for the boundary.

\begin{proposition}
    \label{proposition:partial-is-band}
    If $M \in \repf(\Zbb^2)$, then $\partial M \in \rep(R)$ decomposes as a finite direct sum of band modules.
\end{proposition}
\begin{proof}
    Let $N = \partial M$, and note that, since $M$ is finite dimensional, so is $N$.
    By \cref{theorem:string-algebra-module-classification}, the module $N$ decomposes as a finite direct sum of string modules, finite dimensional band modules, and primitive injective band modules.
    Since $N$ is finite dimensional, it must in fact decompose as a finite direct sum of finite dimensional string modules and finite dimensional band modules.
    So it is sufficient to show that $N$ cannot have a finite dimensional string module as a direct summand.

    The proof relies on the following notion.
    Let $X$ be a representation of $R$,
    let $c$ be a vertex of $R$, let $z \neq 0 \in X_{c}$, and let
    $e$ be an arrow of $R$ incident to $c$.
    We say that $x$ \emph{interacts with} $e$ if either the source of $e$ is $c$ and $X_e(z) \neq 0$, or the target of $e$ is $c$ and $z$ is in the image of $X_e$.

    If $X$ is a finite dimensional string representation of $R$, then, since the word associated to the representation must be finite, there must exists a vertex $c$ of $R$ and $z \neq 0 \in X_c$ such that, from all arrows incident to $c$, there is at most one that interacts with $z$.

    So it is enough to show that, for every vertex $c$ of $R$ and every $z \neq 0 \in N_c$, there are at least two arrows incident to $c$ that interact with $z$.
    By the symmetry given by swapping $\xbf$ and $\ybf$, it is enough to prove this for the vertices $c = u_i$ and $c = ur_i$,
    and for the vertices $c = \ell_i$ and $c = d\ell_i$.
    Since the proof for the second pair of vertices is analogous to that for the first pair, we only give the proof for $c = \ell_i$ and $c = d\ell_i$

    If $c = ur_i$ and the arrow $ur_i \to u_{i + (1,0)}$ does not interact with $z$, then the corresponding element $z \in M_i$ satisfies $\xbf z = 0$ and thus $z \in N_{r_i}$, so the arrow $r_i \to ur_i$ interacts with $z$.
    An analogous argument show that, if the arrow $ur_i \to r_{i + (0,1)}$ does not interact with $z$, the arrow $u_i \to ur_i$ interacts with $z$.
    Thus, the element $z$ interacts with at least two arrows adjacent to $ur_i$.

    To conclude, we do the case $c = u_i$.
    The arrow $u_i \to ur_i$ always interacts with $z$, since the morphism $N_{u_i} \to N_{ur_i}$ is injective (as it is the inclusion of the elements annihilated by $\ybf$ into the elements annihilated by $\xbf\ybf$).
    Now, suppose that the arrow $u_i \to \ell_i$ does not interact with $z$.
    This means that the corresponding element $z \in M_i$ is mapped to zero in the quotient $M_i/\xbf M_{i-(1,0)}$.
    This implies that there exists $y \in M_{i-(1,0)}$ such that $\xbf y = z$, which means that the arrow $ur_{i-(1,0)} \to u_i$ interacts with $z$, as required.
\end{proof}

%\begin{construction}
%    \label{remark:boundary-component-module}
%    \todo{a boundary component module $\in \boundaries$}
%\end{construction}

\begin{proposition}
    \label{proposition:ephemeral-pair-contains-bands}
    Every band representation of $R$ belongs to $\boundaries \subseteq \rep(R)$.
    %Thus, to each boundary component $(\gamma, \Tcal)$ we associate the unique representation (up to isomorphism) $\kbb_{(\gamma, \Tcal)} \in \boundaries$ such that the isomorphism type of $\iota(\kbb_{(\gamma,\Tcal)})$ is that classified by $(\gamma, \Tcal)$, according to \cref{proposition:classification-of-bands}.
\end{proposition}
\begin{proof}
    Specifically, we need to show that every band representation $R$ belongs to the essential image of $\iota : \boundaries \to \rep(R)$.
    Let $X \in \rep(R)$ be a band representation.
    We define $B, D \in \rep(\Zbb^2)$ as $B_i \coloneqq X_{d\ell_i}$ and $D_i \coloneqq X_{ur_i}$, with structure morphisms induced by those of $X$.
    Next, we define morphisms $f : \keryc{D} \to \cokerx{B}$ and $g : \kerxc{D} \to \cokery{B}$.
    Note that, if $X$ is a band representation, then $X_{u_i}$ is the kernel of the morphism $X_{ur_i} \to X_{ur_{i+(0,1)}}$, which implies that $\keryc{D}_i \cong X_{u_i}$.
    Similarly, we have $\kerxc{D}_i \cong X_{r_i}$, $\cokerx{B}_i \cong X_{\ell_i}$ and $\cokery{B}_{i} \cong X_{d_i}$.
    Then, the structure morphism $X_{u_i} \to X_{\ell_i}$ and $X_{r_i} \to X_{d_i}$ provide the morphisms
    $f : \keryc{D} \to \cokerx{B}$ and $g : \kerxc{D} \to \cokery{B}$.
    It is then straightforward to check that $\iota(B,D,f,g) \cong X$.
\end{proof}

The next definition is well-defined thanks to \cref{proposition:ephemeral-pair-contains-bands}.

\begin{definition}
    \label{remark:boundary-component-module}
    Given a boundary component $(\gamma, \Tcal)$, we define the \emph{boundary component representation} $\kbb_{(\gamma,\Tcal)} \in \boundaries$ as the isomorphism class of representations classified by $(\gamma, \Tcal)$, according to \cref{proposition:classification-of-bands}.
\end{definition}

\boundarytame*
\begin{proof}
    This follows directly from \cref{proposition:classification-of-bands,proposition:partial-is-band,proposition:ephemeral-pair-contains-bands}.
\end{proof}

\boundarycompletespreaddec*
\begin{proof}
    Let $I$ be finite a spread in $\Zbb^2$.
    Consider the discrete closed plane curve $\gamma_I$ given by first going along $\cokerxy{I}$ from its to top-left corner to its right-down corner, and then going along~$\kerxyc{I}$ from its right-down corner to its top-left corner.
    Then $\partial \kbb_I \cong \kbb_{\gamma, \Tcal}$, where $\Tcal$ is the $1\times 1$ identity matrix.
    The result then follows by additivity and the uniqueness of decompositions into boundary component representations (\cref{theorem:boundary-tame}).
%    This follows by observing that, 
%    \todo{this proof needs to be adjusted}
%    Let $I \subseteq \Zbb^2$ be a finite spread, let $B = \cokerxy{I}$, and let $D  = \kerxy{I}$.
%    From \cref{lemma:birth-death-of-spread} and \cref{theorem:boundary-tame}, it follows that $\partial \kbb_I \cong \kbb_{\left(B, D, \id : \kbb \to \kbb\right)}$, which is indecomposable, and completely characterizes $I$.
%    By additivity, it follows that the boundary of a spread-decomposable module completely characterizes the module up to isomorphism, and the result follows.
\end{proof}

\begin{figure}
    \includegraphics[width=0.65\textwidth]{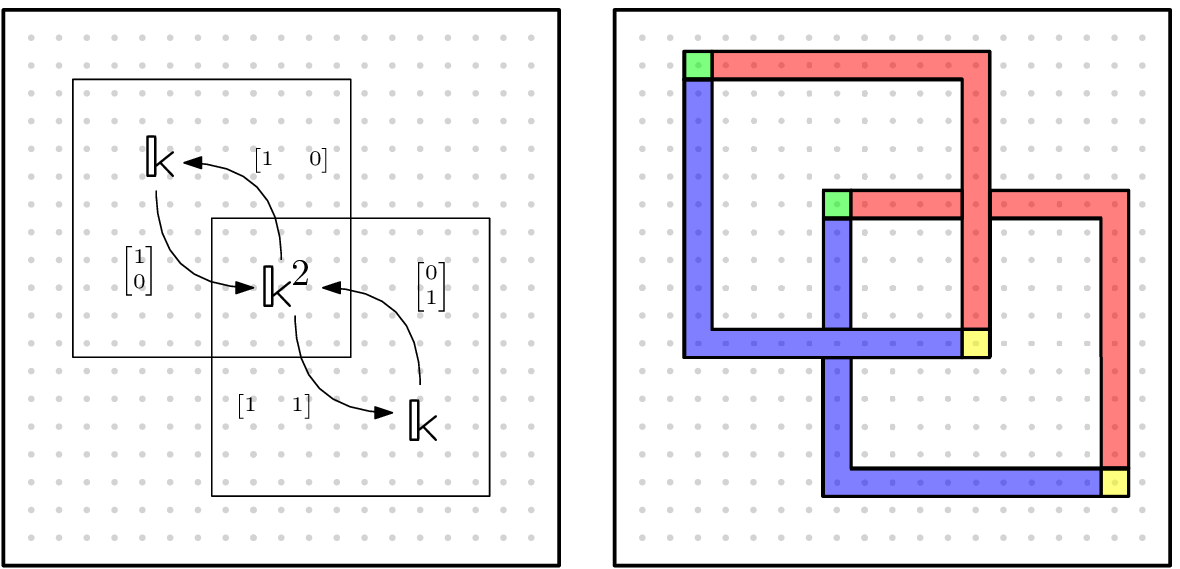}
    \caption{\emph{Left.} A representation $M$ of $\Zbb^2$.
    \emph{Right.}
    The boundary $\partial M$ is isomorphic to the boundary of a direct sum of two spread representations with square support, but the rank invariant of $M$ is different from that of a direct sum of two spread representations.}
    \label{figure:band-not-rank}

    \vspace{0.5cm}

    \includegraphics[width=0.65\textwidth]{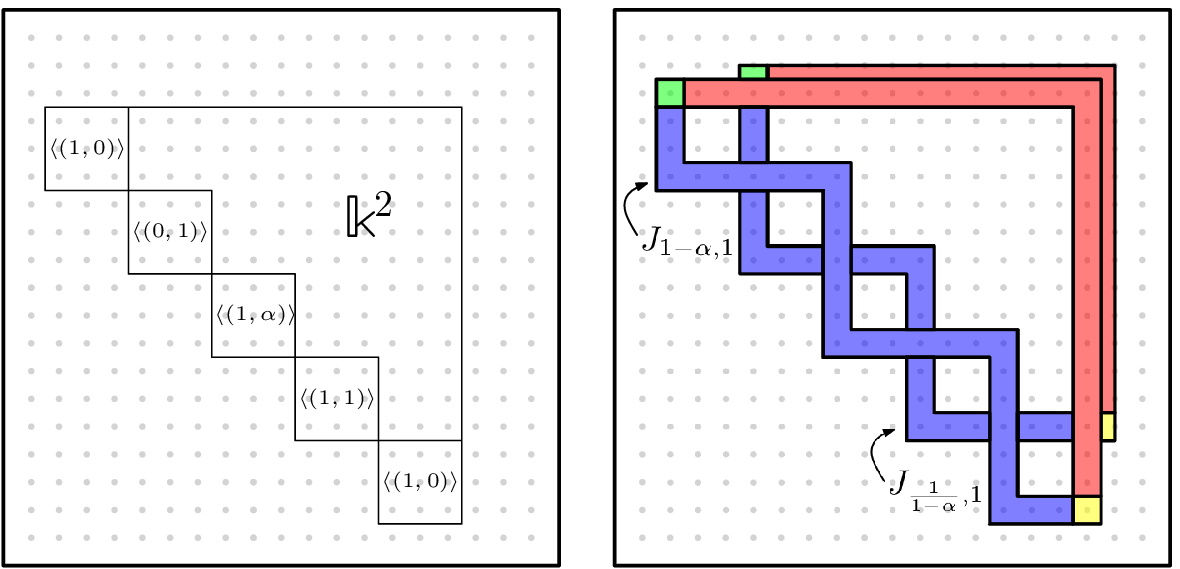}
    \caption{\emph{Left.} A representation $M$ of $\Zbb^2$, for any chosen $\alpha \neq 1$.
    The representation takes values in subspaces of $\kbb^2$, and all structure morphisms are the inclusions between these.
    \emph{Right.} The boundary $\partial M$ consists of two boundary components with associated matrices the Jordan blocks $J_{1-\alpha,1} \in \kbb^{1 \times 1}$ and $J_{1/(1-\alpha),1} \in \kbb^{1 \times 1}$.
    }
    \label{figure:non-trivial-band}

    \vspace{0.5cm}

    \includegraphics[width=1\textwidth]{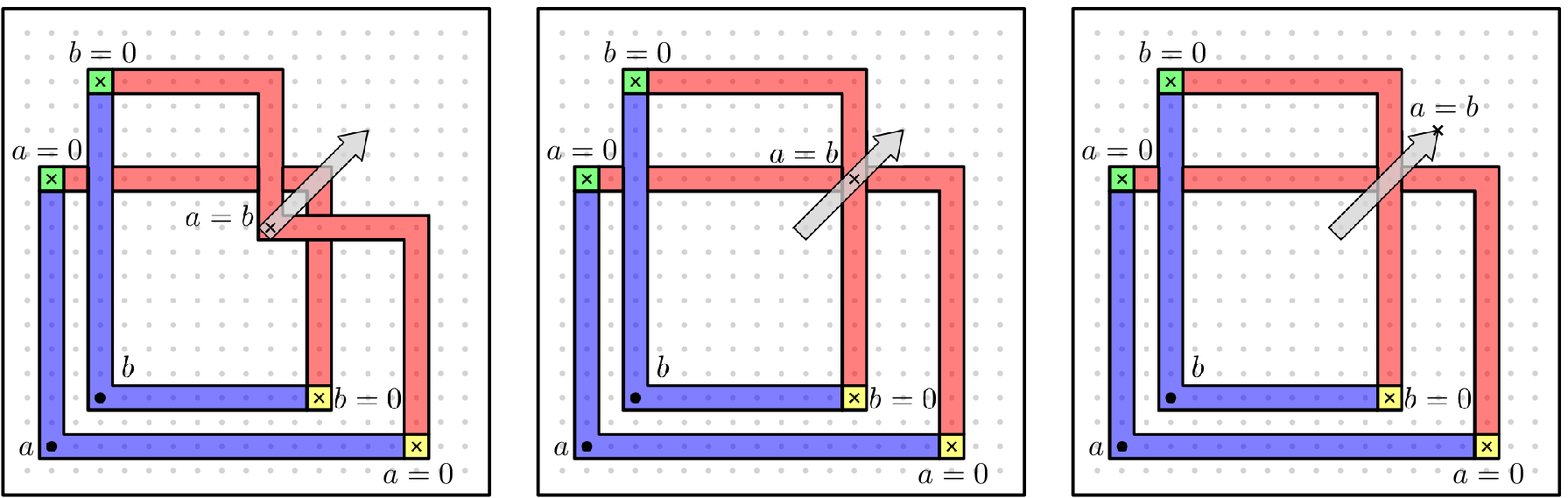}
    \caption{Three representations of $\Zbb^2$.
    Each representation has two generators $a$ and $b$ (dots), two relations of the form $a=0$, two relations of the form $b=0$, and a relation of the form $a=b$ (crosses).
    The difference is in the location of the relation $a=b$, which moves up and to the right.
    Moving the relation makes it so that the boundary of first representation has a single component, while the boundary of the other representations have two.}
    %in such a way that the single boundary component of the first representation decomposes into two boundary components in the other two representations.
    \label{figure:instability}
\end{figure}

\begin{figure}
    \includegraphics[width=1\textwidth]{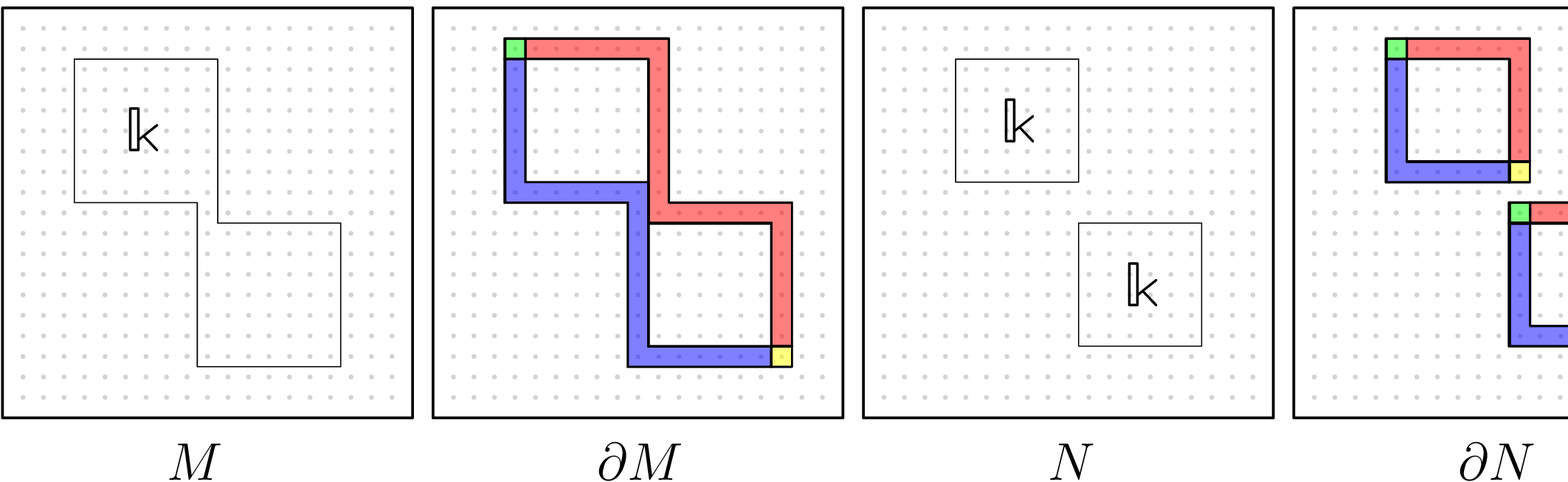}
    \caption{A spread representation $M$ for which $\Ncaltwo(M) = 1$, and $\Ncaltwo\left(N\right) = 2$, where $N$ is the image of the morphism $\xbf\ybf : M[-1,-1] \to M$.}
    \label{figure:monotonicity-fail}

    \vspace{0.5cm}

    \reflectbox{\rotatebox[origin=c]{90}{\includegraphics[scale=0.8]{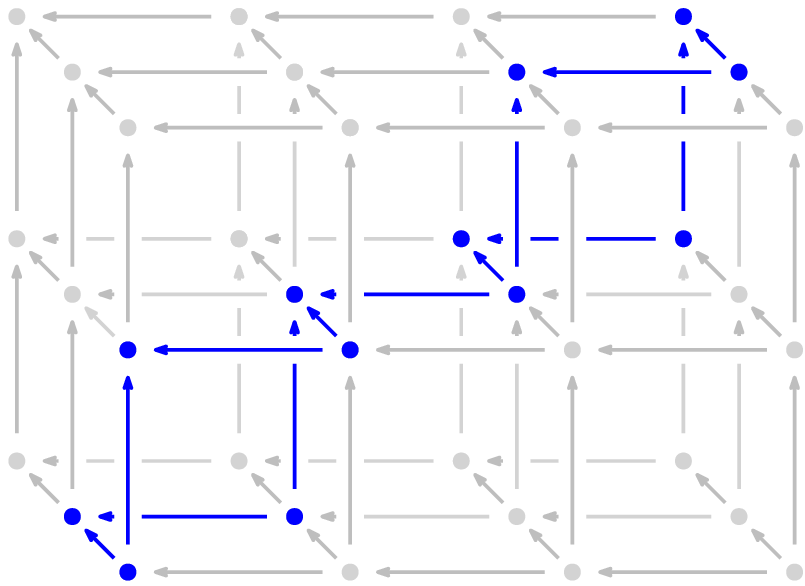}}}
    \vspace{-0.9cm}
    \caption{A spread $I$ in a finite sublattice of $\Zbb^3$.
    The order complex of $I$ is homotopy equivalent to a wedge of two circles, and thus $\chi(\Delta(I)) = -1$; in particular, by \cref{theorem:count-is-euler,proposition:signed-barcode-equals-inclusion-exclusion,lemma:signed-barcode-hooks}, we have $\Ninclexcl(M) = \chirkdec(\kbb_I) = \chirkexact(\kbb_I) = -1$.
    On the other hand, we have $\chigpd(\kbb_I) = \chispreadexact(\kbb_I) = 1$.}
    \label{figure:three-parameter-negative}
\end{figure}

\section{Computational examples}
\label{section:examples}

\subsection{The rank does not determine the boundary and viceversa}
\label{section:rank-and-boundary}

The boundary is a complete invariant on spread-decomposable representations (\cref{corollary:boundary-complete-spread-dec}), and the rank invariant is not complete on spread-decomposable representations (see, e.g. \cite[Example~1.3]{botnan-oppermann-oudot}).
Thus, the rank invariant does not determine the boundary of a representation.
Conversely,
\cref{figure:band-not-rank} shows that the boundary does not determine the rank of a representation.
The representation in the example is taken from \cite[Example~4.6]{botnan-lesnick}.

\subsection{A boundary with non-trivial associated matrix}

\cref{theorem:boundary-tame} guarantees that the boundary $\partial M$ of a representation $M \in \repf(\Zbb^2)$ decomposes as a direct sum of boundary component representations $\kbb_{(\gamma, \Tcal)}$.
A priori, it could be that $\partial M$ always decomposes as boundary components where the matrix $\Tcal$ is the identity $\id_1 \in \kbb^{1\times 1}$.
\cref{figure:non-trivial-band} shows that this is not the case.
It is not hard to turn the example into one where the boundary components have matrices of higher dimension.

\subsection{Behavior of the boundary with respect to perturbations}
\label{section:perturbations}
\cref{figure:instability} shows that the death-curves of a representation can interact in non-trivial ways when the representation is perturbed in the interleaving distance.
The interleaving distance \cite{chazal-et-al-2,chazal-et-al,lesnick} is a distance for representations of posets such as $\Zbb^n$, and $\Rbb^n$, but the concept itself is not required to understand the example, since here the perturbation is simply given by moving the grade of one of the relations of the representation.

\subsection{Lack of monotonicity}
\label{section:monotonicity-fail}
\cref{figure:monotonicity-fail} shows that any additive count $\Ncal$ that takes the value $1$ on spreads automatically fails to be monotonic with respect to inclusions.
The representation in the example is taken from \cite[Example~9.1]{botnan-lesnick-2}.

\subsection{Higher dimensional counts}
\label{section:higher-dimensional-counts}
\cref{figure:three-parameter-negative} shows that one of the obvious candidates
for a count $\Ncal^3$ on $\rep(\Gcal^3)$ is, in general, not positive, and does not take the value $1$ on spreads.
The candidate is $\Ninclexcl = \chirkdec = \chirkexact$, where the equalities follow from \cref{proposition:signed-barcode-equals-inclusion-exclusion,lemma:signed-barcode-hooks}.
The example also shows that $\chigpd \neq \chirkdec \neq \chispreadexact$ in general.

\section{Algorithmic computation of end-curves}
\label{section:algorithm}

The goal of this section is to prove the following result, which relies mainly on \cite{lesnick-wright,bauer_et_al:LIPIcs.SoCG.2023.15}.
In this section, when $P$ is a (finitely generated) projective representation, we let $|P| = |\beta_0^P|$ represent the number of indecomposable summands (counted with multiplicity).
The \emph{size} of a projective presentation $Q \to P \to M$ is $|P| \times |Q|$, since such presentation is usually encoded by the grades of the indecomposable (projective) summands of $P$ and $Q$ together with a $|P|\times|Q|$-matrix.

\begin{proposition}
    \label{proposition:main-computation-result}
    Let $M \in \repf(\Zbb^2)$, and let $Q \to P \to M \to 0$ be a finite projective presentation.
    \begin{itemize}
    \item A projective presentation of $\cokerxy{M}$ of size $O(|P| \times |Q|)$ can be constructed in $O(|P| \times |Q|)$ time.
    \item A projective presentation of $\kerxy{M}$ of size $O(|P|\times|Q|)$ can be constructed in $O\left(|P|^3+|Q|^3\right)$ time.
    \end{itemize}
\end{proposition}

Before giving the proof of \cref{proposition:main-computation-result}, we comment on its context and implications.
\cref{proposition:main-computation-result} is relevant to applied persistence and to topological inference by means of multiparameter persistence \cite{botnan-lesnick}.
In this case, the input is typically a bifiltered simplicial complex $f : K \to \Zbb^2$.
Given such bifiltered complex, a minimal presentation of any of its homology representations $H_i(f) \in \rep(\Zbb^2)$ has size $O(|K|\times |K|)$ and can be computed in $O(|K|^3)$ time using the Lesnick--Wright algorithm \cite{lesnick-wright}.
In the special case of zero-dimensional homology ($i=0$), the computation can be done in $O(|K| \log |K|)$ time \cite{morozov-scoccola}.
Moreover, computational shortcuts exist, which can significantly reduce computational time in practice \cite{kerber2021fast,alonso2023filtration,bauer_et_al:LIPIcs.SoCG.2023.15}.
This justifies taking a presentation as input in \cref{proposition:main-computation-result} (as is standard, the finite support assumption can be ensured using coning; see, e.g., \cite[Section~4.1]{bauer_et_al:LIPIcs.SoCG.2023.15}).
The output of \cref{proposition:main-computation-result} is again a presentation, which is justified using recent work \cite{dey_et_al:LIPIcs.SoCG.2025.41}, which provides an algorithm for decomposing representations of $\Zbb^n$ into indecomposables, given a presentation as input.
%Note that the algorithm in \cite{dey_et_al:LIPIcs.SoCG.2025.41} is cubic in the case of spread-decomposable representations, which, together with 
%\cref{proposition:main-computation-result}, and the fact that $\cokerxy{M}$ and $\kerxy{M}$ are spread-decomposable (\cref{corollary:birth-death-curves}), implies that end-curves, and thus also the two-parameter count, can be computed in cubic time.

\begin{lemma}
    \label{lemma:presentation-birth-curves}
    Let $M \in \rep(\Zbb^2)$, and let $Q \xrightarrow{\,\,f\,\,} P \to M \to 0$ be a projective presentation of $M$.
    Let $g : P[-1,-1] \to P$ be given by the action of $\xbf\ybf$.
    The morphism $f + g : Q \oplus P[-1,-1] \to P$ provides a presentation of $\cokerxy{M}$.
\end{lemma}
\begin{proof}
    The morphism $P \to \coker(f) = M$ is surjective, and thus $\cokerxy{M} = \coker(M[-1,-1] \to M) \cong \coker(P[-1,-1] \to M)$.
    We thus get $\coker(f + g) \cong \coker(P[-1,-1] \to M) \cong \cokerxy{M}$.
\end{proof}

%Recall that, given $p \in \Zbb^n$, we let ${\uparrow}p = \{q \in \Pscr : p \leq q\} \subseteq \Pscr$ denote the principal upset of~$p$.
%Let $d : \Zbb^n \to (\Zbb^n)^\op$ denote the poset isomorphism given by mapping $p$ to $-p$, and 
Let $D = \Hom(-,\kbb) : \vect \to \vect^\op$ denote the dualization functor, and let
$D : \rep(\Zbb^n) \to \rep(\Zbb^n)$ also denote the following composite functor:
\[
    \rep(\Zbb^n) = 
    \vect^{\Zbb^n}
        \xrightarrow{D_*}
    (\vect^\op)^{\Zbb^n}
        \cong
    \vect^{(\Zbb^n)^\op}
    = \rep((\Zbb^n)^\op).
    %    \xrightarrow{d^*}
    %\vect^{\Zbb^n} = \rep(\Zbb^n).
\]
%Consider the functor $F : (\Zbb^n)^\op \to \rep(\Zbb^n)$ given by the Yoneda embedding, that is, given by mapping $p$ to $\kbb_{{\uparrow}p}$.
%This induces a functor $\rep(\Zbb^n) \to \rep(\Zbb^n)$ mapping $M$ to $\Hom(M, F)$.
%Recall that the \emph{Nakayama functor} $\nu : \rep(\Zbb^n) \to \rep(\Zbb^n)$ is defined by $\nu(M) = D \Hom(M, F)$, with structure morphisms induced by the natural morphisms $\kbb_{{\uparrow}q'} \to \kbb_{{\uparrow}q}$ whenever $q \leq q'$.

%\begin{lemma}
%    \label{lemma:presentation-of-dual}
%    Let $M \in \rep(\Zbb^2)$, and let $0 \to R \xrightarrow{\,\,f\,\,} Q \to P \to M \to 0$ be a projective resolution of $M$.
%    The morphism $\nu f : \nu R \to \nu Q$ provides an injective copresentation of $M$, and hence $D(\nu f)$ provides a projective presentation of $DM$.
%\end{lemma}
%\begin{proof}
%    The second statement follows directly from the first one.
%    \todo{finish}
%\end{proof}

The proof of the following result uses \cite{bauer_et_al:LIPIcs.SoCG.2023.15}; we do not introduce the language of that paper since it is only used here.

\begin{lemma}
    \label{lemma:computation-presentation-dual}
    Let $M \in \repf(\Zbb^2)$, and let $Q \to P \to M \to 0$ be a finite projective presentation.
    A projective presentation of $DM \in \repf(\Zbb^2)$ of size $O(|P| \times |Q|)$ can be constructed in $O\left(|P|^3+|Q|^3\right)$ time.
\end{lemma}
\begin{proof}
    By \cite[Corollary~14~(1$\Leftrightarrow$3)]{bauer_et_al:LIPIcs.SoCG.2023.15} (see also \cite[Corollary~3.27]{miller-thesis}), in order to build a projective presentation of $DM$ it is sufficient to complete the presentation $Q \to P \to M \to 0$ to a resolution $0 \to R \to Q \to P \to M \to 0$, and then transpose the matrix representing the morphism $R \to Q$.
    Completing the presentation to a resolution can be done in cubic time using \cite[Section~3]{lesnick-wright}.
    By \cref{lemma:inequalities-betti-tables-2}, we have $|R| \leq |Q|$, so the presentation of $DM$ thus constructed has size $O(|P|\times|Q|)$.
    %, and the presentation of $DM$ can then be reduced to a minimal presentation using \cite[Section~4]{lesnick-wright}.
\end{proof}
%\begin{proof}
%    The result then follows from \cref{lemma:presentation-of-dual}.
%\end{proof}

\begin{proof}[Proof of \cref{proposition:main-computation-result}]
    The claim for $\cokerxy{M}$ follows directly from \cref{lemma:presentation-birth-curves}.
    For the case of $\kerxy{M}$, note that $D\kerxy{M} \cong \cokerxy{(D M[-1,-1])}$.
    A projective presentation of $DM$, and thus of $DM [-1,-1]$, can be computed in cubic time thanks to \cref{lemma:computation-presentation-dual}.
    A presentation of 
    $\cokerxy{(D M[-1,-1])} \cong D \kerxy{M}$ can then be computed in linear time using the first claim in this result.
    Since $DD \kerxy{M} \cong \kerxy{M}$, we can use \cref{lemma:computation-presentation-dual} again to obtain a presentation of $\kerxy{M}$ in cubic time, as required.
\end{proof}

%\section{Monotonicity and stability}
%A $\Zbb$-valued additive invariant $\alpha : \Acal \to \Zbb$ is \emph{inclusion-projection monotonic} if whenever there exists a monormophism $M \hookrightarrow N$ or an epimorphism $N \twoheadrightarrow M$ for $M,N \in \Acal$, we have $\alpha(M) \leq \alpha(N)$.
%The main reason why the count $\Ncaltwo$ is not as well-behaved as the bar-count $\Ncalone(M) = |\dec(M)|$ in one-parameter persistence is that it is not inclusion-projection monotonic.
%This is not a fluke of our approach though.
%
%\begin{proposition}
%    \label{proposition:no-go-monotonic-count}
%    There is no additive invariant $\Ncal : \fintwoparam \to \Zbb$ that is normalized on spreads and inclusion-projection monotonic.
%\end{proposition}
%
%In fact, the example used to prove \cref{proposition:no-go-monotonic-count} shows that there is no count that is normalized on spreads and smoothing monotonic, i.e., such that $\Ncal(\smooth_\delta M) \leq \Ncal(M)$ for all $M$ and all $\delta \geq 0$.
%\luis{introduce smoothing $\smooth$}
%
%However, monotonicity does hold for a restricted class of modules, which we now introduce.
%For this kind of modules, the counting function $\Ncaltwo$ can easily be extended to any indexing poset.
%
%A finitely generated module $M : \Pscr \to \vect$ is \emph{totally order generated} (TOG) if the support of the $0$th Betti table $\beta_0(M) : \Pscr \to \Zbb$ is a totally ordered subset of $\Pscr$.
%If $M$ is TOG, we let $\Ccal(M) \coloneqq |\beta_0(M)|$.
%
%
%\togmainresult*

\appendix

\section{Technical results about two-parameter spreads}

\begin{definition}
    If $I \subseteq \Zbb^2$ is a spread, we let $\min(I) \subseteq I$ denote the set of minimal elements of $I$.
    The \emph{$y$-ordering} on $\min(I)$ is the unique total ordering by increasing $y$-coordinate.
\end{definition}

The $y$-ordering is unique (i.e., no two elements of $\min(I)$ can have the same $y$-coordinate), since if $a,b \in \Zbb^2$ have the same $y$-coordinate, then they must be comparable, and thus they cannot both be minimal elements of $I$.
Note that the $y$-ordering on $\min(I)$ is, equivalently, the unique ordering by decreasing $x$-coordinate.

If $I$ is a finite spread, we denote the $y$-ordering of $\min(I)$ by $\{z_i\}_{i \in S}$ where $S = \{1, \dots, n\}$.

\begin{lemma}
    \label{lemma:join-of-consecutive-minima}
    Let $I \subseteq \Zbb^2$ be a finite spread, and let $\min(I) = \{z_i\}_{i \in S}$ be the $y$-ordering on its set of minimal elements.
    For every $i \in S$ such that $i+1 \in S$, the join $z_i \vee z_{i+1}$ is in $I$.
\end{lemma}
\begin{proof}
    %So given $1 \leq i \leq k-1$, we must show that $z_i \vee z_{i+1} \in I$.
    %Since $z_i$ is minimal, $z_i - (1,0) \notin I$.
    %Now, let $\ell > 0 \in \Nbb$ be the minimum such that $z' \coloneqq z_i + (0,\ell) - (1,0) \in I$, which must exist since $z_{i+1}$ has strictly larger $y$-coordinate, and strictly smaller $x$-coordinate, than $z_i$.
    %We claim that $z' 
    Choose a path from $z_i$ to $z_{i+1}$ contained in $I$ of minimum length:
    \[
        z_i \leq a_1 \geq a_2 \leq a_3 \geq a_4 \leq \cdots \geq a_{\ell-1} \leq a_\ell \geq z_{i+1}.
    \]
    Since every element in this path is in $I$, we can replace the $a_j$ with $j$ even by some $z_{i_j}$, to get a path
    \[
        z_i \leq a_1 \geq z_{i_2} \leq a_3 \geq z_{i_4} \leq \cdots \geq z_{i_{\ell-1}} \leq a_\ell \geq z_{i+1}
    \]
    contained in $I$.
    Next, since $I$ is poset-convex, we can replace the $a_j$ with $j$ odd by a join, as follows
    \[
        z_i \leq (z_i \vee z_{i_1}) \geq z_{i_2} \leq (z_{i_2} \vee z_{i_4}) \geq z_{i_4} \leq \cdots \geq z_{i_{\ell-1}} \leq (z_{i_{\ell-1}} \vee z_{i+1}) \geq z_{i+1},
    \]
    to again get a path contained in $I$.

    For every even $j$, it must be the case that the $y$-coordinate of $z_{i_{j+2}}$ is strictly larger than that of~$z_{i_{j}}$.
    This is because the $y$-coordinate of $z_{i+1}$ is larger than that of $z_i$, so if the path either decreases the $y$-coordinate at some point, or it keeps it the same, one could get a path of smaller length by skipping this.
    Finally, the only possible path of that form that strictly increases the $y$-coordinate is
    $z_i \leq (z_i \vee z_{i+1}) \geq z_{i+1}$, and thus $z_i \vee z_{i+1} \in I$.
    %For every even $\ell$, it must be the case that $i+1 \neq z_{i_\ell} \neq i$, since otherwise the path would not be of minimum length.
\end{proof}

\begin{definition}
    \label{definition:cokerxy-spread}
    Let $I \subseteq \Zbb^2$ be a spread.
    Define
    \begin{align*}
        %\cokerxyspread
        \cokerxyspread{I} = \left\{a \in I : a-(1,1) \notin I \right\}
        \;\; \text{and} \;\;\;
        %\kerxyspread
        \kerxyspread{I} = \left\{a \notin I : a-(1,1) \in I\right\}.
    \end{align*}
\end{definition}

\begin{lemma}
    \label{lemma:boundary-of-spread-is-spread}
    If $I \subseteq \Zbb^2$ is a finite spread, then $\cokerxyspread{I}$ and $\kerxyspread{I}$ are poset-connected.
\end{lemma}
\begin{proof}
    By duality (\cref{section:duality}), it is sufficient to prove that $\cokerxyspread{I}$ is poset-connected.

    Let $\min(I) = \{z_i\}_{i \in S}$ be the $y$-ordering on the set of minimal elements of $I$.
    Clearly, we have $\min(I) \subseteq \cokerxyspread{I}$.
    Now, for every $p \in \cokerxyspread{I}$, there exists a path from $p$ to some $z_i \in \min(I)$ in $\cokerxyspread{I}$: indeed, since $p \in I$, there exists $i \in S$ such that $z_i \leq p$.
    Thus, we need to prove that, for every $i < j \in S$ there exists a path from $z_i$ to $z_j$ in $\cokerxyspread{I}$.
    We claim that the following is a path from $z_i$ to $z_j$ in $\cokerxyspread{I}$:
    \[
        z_i \leq (z_i \vee z_{i+1}) \geq z_{i+1} \leq (z_{i+1} \vee z_{i+2}) \geq z_{i+2} \leq \dots \geq z_{j-1} \leq (z_{j-1} \vee z_j) \geq z_j.
    \]

    To prove the above, it suffices to show that, for every $k \in S$ such that $k+1 \in S$, we have that $z_k \vee z_{k+1}$ is in $\cokerxyspread{I}$.
    By \cref{lemma:join-of-consecutive-minima}, we know that $z_k \vee z_{k+1} \in I$, so we need to prove that $(z_k \vee z_{k+1}) - (1,1) \notin I$.
    If $(z_k \vee z_{k+1}) - (1,1)$ were in $I$, then there would exist $z_\ell \leq (z_k \vee z_{k+1}) - (1,1)$.
    But then either $\ell < k$ or $\ell > k+1$.
    In the first case, the $x$-coordinate of $z_\ell$ would be strictly larger than that of $z_k$, which contradicts $z_\ell \leq (z_k \vee z_{k+1}) - (1,1)$, since the $x$-coordinate of $z_k \vee z_{k+1}$ is that of $z_k$.
    In the second case, the $y$-coordinate of $z_\ell$ would be strictly larger than that of $z_{k+1}$, which also contradicts $z_\ell \leq (z_k \vee z_{k+1}) - (1,1)$, since the $y$-coordinate of $z_k \vee z_{k+1}$ is that of $z_{k+1}$.
    So $(z_k \vee z_{k+1}) - (1,1) \notin I$, concluding the proof.
\end{proof}

\begin{lemma}
    \label{lemma:birth-death-of-spread}
    If $I \subseteq \Zbb^2$ is a finite spread, then
    $\births(\kbb_I) = \{\cokerxyspread{I}\}$ and
    $\deaths(\kbb_I) = \{\kerxyspread{I}\}$.
    Equivalently, we have
    $\cokerxy{\kbb_I} = \kbb_{\cokerxyspread{I}}$ and $\kerxy{\kbb_I} = \kbb_{\kerxyspread{I}}$.
\end{lemma}
\begin{proof}
    By duality (\cref{section:duality}), it is sufficient to prove $\deaths(\kbb_I) = \{\kerxyspread{I}\}$.
    The support of $\kerxy{\kbb_I} = \ker\left(\kbb_I[-1,-1] \xrightarrow{\xbf\ybf} \kbb_I\right) \subseteq \kbb_I[-1,-1]$
    consists of those $a \notin I$ such that $a-(1,1) \in I$, which is equal to $\kerxyspread{I}$, by definition.
    Since any subrepresentation of a spread representation is spread-decomposable (\cref{lemma:submodule-of-spread-is-spread}), we have that $\births(\kbb_I)$ is equal to $\{\kbb_C\}$ with $C$ ranging over the poset-connected components of $\kerxyspread{I}$.
    To conclude the proof, we use \cref{lemma:boundary-of-spread-is-spread}, which implies that  $\kerxyspread{I}$ is poset-connected.
\end{proof}

\begin{lemma}
    \label{lemma:unique-minimum-below}
    Let $I \subseteq \Zbb^2$ be a finite spread curve, and let $\min(I) = \{z_i\}_{i \in S}$ be the $y$-ordering on the set of minimal elements of $I$.
    Suppose that $p \in I$ and $p \neq z_i \vee z_{i+1}$ for all $i \in S$ such that $i+1 \in S$.
    Then, there exists a unique $i \in S$ such that $z_i \leq p$.
\end{lemma}
\begin{proof}
    Since $p \in I$, there must exist at least one $i \in S$ such that $z_i \leq p$.
    Assume that there also exists $j \in S$ such that $z_j \leq p$, with $i \neq j$, and assume, without loss of generality, that $i < j$.
    Then, we must also have $z_{i+1} \leq p$, since the $x$-coordinate of $z_{i+1}$ is at most that of $z_i$, which is at most that of $p$; and the $y$-coordinate of $z_{i+1}$ is at most that of $z_j$, which is at most that of $p$.
    This implies that $(z_i \vee z_{i+1}) \leq p$.
    All elements between $z_i \vee z_{i+1}$ and $p$ must be in $I$, since $I$ poset-convex.
    Towards a contradiction, assume that $p \neq (z_i \vee z_{i+1})$.
    Then either $(z_i \vee z_{i+1}) + (1,0) \in I$ or
    $(z_i \vee z_{i+1}) + (0,1) \in I$.
    Let us assume $a = (z_i \vee z_{i+1}) + (1,0) \in I$ (the other case is analogous).
    Then, we have that $z_i \leq a - (1,1) \leq a$, so both $a$ and $a - (1,1)$ are in $I$, which violates the fact that $I$ is a spread curve.
    This means that $p = (z_i \vee z_{i+1})$, concluding the proof.
\end{proof}

\begin{lemma}
    \label{lemma:nested-union-contractible}
    Let $K = \bigcup_{n > 1 \in \Nbb} K_n$ be a nested, countable union of contractible simplicial complexes.
    Then $K$ is contractible.
\end{lemma}
\begin{proof}
    It is standard that $\pi_k(K) \cong \colim_{n} \pi_k(K_n)$ for all $k \in \Nbb$ (since spheres and the unit interval are compact), and the right-hand side is trivial by assumption.
    Then $K$ is contractible by Whitehead theorem.
\end{proof}

The next lemma has a somewhat tedious proof, although the statement is easy to believe.

\begin{lemma}
    \label{lemma:spread-contractible}
    If $I \subseteq \Zbb^2$ is a spread, then $\Delta(I)$ is contractible.
\end{lemma}
\begin{proof}
    For $n > 1 \in \Nbb$, let $I_n \coloneqq I \cap [-(n,n),(n,n)]$, which is easily seen to be either empty or a finite spread.
    Every chain in $I$ must lie in $I_n$ for some $n$, so $\Delta(I) = \bigcup_{n>1 \in \Nbb} \Delta(I_n)$ is a nested, countable union of simplicial complexes.
    By \cref{lemma:nested-union-contractible}, it is then sufficient to prove that $\Delta(I_n)$ is contractible for sufficiently large $n$.
    Thus, without loss of generality, we assume that $I$ is a finite spread, and we prove that $\Delta(I)$ is contractible.

    %Let $Z \subseteq I$ be the set of minimal elements of $I$.
    We proceed by induction on the cardinality of $\min(I)$.
    If $|\min(I)| = 1$, then $I$ has minimum, and thus $\Delta(I)$ is contractible by Claim 1 in the proof of \cref{theorem:count-is-euler}.
    Otherwise, let $Z \coloneqq \min(I) = \{z_1, \dots, z_k\}$ be the $y$-ordering of the minimal elements of $I$.

    Let $Z' \coloneqq Z \setminus \{z_1\}$.
    Let $J = \{a \in I : z_1 \leq a\}$ and $J' = \{a \in I : \exists z' \in Z', z' \leq a\}$, so that $I = J \cup J'$.
    Every chain in $I$ is contained in either $J$ or $J'$ (or both), so $\Delta(J) \cup \Delta(J') = \Delta(I)$.
    Let us assume for the moment that $J$, $J'$, and $J \cap J'$ are spreads with fewer than $k$ minimal elements.
    If this is the case, then $\Delta(J)$, $\Delta(J')$, and $\Delta(J \cap J')$ are contractible by the inductive hypothesis, and, by the nerve lemma for simplicial complexes \cite[Theorem~10.6]{bjorner}, we get that $\Delta(I)$ is homotopy equivalent to a simplicial interval, which is contractible.
    So we need to prove that $J$, $J'$, and $J \cap J'$ are spreads with fewer than $k$ minimal elements.
    It is easy to see that they are poset-convex, so we first prove that they are poset-connected.

    Every element of $J$ can be connected to $z_1$ via a path that is monotonically decreasing and thus included in $J$, so $J$ is poset-connected.

    We now prove that $J'$ is poset-connected.
    If $a,b \in J'$, then $a \geq z_i$ and $b \geq z_j$, for some $i,j \in \{2, \dots, k\}$.
    The element $a$ (resp.~$b$) can be connected to $z_i$ (resp.~$z_j$) via a path that is monotonically decreasing, and thus included in $J'$.
    So we must show that $z_i$ and $z_j$ can be connected through a path included in $J'$.
    Without loss of generality, we may assume that $i \leq j$, and it is thus sufficient to show that, for every $2 \leq \ell < k$, the element $z_\ell$ can be connected to $z_{\ell + 1}$ through a path included in $J'$.
    We use the path $z_\ell \leq (z_{\ell} \vee z_{\ell+1}) \geq z_{\ell+1}$; this works since $z_{\ell} \vee z_{\ell+1} \in I$, by \cref{lemma:join-of-consecutive-minima}.

    We now prove that $J \cap J'$ is poset-connected, and, by an argument analogous to the one for the case of $J$, it is sufficient to prove that $J \cap J' = \{a \in I : z_1 \vee z_2 \leq a\}$.
    The inclusion $(\supseteq)$ is clear, so let us prove the other inclusion.
    If $a \in J \cap J'$, then $a \in J$ so that $a \geq z_1$.
    So we must prove that, if $a \in J \cap J'$, then $a \geq z_2$.
    The $x$-coordinate of $z_2$ must be strictly smaller than that of $z_1$, since they are incomparable.
    So, to prove that $a \geq z_2$, it is sufficient to prove that the $y$-component of $a$ is larger than that of $z_2$.
    This is clear, since if $a \in J'$, then $a \geq z_j$ for some $2 \leq j \leq k$, and the $y$-coordinates of the elements $z_2, \dots, z_k$ are monotonically increasing.

    Finally, note that $\min(J) = \{z_1\}$, $\min(J') = \{z_2, \dots, z_k\}$, and $\min(J \cap J') = \{z_1 \vee z_2\}$, by the last paragraph, so the spreads have fewer than $k$ minimal elements, as required.
\end{proof}

\section{Proofs of known results}
\label{section:proofs-known-results}

\subsection{Additive invariants and bases}
\label{section:appendix-additive-invariants-bases}

\begin{lemma}
    \label{lemma:additivity-finer-invariant}
    Let $\alpha \succcurlyeq \alpha'$ be additive invariants on an additive category $\Acal$.
    Let $\{M_i \in \Acal\}_{i \in I}$ and $\{c_i \in \Zbb\}_{i \in I}$ be such that $c_i \neq 0$ for finitely many $i \in I$, and $\sum_{i \in I} c_i \cdot \alpha(M_i) = 0$.
    Then, $\sum_{i \in I} c_i \cdot \alpha'(M_i) = 0$.
\end{lemma}
\begin{proof}
    Let $I = A \sqcup B \sqcup C$ be such that $c_i = 0$ if $i \in A$, $c_i > 0$ if $i \in B$, and $c_i < 0$ if $i \in C$.
    Then, by additivity of $\alpha$, we have
    \[
        \alpha\left( \bigoplus_{i \in B} M_i^{c_i} \right)
        =
        \alpha\left( \bigoplus_{i \in C} M_i^{-c_i} \right),
    \]
    which implies, by the fact that $\alpha \succcurlyeq \alpha'$, that
    \[
        \alpha'\left( \bigoplus_{i \in B} M_i^{c_i} \right)
        =
        \alpha'\left( \bigoplus_{i \in C} M_i^{-c_i} \right).
    \]
    By additivity of $\alpha'$, it follows that $\sum_{i \in I} c_i \cdot \alpha'(M_i) = 0$, as required.
\end{proof}

\begin{lemma}
    \label{lemma:any-decomposition-gives-count}
    Let $\alpha$ be an additive invariant on an additive category $\Acal$, let $\Bcal$ be a basis for $\alpha$, and let $f : \Bcal \to \Zbb$.
    If $M \in \Acal$, and $A_1, \dots, A_k, B_1, \dots, B_\ell \in \Bcal$ are such that
    \[
        \alpha(M) = \sum_{i = 1}^k \alpha(A_i) - \sum_{j = 1}^\ell \alpha(B_j),
    \]
    then $\Ncal^{(\alpha,\Bcal,f)}(M) = \sum_{i = 1}^k f(A_i) - \sum_{j = 1}^\ell f(B_j)$, and in particular $\Ncal^{(\alpha,\Bcal)}(M) = k - \ell$.
    \qed
\end{lemma}
\begin{proof}
    This is an immediate consequence of additivity and \cref{lemma:basic-results-counts-basis}(2).
\end{proof}

\begin{proof}[Proof of \cref{proposition:I-is-basis-of-rk-I}]
    Consider the partial order on $\Ical$ given by inclusion.
    If $I \in \Ical$, then the support of $\Rk_\Ical(\kbb_I) \in \Zbb^{\Ical}$ is the principal upset of $\Ical$ corresponding to $I$, where it is constantly $1$.
    By M\"obius inversion \cite[Chapter~3.7]{stanley}, the set $\{\Rk_{\Ical}(\kbb_{I})\}_{I \in \Ical}$ forms a basis of $\Zbb^{\Ical}$, and the result follows.
    Alternatively, one could use \cref{proposition:projectives-are-basis}(3), but we omit the details.
\end{proof}

\begin{definition}
    \label{definition:relative-projective-resolution}
    Let $\Lambda$ be a finite dimensional $\kbb$-algebra and let $\Xcal$ be a set of pairwise non-isomorphic finite dimensional, indecomposable $\Lambda$-modules.
    A \emph{$\Xcal$-projective resolution} of a $\Lambda$-module $M$ is an $\Xcal$-exact sequence
    $\cdots \to C_1 \to C_0 \to M \to 0$ such that $C_i$ is $\Xcal$-decomposable for every $i \geq 0$.
\end{definition}

In the following proposition, all modules are assumed to be finite dimensional.

\begin{proposition}
    \label{proposition:projectives-are-basis-2}
    Let $\Lambda$ be a finite dimensional $\kbb$-algebra.
    \begin{enumerate}
        \item The set $\indsimpl_\Lambda$ is a basis for $\dimhom_{\indproj_\Lambda}$ and $\Ksf_{\indproj_\Lambda}$.
        \item The additive invariants $\dimhom_{\indproj_\Lambda}$ and $\Ksf_{\indproj_\Lambda}$ are equivalent.
        \item If $\Lambda$ has finite global dimension, then $\indproj_\Lambda$ is a basis for $\dimhom_{\indproj_\Lambda}$.
        \item If $\Xcal$ is a set of pairwise non-isomorphic indecomposable $\Lambda$-modules, then $\dimhom_{\Xcal} \approx \Ksf_{\Xcal}$.
        \item Let $\Xcal$ be a set of pairwise non-isomorphic indecomposable $\Lambda$-modules.
              If every $\Lambda$-module admits a finite $\Xcal$-projective resolution,
              then $\Xcal$ is a basis for $\dimhom_{\Xcal} \approx \Ksf_\Xcal$.
    \end{enumerate}
    % containing one representative of each isomorphism class of indecomposable projective $\Lambda$-module.
\end{proposition}
\begin{proof}
    %    Recall that $\{X/rX\}_{X \in \Xcal}$ contains a representative of all isomorphism classes of simple modules, where $r$ denotes radical, and that no two of these modules are isomorphic \luis{cite}.
    Let $\indproj_\Lambda = \{P_i\}_{i \in I}$ and $\indsimpl_\Lambda = \{S_i\}_{i \in I}$ with $S_i = P_i/\mathrm{rad}(P_i)$.

    \medskip

    \noindent
    \emph{(1)}.
    The set $\{\dimhom_{\indproj_\Lambda}(S_i)\}_{i \in I}$ is linearly independent since $\dim \hom(P_i, S_j)\neq 0$ if and only if $i=j$.
    Let $M \in \mod_\Lambda$ and let $0 = J_0 \subset \cdots \subset J_k = M$ be a composition series of $M$.
    Then $\dimhom_{\indproj_\Lambda}(M) = \sum_{i = 1}^{k-1}\, \dimhom_{\indproj_\Lambda}(J_{i+1}/J_i)$ exhibits the dimension vector of $M$ as a sum of dimension vectors of simple modules, so
    $\{\dimhom_{\indproj_\Lambda}(S_i)\}_{i \in I}$ generates $\dimhom_{\indproj_\Lambda}$.

    Simples form a basis for $\Ksf_{\indproj}$, by \cite[Theorem~1.7]{auslander-reiten-smalo}.
    %is the usual Grothendieck group, that is, the free Abelian group on isomorphism of $\Lambda$-modules modulo relations induced by short exact sequences.
    %Thus, $\{[S_i]\}_{i \in I}$ generates $\Ksf_{\indproj_\Lambda}$ by the existence of decomposition series, and the set is linearly independent because the simple factors in the decomposition series are unique up to permutation and isomorphism.
    %So $\indsimpl_\Lambda$ is a basis for $\Ksf_{\indproj}$.

    \medskip

    \noindent
    \emph{(2)}.
    If $0 \to A \to B \to C \to 0$ is a short exact sequence of $\Lambda$-modules, then $\dimhom_{\indproj_\Lambda}(B)
        = \dimhom_{\indproj_\Lambda}(A)
        + \dimhom_{\indproj_\Lambda}(C)$, so the invariant $\dimhom_{\indproj_\Lambda}$ factors through $\Ksf_{\indproj_\Lambda}$, so that $\Ksf_{\indproj_\Lambda} \succcurlyeq \dimhom_{\indproj_\Lambda}$.
    Since the two invariants admit the same basis, by (1), they are equivalent.

    \medskip

    \noindent
    \emph{(3)}.
    Any two bases of $\dimhom_{\indproj_\Lambda}$ have the same number of elements.
    Since $\indproj_\Lambda$ has the same number of elements as $\indsimpl_\Lambda$, it is sufficient, by (1), to show that $\indproj_\Lambda$ generates $\dimhom_{\indproj_\Lambda}$.
    If $M \in \mod_\Lambda$, let $0 \to C_k \to \cdots \to C_0 \to M$ be a finite $\Lambda$-projective resolution of $M$.
    Then, by the exactness of $\hom(P_i,-)$ for every $1 \leq i \leq k$, we have $\dimhom_{\indproj_\Lambda}(M) = \sum_{i = 1}^k (-1)^i\, \dimhom_{\indproj_\Lambda}(C_i)$, so $\indproj_\Lambda$ generates $\dimhom_{\indproj_\Lambda}$.

    \bigskip

    In (4) and (5) we use the following construction, known as \emph{projectivization}, which allows us to reduce the case of arbitrary $\Xcal$ to the case where $\Xcal$ consists of indecomposable projectives.
    If $P \coloneqq \bigoplus_{X \in \Xcal} X$ and $E \coloneqq \End(P)$, the functor $(-)^P : \mod_{\Lambda} \to \mod_E$ restricts to an equivalence of categories $\add(\Xcal) \to \proj(E)$ \cite[Chapter~II,~Section~2]{auslander-reiten-smalo}; let $G : \proj(E) \to \add(\Xcal)$ be its inverse.
    This gives us a bijection $f : \Xcal \to \indproj_E$ such that $\dimhom_{\indproj_E}(M^P) \circ f = \dimhom_{\Xcal}(M) : \Xcal \to \Zbb$, for every $M \in \mod_\Lambda$.

    \medskip
    \noindent
    \emph{(4)}.
    This follows from projectivization and (2).

    \medskip
    \noindent
    \emph{(5)}.
    By projectivization, it is enough to prove that $\indproj_E$ is a basis for $\dimhom_{E} : \mod_E \to \Zbb^{\indproj_E}$.
    By (3), it is enough to show that $E$ has finite global dimension, which is equivalent to every $E$-module $M \in \mod_E$ admitting a finite $E$-projective resolution (this is because every $E$-module is an iterated extension by simple modules, of which there are finitely many isomorphism types, so the global dimension of $E$ is attained as the projective dimension of a simple \cite[Corollary~11]{auslander}).
    To build a finite projective resolution of any $M \in \mod_\Lambda$ we proceed as follows.
    Let $A \to B \to M$ be a finite $E$-projective presentation of $M$.
    Let $0 \to C_k \to \cdots \to C_0 \to \ker(G(A) \to G(B))$ be a finite $\Xcal$-projective resolution of $\ker(G(A) \to G(B)) \in \mod_\Lambda$.
    We have
    \[
        \ker\left(G(A) \to G(B)\right)^P
        \cong
        \ker\left(G(A)^P \to G(B)^P\right)
        \cong
        \ker(A \to B),
    \]
    by left-exactness of $(-)^P$.
    This allows us to splice the exact sequence of $E$-projectives
    $0 \to C_k^P \to \cdots \to C_0^P$ with $A \to B \to M$ to get a finite $E$-projective resolution $0 \to C_k^P \to \cdots \to C_0^P \to A \to B \to M$, as required.
    %    Moreover, if $S$ is a simple module and $X \in \Xcal$, then $\hom(X,S) \neq 0$
    %    Clearly, a choice of isomorphism class of each simple $\Lambda$-module gives a basis for $\dimhom_\Xcal$, since, if $S$ is a simple $\hom(X,S).$
    %    Since there are as many (isomorphism classes of) simple modules as there are (isomorphism classes of) indecomposable projective modules, it is enough to show $\Xcal$ also generates $\vec{\dim}$.
    %    If $0 \to C_k \to \cdots \to C_0 \to M$ is a projective resolution, then $\vec{\dim}(M) = \sum_{1 \leq i \leq k} (-1)^i \vec{\dim}(C_i) = \dimhom_{\Xcal}(M)$, so $\{\dimhom_{\Xcal}(X)\}_{X \in \Xcal}$ generates, as required.
    %    %Equivalently, we need to prove that if $A,B \in \add \Xcal$ are such that $\dimhom_{\Xcal}(A) = \dimhom_{\Xcal}(B)$, then $A \cong B$; and .
    %    %By \cite[Theorem~1.1]{blanchette-brustle-hanson}, the invariant $\dimhom_{\Xcal}$ is equivalent to the invariant $\mod_\Lambda \to K(\Lambda, \Xcal)$ given taking equivalence class in the Grothendieck group relative to the exact structure associated with $\Xcal$, and $\Xcal$ is a basis for this second invariant.
\end{proof}

\begin{proof}[Proof of \cref{remark:basis-for-known-invariants}]
    The first and second statements follow from \cref{proposition:I-is-basis-of-rk-I}.
    The third statement follows from \cref{proposition:projectives-are-basis} and the fact that the \emph{rank exact structure} (i.e., the exact structure whose indecomposable projectives are the hook representations) on the category of representations of a finite poset has finite global dimension \cite[Theorem~1.2]{blanchette-brustle-hanson}.
    The fourth statement follows from \cref{proposition:projectives-are-basis} and the fact that the \emph{spread exact structure} (i.e., the exact structure whose indecomposable projectives are the spread representations) on the category of representations of a finite poset has finite global dimension \cite[Proposition~4.5]{asashiba-escolar-nakashima-yoshiwaki-2}.
\end{proof}

\subsection{Spread representations}

\begin{lemma}[{cf. \cite[Lemma~4.4]{asashiba-escolar-nakashima-yoshiwaki-2}}]
    \label{lemma:submodule-of-spread-is-spread}
    Let $\Pscr$ be a poset, let $I \subseteq \Pscr$ be a spread, and let $M \subseteq \kbb_I$ be a subrepresentation.
    Then $M$ is spread-decomposable.
\end{lemma}
\begin{proof}
    Let $i \in \Pscr$.
    Note that, if $M_i \neq 0$, then $M_i = \kbb$.
    So it is sufficient to show that the support of $M$ is a disjoint union of spreads.
    We prove the stronger fact that the support of $M$ is poset-convex.
    Let $i \leq j \leq k \in \Pscr$, with both $i$ and $k$ in the support of $M$.
    The morphism $\phi^{M}_{i,k} : M_i \to M_k$ is the identity $\kbb \to \kbb$, so, since $\phi^{M}_{i,k} = \phi^M_{j,k} \circ \phi^M_{i,j}$, we have $M_j \neq 0$, and thus $j$ is in the support of $M$, as required.
\end{proof}

\begin{theorem}[{\cite[Theorem~24]{asashiba-buchet-escolar-nakashima-yoshiwaki}}]
    \label{theorem:indecomposable-think-is-spread}
    Let $M \in \repf(\Zbb^2)$.
    If $M$ is indecomposable and $\dim M_i \in \{0,1\}$ for every $i \in \Zbb^2$, then $M$ is a spread representation.
\end{theorem}
\begin{proof}
    By translating $M$ if necessary, we may assume that its support is contained in some finite grid $\Gcal^2$, and by the fully faithful embedding $\rep(\Gcal^2) \hookrightarrow \repf(\Zbb^2)$ induced by the inclusion $\Gcal^2 \hookrightarrow \Zbb^2$ and zero-padding, we may assume that $M \in \rep(\Gcal^2)$.
    The result then follows from \cite[Theorem~24]{asashiba-buchet-escolar-nakashima-yoshiwaki}, which states that, over a finite, two-dimensional grid, an indecomposable of pointwise dimension at most one is necessarily a spread.
\end{proof}

\subsection{Betti tables}

\begin{lemma}
    \label{lemma:inequalities-betti-tables-1}
    Let $\Pscr$ be a poset and let $M, N \in \rep(\Pscr)$ be finitely generated.
    If there exists a surjection $N \twoheadrightarrow M$, then $|\beta_0^M| \leq |\beta_0^N|$.
\end{lemma}
\begin{proof}
    Let $P \to N$ be a projective cover, so that $P$ is projective with $|\beta_0^N|$-many summands.
    Then $P \to N \to M$ is surjective, so that $|\beta_0^M| \leq |\beta_0^N|$.
\end{proof}

\begin{lemma}
    \label{lemma:inequalities-betti-tables-2}
    If $M,N \in \rep(\Zbb^2)$ are finitely generated, projective representations, and $\psi : M \hookrightarrow N$ is a morphism, then $\ker(\psi)$ is finitely generated and projective, and $|\beta_0^{\ker(\psi)}| \leq |\beta_0^M|$.
\end{lemma}
\begin{proof}
    The global dimension of the category of finitely generated, bigraded $\kbb[\xbf,\ybf]$-modules is two, due to Hilbert's syzygy theorem.
    This implies that $\ker(\psi)$ is finitely generated and projective.
    Moreover, the representation $\ker(\psi)$ must have at most as many indecomposable summands as $M$, since we have a monomorphism $\ker(\psi) \hookrightarrow M$ and the number of indecomposable summands of any of these representations can be computed by evaluating at $(x,y) \in \Zbb^2$ with $x,y \gg 0$, since $\ker(\psi)$ and $M$ are finitely generated and projective.
\end{proof}

\bibliographystyle{plain}
\bibliography{references}

\end{document}